\documentclass[10pt]{siamltex}
\usepackage{amsmath,amsfonts,amscd,amssymb,bm,cite}
\usepackage{mathrsfs,listings,url}
\usepackage{graphics,color,ulem}
\usepackage{float}
\usepackage[titletoc,page]{appendix}
\usepackage{subfigure}
\usepackage[colorlinks=true, bookmarksopen,
pdfauthor={CST},
pdfcreator={pdftex},
pdfsubject={algorithms},
linkcolor={blue},
anchorcolor={black},
citecolor={firebrick},
filecolor={magenta},
menucolor={black},
pagecolor={red},
plainpages=false,pdfpagelabels,
urlcolor={db} ]{hyperref}
\usepackage[capitalise]{cleveref}
\usepackage{comment}
\usepackage{verbatim}
\usepackage{marginnote}
\usepackage{float}
\usepackage[makeroom]{cancel}
\usepackage[pdftex]{graphicx}
\usepackage[section]{placeins}
\usepackage{mathptmx}
\usepackage{cases}
\usepackage{subeqnarray}
\usepackage{hhline}
\usepackage[linesnumbered,ruled,vlined]{algorithm2e}
\usepackage{xcolor}
\usepackage{tikz}
\usepackage{caption} 
\usepackage{lmodern}

\usepackage{comment}
\usepackage{comment}
\excludecomment{confidential}
\usepackage{mathtools}
\usepackage[inline]{enumitem}

\counterwithin{table}{section}
\counterwithin{figure}{section}
\graphicspath{{figures/}}

\newcommand{\vertnorm}[1]{{\big|\kern-0.25ex\big|\kern-0.25ex\big| #1 
    \big|\kern-0.25ex\big|\kern-0.25ex\big|}}

\usetikzlibrary{arrows}
\usetikzlibrary{positioning,shapes,snakes,calc,decorations,decorations.markings}

\restylefloat{table}
\allowdisplaybreaks
\definecolor{red}{rgb}{0.0, 0.0, 0.0}
\definecolor{db}{rgb}{0.0470,0,0.5294}
\definecolor{dg}{rgb}{0.0,0.392,0.0}
\definecolor{firebrick}{rgb}{0.698,0.133,0.133}
\definecolor{bl}{rgb}{0.0,0.0,0.0}
\definecolor{linen}{rgb}{0.980,0.941,0.902}
\definecolor{ivory}{rgb}{1.0,1.0,0.941}
\definecolor{aliceblue}{rgb}{0.941,0.973,1.0}
\definecolor{beige}{rgb}{0.961,0.961,0.863}
\definecolor{tan}{rgb}{0.824,0.706,0.549}
\definecolor{lightsteelblue}{rgb}{0.690,0.769,0.871}
\definecolor{paleturquoise}{rgb}{0.686,0.933,0.933}
\definecolor{lightblue}{rgb}{0.678,0.847,0.902}
\definecolor{skyblue}{rgb}{0.529,0.808,0.922}
\definecolor{palegoldenrod}{rgb}{0.933,0.910,0.667}
\definecolor{lightgoldenrod}{rgb}{0.933,0.867,0.510}
\definecolor{lightyellow}{rgb}{1.0,1.0,0.878}
\definecolor{yellow}{rgb}{1.0,1.0,0.0}
\definecolor{lightyellow1}{rgb}{1.0,1.0,0.878}
\definecolor{lemonchiffon}{rgb}{1.0,0.980,0.804}
\definecolor{myyellow}{rgb}{1,1,.9}
\definecolor{darkgreen}{rgb}{0.0,0.392,0.0}
\definecolor{darkviolet}{rgb}{0.580,0.0,0.827}
\definecolor{lightsalmon}{rgb}{1.0,0.627,0.478}
\definecolor{orange}{rgb}{1.0,0.647,0.0}
\definecolor{darkblue}{rgb}{0.00,0.00,0.55}
\numberwithin{equation}{section}

\crefname{table}{table}{tables}
\Crefname{table}{Table}{Tables}

\crefname{figure}{figure}{figures}
\Crefname{figure}{Figure}{Figures}

\begin{document}
	
\title{Partitioned Conservative, Variable Step, Second-Order Method for Magneto-hydrodynamics In Els\"asser Variables} 
\author{Zhen Yao\thanks{Department of Mathematics, University of Pittsburgh, Pittsburgh, PA 15260, USA. Email: \href{mailto:zhy@pitt.edu}{zhy@pitt.edu}.}
\and
Catalin Trenchea\thanks{Department of Mathematics, University of Pittsburgh, Pittsburgh, PA 15260, USA. Email: \href{mailto:trenchea@pitt.edu}{trenchea@pitt.edu}.}
\and
Wenlong Pei\thanks{Department of Mathematics, The Ohio State University, Columbus, OH 43210,USA. Email: \href{mailto:pei.176@osu.edu}{pei.176@osu.edu}.} 
}
\date{\emty}
\maketitle
	
\begin{abstract}
 Magnetohydrodynamics (MHD) describes the interaction between electrically conducting fluids and electromagnetic fields. 
		We propose and analyze a symplectic
		, second-order 
		algorithm for the evolutionary MHD system in Els\"asser variables.
        We reduce the computational cost of the iterative non-linear solver, at each time step, 
        by partitioning the coupled system into two subproblems of half size, solved in parallel. 
	We prove that the 
	iterations converge linearly
	, under a 
	time step restriction similar to the one required in the full space-time error analysis. 
The variable step algorithm 
unconditionally conserves the energy, cross-helicity and magnetic helicity, and numerical solutions are second-order accurate 
in the $L^{2}$ and $H^{1}$-norms.
        The time adaptive mechanism, based on a local truncation error criterion, helps the variable step algorithm balance accuracy and time efficiency. 
        Several numerical tests 
        support the theoretical findings and verify the advantage of time adaptivity.		
\end{abstract}

	
	
\section{Introduction}
    The equations of magnetohydrodynamics \cite{Alfv42} (MHD) 
    describe the motion of electrically conducting, incompressible flows in the presence of magnetic fields. 
    The understanding of the MHD system is essential to numerous applications in science and engineering, 
    astrophysics \cite{MR2328834,MR2549136,MR2729936}, 
    geophysics \cite{MR2360172,SHA15_BPAS}, sea water propulsion \cite{lin1990sea}, cooling system design \cite{Barleon1991401,Hidetoshi20061431} and process metallurgy \cite{MR1825486}. 
    Given a bounded domain $\Omega \subset \mathbb{R}^{d}$ ($d = 2,3$) and time interval $[0,T]$, the fluid velocity field $u$, magnetic field $B$ and the 
    pressure $p$ 
    satisfy
    \begin{align}
        &\partial_{t}u + (u \cdot \nabla) u - (B \cdot \nabla) B - \nu \Delta u + \nabla p = f, \qquad \nabla \cdot u = 0, 
        \label{eq:NSE-MHD} \\
        &\partial_{t}B + (u \cdot \nabla) B - (B \cdot \nabla) u - \nu_{m} \Delta B = 0, \qquad
        \ \ \ \ \nabla \cdot B = 0,
        \label{eq:Maxwell-MHD}
    \end{align}
with 
Dirichlet boundary conditions: $u|_{\partial{\Omega}} = 0$, $B|_{\partial{\Omega}} = B_{\circ}$.
Here 
$\nu$ is the kinematic viscosity, $\nu_{m}$ is the magnetic resistivity, 
$B_{\circ}$ is 
a constant external magnetic field,
and $f$ is an external force. 
The MHD flows entail two distinct physical processes:  the motion of fluid - 
governed by the hydrodynamics equations \eqref{eq:NSE-MHD}, 
and its interaction with the magnetic fields - governed by the Maxwell equations \eqref{eq:Maxwell-MHD}. 
Splitting the magnetic field $B$ into its mean and fluctuation $B = B_{\circ} + b$, 
the Els\"asser variables \cite{0037.28802}
    \begin{gather*}
        z^{+} = u + b, \qquad \qquad z^{-} = u - b, 
    \end{gather*}
merge the physical properties of the Navier-Stokes 
and Maxwell equations. 
The Els\"asser variables, used in plasma turbulence, differentiate between the MHD waves propagating parallel or anti-parallel to the main magnetic field $B_\circ$.
The MHD equations \eqref{eq:NSE-MHD}-\eqref{eq:Maxwell-MHD} are equivalently written in terms of the Els\"asser variables as
    \begin{align}
        \begin{cases}
            \partial_{t} z^{\pm} \mp ( B_{\circ} \cdot \nabla ) z^{\pm} 
            + (z^{\mp} \cdot \nabla) z^{\pm} - \nu^{+} \Delta z^{\pm} 
            - \nu^{-} \Delta z^{\mp} + \nabla p = f ,
\\
\nabla \cdot z^{\pm} = 0,
        \end{cases}
        \label{eq:MHD-Elsa}
    \end{align}
    and $z^{\pm} = 0$ on $\partial{\Omega}$. 
Here $\nu^{+} = (\nu + \nu_{m})/2$ and $\nu^{-} = (\nu - \nu_{m})/2$.
The nonlinear interactions between the Alfvénic fluctuations $z^\pm$,
via the cross-coupling term $(z^{\mp} \cdot \nabla) z^{\pm}$
,
are the basis of the Alfv\'en effect, a fundamental interaction process \cite{MR575264,galtier-2000,1995ApJ...438..763G,Krai65}. 

The numerical simulation of the fully-coupled system \eqref{eq:MHD-Elsa} is computationally challenging.
The monolithic methods or implicit fully-coupled algorithms in the $u,B$ variables \cite{MR3011862}, which assemble the coupled system at each time step and solve it iteratively,
are stable and accurate, but quite demanding in computational speed and storage.

The implicit-explicit (IMEX) partitioned methods \cite{MR3831959,MR4176075}, at each time step, decouple the computations and treat the subphysics/subdomain 
variables implicitly in time,  while the coupling terms are evaluated explicitly in time (by lagging or extrapolating values), hence solving two subproblems in parallel. 
The 
IMEX decoupling highly reduces the computational complexity, but often introduces a Lyapunov-type instability.
  
Inhere we propose 
an algorithm which, at each time step, partitions the 
computations 
via an iterative process, and prove that the iterations linearly converge to the solution of the fully-coupled implicit monolithic method.
%
The fully-coupled implicit monolithic method is the variable step, one-step, second-order, symplectic, midpoint method.
\\
Given $z_{n}^{\pm}$ at time $t_{n}$, the fully-coupled scheme solves for $z^{\pm}$ at time $t_{n+1}$ by
    \begin{align}
    	\begin{cases}
    		\displaystyle \frac{z_{n+1}^{\pm} - z_{n}^{\pm}}{\tau_{n}} \mp \big( B_{\circ} \cdot \nabla \big) z_{n+1/2}^{\pm} 
    		+ \big( z_{n+1/2}^{\mp} \cdot \nabla \big) z_{n+1/2}^{\pm}  
\\ 
\qquad
    		- \nu^{+} \Delta z_{n+1/2}^{\pm}  
		- \nu^{-} \Delta z_{n+1/2}^{\mp} 
    		+ \nabla p_{n+1/2}^{\pm} = f(t_{n+1/2}),  
\\
\nabla \cdot z_{n+1/2}^{\pm} = 0, 
    	\end{cases}
    	\label{eq:CN-MHD}
    \end{align}
where $t_{n+1/2} = (t_{n+1}+t_{n})/2$, 
$\tau_{n} = t_{n+1} - t_{n}$ is the local time step,
and 
$z_{n+1/2}^{\pm}$, 
$p_{n+1/2}^{\pm}$ approximate  $z^{\pm}$, $p$ at $t_{n+1/2}$, respectively. 
The implementation of 
the midpoint method 
\eqref{eq:CN-MHD} can be simplified by 
refactorization 
into a
backward Euler (BE) solve on the first half-interval, 
and a linear extrapolation (post-process):  \\
    \textit{Step 1}. BE solver for $ z_{n+1/2}^{\pm}$ on $[t_{n}, t_{n+1/2}]$ 
    \begin{align}
    \begin{cases}
    \displaystyle \frac{z_{n+1/2}^{\pm} - z_{n}^{\pm}}{\tau_{n}/2} \mp \big( B_{\circ} \cdot \nabla \big) z_{n+1/2}^{\pm} 
    + \big( z_{n+1/2}^{\mp} \cdot \nabla \big) z_{n+1/2}^{\pm}  
\\ 
\qquad
- \nu^{+} \Delta z_{n+1/2}^{\pm}  
- \nu^{-} \Delta z_{n+1/2}^{\mp} 
    + \nabla p_{n+1/2}^{\pm} = f(t_{n+1/2}),  
\\
\nabla \cdot z_{n+1/2}^{\pm} = 0.
    \end{cases}
    \label{eq:BE-MHD}
    \end{align}
    \textit{Step 2}. Post-process for $ z_{n+1}^{\pm}$: $z_{n+1}^{\pm} = 2 z_{n+1/2}^{\pm} - z_{n}^{\pm}$.\\
The post-process is equivalent to the forward Euler (FE) solver on $[t_{n+1/2}, t_{n+1}]$
\cite{MR4092601}.


We 
decouple
the implicit monolithic midpoint method \eqref{eq:CN-MHD} by 
partitioning
via iterations in the refactorized implicit BE solver \eqref{eq:BE-MHD} for $z_{n+1/2}^{\pm}$. 
    With an initial guess $z^{\pm}_{(0)} = \frac{3}{2}z^{\pm}_{n} - \frac{1}{2}z^{\pm}_{n-1}$,
    we seek
    $z^{\pm}_{(k)}$ and $p_{(k)}$ such that
    \begin{align}
        \begin{cases}
        \displaystyle \frac{z_{(k)}^{\pm} - z_{n}^{\pm}}{\tau_{n}/2} \mp \big( B_{\circ} \cdot \nabla \big) z_{(k)}^{\pm} 
        + \big( z_{(k-1)}^{\mp} \cdot \nabla \big) z_{(k)}^{\pm}  
        \\ 
        \displaystyle 
     \qquad
     - \nu^{+} \Delta z_{(k)}^{\pm}  
     - \nu^{-} \Delta z_{(k-1)}^{\mp} 
        + \nabla p_{(k)}^{\pm} = f(t_{n+1/2}),
    \\
    \nabla \cdot z_{(k)}^{\pm} = 0.
        \end{cases}
        \label{eq:BE-iter}
    \end{align}
In Theorem \ref{themreo_1}
we prove 
    that the sequence $\{ z_{(k)}^{\pm} \}_{k=0}^{\infty}$ in the BE partitioned iteration \eqref{eq:BE-iter} 
converges linearly
to $z_{n+1/2}^{\pm}$, strongly in $H^1(\Omega)$,
    provided the time step $\tau_{n}$ is small enough.
The BE partitioned iteration \eqref{eq:BE-iter} is 
    computationally efficient \cite{MR3111615}, as it decouples the cross-coupling terms $(z^{\mp} \cdot \nabla) z^{\pm}$.
    Due to the fact that the symplectic, second-order accurate, one-step
    implicit midpoint method conserves all quadratic Hamiltonians \cite{MR1755178,MR1429934,MR4092601}, the 
    algorithm 
    \eqref{eq:CN-MHD}
    conserves 
\begin{itemize}
        \item the model energy: $\displaystyle \mathcal{E}(t) =  \frac{1}{2}\int_{\Omega} (|u|^2 + |B|^{2}) \, dx$,
        \item the cross-helicity: $\displaystyle \mathcal{H}_{C}(t) = \frac{1}{2}\int_{\Omega} u \cdot B \, dx$,
        \item the magnetic-helicity: $\displaystyle \mathcal{H}_{M}(t) = \frac{1}{2}\int_{\Omega} \mathbb{A} \cdot B \, dx$ for the vector potential $\mathbb{A}$ such that $B = \nabla \times \mathbb{A}$,
\end{itemize}
    and its numerical solutions 
are second-order accurate in time.
The midpoint rule
is a special member of the robust Dahlquist-Liniger-Nevanlinna (DLN) family of one-leg second-order variable step  methods
\cite{MR714701,MR4342399,LPT23DLNadaptive,Pei24_NM,MR4501942}, 
which are unconditionally stable for any arbitrary sequence of time steps.
Also having the smallest error constant in the DLN family makes the midpoint method an ideal choice for time adaptivity, 
based on various criteria (minimum numerical dissipation, global error 
or local truncation error (LTE)). 
    Time adaptivity is an effective approach 
    which
    balances 
    computational costs and accuracy in the simulation of many stiff ordinary differential equations and fluid models \cite{pub:27216,MR4501942,LPT23DLNadaptive,Pei24_NM,Pei24_arXiv,SP23_arXiv}.  
    Here we utilize the explicit AB2-like method (a 
    variant of the two-step Adams-Bashforth method 
    \cite{MR4092601}) to estimate the LTE of the algorithm \eqref{eq:CN-MHD}, with 
    negligible extra computational costs. 
	This adaptive mechanism achieves outstanding performance in numerous highly stiff evolutionary equations \cite{MR4432611}
	 and thus deserves a fair trial in the simulations of MHD flows. 

    The rest of report is organized as follows. 
In Section \ref{sec:preliminaries} we give some  necessary notations and preliminary results.
In Section \ref{sec:num-analysis} we present detailed numerical analysis of the implicit fully-decoupled algorithm \eqref{eq:CN-MHD} with the partitioned BE iterations \eqref{eq:BE-iter}.
In subsection \ref{subsec:Parti-Iter} we prove that the sequence $\{z_{(k)}^{\pm} \}_{k=0}^{\infty}$ in \eqref{eq:BE-iter} converges linearly to  $z_{n+1/2}^{\pm}$ in \eqref{eq:BE-MHD}, strongly in $H^1(\Omega)$. 
We show in subsection \ref{subsec:Conservation} 
	that the implicit algorithm \eqref{eq:CN-MHD} is unconditionally stable, and 
	(in the ideal case)
conserves the quadratic Hamitonians: 
energy, cross-helicity and magnetic-helicity. 
The variable step error analysis for 
the Els\"asser variables $z^{\pm}$ 
is provided in subsection \ref{subsec:Err-analy}. 
%
Several numerical tests are presented in Section \ref{sec:Num-Tests} to support main results of the report. 
We verify that the algorithm \eqref{eq:CN-MHD}-\eqref{eq:BE-iter}  is second-order accurate and 
conserves over long time 
the model energy, the cross-helicity and the magnetic-helicity,
on the
two-dimensional travelling wave problem \cite{WiLaTr15}, and
    Hartmann flows \cite{MR2289481}. 

    The efficiency of time adaptivity using a LTE criterion is confirmed on these two examples, endowed a with 
    Lindberg type time component \cite{Lin74_BIT}.

    \section{Preliminaries and Notations}
	\label{sec:preliminaries}
Let $\Omega \subset \mathbb{R}^{d}$ ($d = 2,3$) be a bounded domain, and denote by
    $L^{p}(\Omega)$ the Banach space of Lebesgue measurable function $f$ such that $|f|^{p}$ is integrable. 
We denote by $H^{\ell}(\Omega)$ 
    is the usual Sobolev space $W^{\ell,2}(\Omega)$ with norm $\| \cdot \|_{\ell}$ and semi-norm $| \cdot |_{\ell}$.
In the case $\ell = 0$, $H^0(\Omega)$ reduces to the $L^{2}(\Omega)$ Hilbert space, with norm $\| \cdot \|$ and inner product 
	$( \cdot, \cdot )$. 
The function spaces $X$ for Els\"asser variables $z^{\pm}$ and $Q$ for pressure $p$ are 
\begin{gather*}
X = \big(H^1_0(\Omega) \big)^d 
= \Big\{v \in \big( H^1(\Omega) \big)^d : v|_{\partial \Omega} = 0 \Big\}, \qquad Q = \Big\{q \in L^2(\Omega) : (q, 1) = 0 \Big\}.
\end{gather*}
The divergence-free space for $X$ is 
    \begin{gather*}
        V = \Big\{v \in X : \big( \nabla \cdot v , q \big) = 0, \quad \forall q \in Q \Big\},
    \end{gather*}
 and let    
$X'$ denote the dual space of $X$ with the dual norm 
    \begin{gather}
        \big\|f\big\|_{-1} := \sup_{v \in X, \ v \neq 0} \frac{ (f,v)}{\big\|\nabla v\big\|}, \quad \forall f \in X'.
        \label{eq:dual-norm}
    \end{gather}
For spatial discretization we use the finite element method.
    Let $\{ \mathcal{T}_{h} \}$ be a family of edge-to-edge triangulation with diameter $h \in (0,1)$.  
    We denote $X^h \subset X$ and $Q^h \subset Q$ as certain finite element spaces for Els\"asser variables $z^{\pm}$ and pressure $p$ respectively. 
    The divergence-free space for $X^{h}$ is 
    \begin{gather*}
		V^{h} := \Big\{ v^{h} \in X^{h}: \big( q^{h}, \nabla \cdot v^{h} \big) = 0, \ \ 
		\forall q^{h} \in Q^{h} \Big\}.
	\end{gather*}
We assume that $X^{h}$ is the $C^{m}$-space containing polynomials of highest degree $r$
, 
$Q^{h}$ is the $C^{m}$-space containing polynomials of highest degree $s$, 
and we have the following approximations for $v \in (H^{r+1})^{d} \cap X$ 
and $q \in H^{s+1} \cap Q$ 
    \begin{equation}
        \label{eq:approx-thm}
	    \begin{aligned}
		\inf_{v^{h} \in X^{h}} \| v - v^{h} \|_{\ell_{1}} &\leq C_{A}^{\ell_{1},r} h^{r+1-\ell_{1}} |v|_{r+1}, \ \ 
		0 \leq  \ell_{1} \leq \min\{ m+1, r+1 \},  \\
		\inf_{q^{h} \in Q^{h}} \| q - q^{h} \|_{\ell_{2}} &\leq C_{A}^{\ell_{2},s} h^{s+1-\ell_{2}} |q|_{s+1}, \ \ \ 
		0 \leq  \ell_{2} \leq \min\{ m+1, s+1 \},
	    \end{aligned}
    \end{equation}
where the constants $C_{A}^{\ell_{1},r}, C_{A}^{\ell_{2},s}>0$ are independent of $h$
(see e.g., \cite{MR1278258,MR1930132}).
We also recall the inverse inequality for $X^{h}$ is 
	\begin{gather}
		|v^{h}|_{1} \leq C_{I} h^{-1} \| v^{h} \|, \qquad \forall v^{h} \in X^{h}.
		\label{eq:inv-inequal}
	\end{gather}
for $C_{I} > 0$ independent of $h$. 
To ensure the robustness of the fully discrete algorithm, we assume that $X^h$ and $Q^h$ satisfy the discrete inf-sup condition (Ladyzhenskaya-Babu\v{s}ka-Brezzi condition)
    \begin{gather}
		\inf_{q^{h} \in Q^{h}} \sup_{v^{h} \in X^{h}} \frac{ \big( \nabla \cdot v^{h}, q^{h} \big)}{ \| \nabla v^{h} \| \| q^{h} \| } \geq \beta_{\tt{is}} > 0,
		\label{eq:inf-sup-cond}
	\end{gather}
    where $\beta_{\tt{is}}$ is independent of $h$.
The Taylor-Hood ($\mathbb{P}$2-$\mathbb{P}$1) space and the Mini element space are typical examples meeting this criterion.
For any pair $(u,p) \in V \times Q$, the Stokes projection $\big( I_{\rm{St}} u, I_{\rm{St}} p \big) \in V^{h} \times Q^{h}$ is defined to be the unique solution to the following Stokes problem 
	\begin{gather}
		\begin{cases}
			\big( \nabla u - \nabla I_{\rm{St}} u, \nabla v^{h} \big) = \big( p - I_{\rm{St}} p , \nabla \cdot v^{h} \big) \\
			\qquad \quad \  - \big( q^{h}, \nabla \cdot I_{\rm{St}} u \big) = 0
		\end{cases}, 
		\quad \forall (v^{h}, q^{h}) \in X^{h} \times Q^{h}.
		\label{eq:Stokes-def}
	\end{gather}
We recall  that if the pair $(X^{h}, Q^{h})$ satisfies the discrete inf-sup condition in \eqref{eq:inf-sup-cond}, 
the following approximation property of the Stokes projection 
 \cite{MR851383,MR3561143} holds
    \begin{equation}
        \label{eq:Stoke-Approx}
        \begin{aligned}
            | u - I_{\rm{St}} u |_{1} \leq& 
            2 \Big( 1 + \frac{1}{\beta_{\rm{is}}} \Big) \inf_{v^{h} \in X^{h}} | u - v^{h} |_{1}
            + \inf_{q^{h} \in Q^{h}} \| p - q^{h} \|.
        \end{aligned}
    \end{equation}
In the error analysis we will make use of the skew-symmetric nonlinear operator 
    \begin{gather}
        \mathcal{N}(u,v,w) = \frac{1}{2} \big(u \cdot \nabla v, w \big) - \frac{1}{2} \big(u \cdot \nabla w, v \big), \qquad \forall u, v, w \in \big( H^1(\Omega) \big)^d,
        \label{eq:N-form}
    \end{gather}
which can also be written as
    \begin{align}
        \mathcal{N}(u,v,w) = \big(u \cdot \nabla v, w \big) + \frac{1}{2} \big( (\nabla \cdot u) v, w \big), \quad \forall u, v, w \in X.
        \label{eq:N-form-equi}
    \end{align}
Finally, we 
    recall the following 
    Sobolev embedding
and
Ladyzhenskaya inequality 
     \cite[p. 52]{MR2808162}
        \begin{align}
            \| u \|_{L^{3}(\Omega)} &\leq \Big( \frac{ 1}{d^{1/2}} \Big)^{d/6} 
            \| u \|^{1-d/6} \| \nabla u \|^{d/6},
            \label{eq:Ladyzh-L3}  \\
            \| u \|_{L^{4}(\Omega)} &\leq \Big( \frac{ 2(d-1) }{d^{3/2}} \Big)^{d/4} 
            \| u \|^{1-d/4} \| \nabla u \|^{d/4}, 
            \qquad \forall u \in \big( H^{1}(\Omega) \big)^{d}.
            \label{eq:Ladyzh} 
        \end{align}

\begin{confidential}
    \color{darkblue}
    For $u \in X$, 
        \begin{align*}
            \| u \|_{L^{6}} &\leq \Big( \frac{ 3(d-1) }{d^{3/2}} \Big)^{d/4}
            \| u \|^{1-d/3} \| \nabla u \|^{d/3} \\
            &\leq \Big( \frac{ 3(d-1) }{d^{3/2}} \Big)^{d/4} \big( C_{P} \| \nabla u \| \big)^{1-d/3} \| \nabla u \|^{d/3} \\
            &\leq C_{P}^{1-d/3} \Big( \frac{ 3(d-1) }{d^{3/2}} \Big)^{d/4} \| \nabla u \|,
        \end{align*}
        where $C_{P}$ is coefficient of Poincar\'e inequality. 
        \begin{align*}
            \| \nabla u \|_{L^{3}} &\leq \Big( \frac{ 1}{d^{1/2}} \Big)^{d/6}
            \| \nabla u \|^{1-d/6} \| \nabla ( \nabla u ) \|^{d/6} 
            \leq \Big( \frac{ 1}{d^{1/2}} \Big)^{d/6} \| u \|_{2}.
        \end{align*}
        Then we prove \eqref{eq:Ladyzh-Linfty}. For simplicity, we let $\Omega = \mathbb{R}^{d}$ and prove the case $u \in C_{c}^{1} (\mathbb{R}^{d})$ first. Let $Q \subset \mathbb{R}^{d}$ be an open cube containing $0$, whose sides of length $r$ are parallel to the coordinate axes. For any $x \in Q$, and any $p>d$,
        \begin{align*}
            | u(x) - u(0)| =& \Big| \int_{0}^{1} \frac{d}{dt} u(tx) dt \Big|
            \leq \int_{0}^{1} \sum_{i=1}^{d} |x_{i}| | \partial_{i} u(tx)| dt  
            \leq \int_{0}^{1} |x| |\nabla u (tx)| dt  \\
            \leq& \sqrt{d} r \int_{0}^{1} |\nabla u (tx)| dt, \\
            | \bar{u} - u(0) | =& \Big| \frac{1}{|Q|} \int_{Q} u(x) dx  
            - \frac{1}{|Q|} \int_{Q} u(0) dx \Big| 
            \leq \frac{1}{|Q|} \Big| \int_{Q} | u(x) - u(0)| dx \Big| \\
            \leq& \frac{\sqrt{d} r}{|Q|} \int_{0}^{1} \Big( \int_{Q} |\nabla u (tx)| dx \Big) dt 
            = \frac{\sqrt{d} r}{|Q|} \int_{0}^{1} \frac{1}{t^{d}} \Big( \int_{tQ} |\nabla u (y)| dy \Big) dt \\
            \leq& \frac{\sqrt{d} r}{r^{d}} \int_{0}^{1} \frac{1}{t^{d}} 
            \Big[ \Big(\int_{tQ} 1^{p'} dy \Big)^{1/p'} \Big(\int_{tQ} |\nabla u (y)|^{p} dy \Big)^{1/p} \Big] dt \\
            \leq& \frac{\sqrt{d} }{r^{d-1}} \int_{0}^{1} \frac{1}{t^{d}} \Big[ (tr)^{d/p'} \| \nabla u \|_{L^{p}(Q)} \Big] dt \qquad \big( t \in (0,1)  \big) \\
            \leq& \frac{\sqrt{d} r^{d/p'} }{r^{d-1}} \| \nabla u \|_{L^{p}(Q)} 
            \int_{0}^{1} t^{d (1/p' - 1)} dt 
            = \frac{\sqrt{d} r^{d/p'} }{r^{d-1}} \| \nabla u \|_{L^{p}(Q)}  
            \int_{0}^{1} t^{-d/p} dt \\
            \leq&  \frac{\sqrt{d} r^{1 - d/p}}{1 - d/p} \| \nabla u \|_{L^{p}(Q)}  
        \end{align*}
        By translation, the inequality remains true for all cubes $Q$ whose sides of length $r$ are parallel to the coordinate axes, i.e. 
        \begin{align*}
            | \bar{u} - u(x) | = \frac{\sqrt{d} r^{1 - d/p}}{1 - d/p} \| \nabla u \|_{L^{p}(Q)}, \qquad \forall x \in Q.
        \end{align*}
        Now for $u \in C_{c}^{1} (\mathbb{R}^{d})$ and any $x \in \mathbb{R}^{d}$, we choose open cube $Q$ whose sides of length $r$ are parallel to coordinate axes and contains $x$
        \begin{align*}
            |u(x)| \leq& | \bar{u} | + |u(x) - \bar{u} |
            \leq \frac{1}{|Q|} \Big( \int_{Q} 1^{p'} dy \Big)^{1/p'}  \Big( \int_{Q} |u|^{p} dy \Big)^{1/p} + \frac{\sqrt{d} r^{1 - d/p}}{1 - d/p} \| \nabla u \|_{L^{p}(Q)} \\
            \leq& \frac{r^{d/p'}}{r^{d}} \| u \|_{L^{p}(Q)}
            + \frac{\sqrt{d} r^{1 - d/p}}{1 - d/p} \| \nabla u \|_{L^{p}(Q)}
            = \frac{1}{r^{d/p}} \| u \|_{L^{p}(Q)} + \frac{\sqrt{d} r^{1 - d/p}}{1 - d/p} \| \nabla u \|_{L^{p}(Q)} \\
            \leq& \Big[ \| u \|_{L^{p}(Q)}^{p} + \| \nabla u \|_{L^{p}(Q)}^{p} \Big]^{1/p} 
            \Big[ \Big(\frac{1}{r^{d/p}} \Big)^{p'} + \Big( \frac{\sqrt{d} r^{1 - d/p}}{1 - d/p} \Big)^{p'} \Big]^{1/p'} \\
            =& \| u \|_{W^{1,p}(\Omega)} \Big[ \frac{1}{r^{dp'/p}} + \Big( \frac{\sqrt{d}}{1 - d/p} \Big)^{p'} r^{p'(1 - d/p)} \Big]^{1/p'}
        \end{align*}
        We minimize 
        \begin{align*}
            F(r) &= a r^{b} + \frac{1}{r^{c}}, \ \  a,b,c > 0. \\
            F'(r) &= ab r^{b-1} - c r^{-c-1} = 0, \ \ ab r^{b-1} = c r^{-c-1}, \ \ 
            r^{b+c} = \frac{c}{ab}, \ \ r = \Big( \frac{c}{ab} \Big)^{1/(b+c)}, \\
            F_{\rm{min}} &= a^{1 - b/(b+c)} \Big( \frac{c}{b} \Big)^{b/(b+c)}
            + \Big( \frac{ab}{c} \Big)^{c/(b+c)}
            = a^{c/(b+c)} \Big[ \Big( \frac{c}{b} \Big)^{b/(b+c)} + \Big( \frac{b}{c} \Big)^{c/(b+c)} \Big]
        \end{align*}
        We have 
        \begin{align*}
            &a = \Big( \frac{\sqrt{d}}{1 - d/p} \Big)^{p'}, \quad 
            b = p'(1 - d/p), \quad c = dp'/p, \quad \\
            &c/b = (dp'/p)/(p'(1 - d/p)) = (d/p)/((1 - d/p)) = d/(p - d), \quad 
            b/c = (p - d)/d,    \\
            &b + c = p' - dp'/p + dp'/p = p', \quad b/(b+c) = 1 - d/p, 
            \quad c/(b+c) = d/p \\
            & \min_{r > 0} \Big\{ \frac{1}{r^{dp'/p}} + \Big( \frac{\sqrt{d}}{1 - d/p} \Big)^{p'} r^{p'(1 - d/p)} \Big\}
            = \Big[ \Big( \frac{\sqrt{d}}{1 - d/p} \Big)^{p'} \Big]^{d/p}
            \Big[ \big( \frac{d}{p - d} \big)^{1 - d/p} + \big( \frac{p - d}{d} \big)^{d/p} \Big] \\
            &= \Big( \frac{\sqrt{d}}{1 - d/p} \Big)^{d/(p-1)} 
            \Big[ \big( \frac{d}{p - d} \big) \big( \frac{p - d}{d} \big)^{d/p} 
            + \big( \frac{p - d}{d} \big)^{d/p} \Big]
            = \Big( \frac{\sqrt{d}}{1 - d/p} \Big)^{d/(p-1)} \Big[ \big( \frac{d}{p - d} \big) + 1 \Big] \big( \frac{p - d}{d} \big)^{d/p} \\
            &= \Big( \frac{\sqrt{d}}{1 - d/p} \Big)^{d/(p-1)} \big( \frac{p}{p-d} \big)
            \big( \frac{p - d}{d} \big)^{d/p} 
        \end{align*}
        Since $1/p' = (p-1)/p$, we have 
        \begin{align*}
            & \min_{r > 0} \Big\{ \frac{1}{r^{dp'/p}} + \Big( \frac{\sqrt{d}}{1 - d/p} \Big)^{p'} r^{p'(1 - d/p)} \Big\}^{1/p'} \\
            =& \Big( \frac{\sqrt{d}}{1 - d/p} \Big)^{d/p} \big( \frac{p}{p-d} \big)^{(p-1)/p}
            \big( \frac{p - d}{d} \big)^{d(p-1)/p^{2}}:= C_{\infty,p}.
        \end{align*}
        For $u \in \big( C_{c}^{1} (\mathbb{R}^{d}) \big)^{d}$, 
    \normalcolor
\end{confidential}

    \section{Partitioned, 
    conservative, variable step, second-order method}
    \label{sec:num-analysis}

In this section, 
first we prove
the linear
convergence of the iterative solution of the partitioned method
\eqref{eq:BE-iter}, 
to the solution at time $t_{n+1/2}$ of the fully coupled midpoint method \eqref{eq:CN-MHD}. 
Secondly, we prove the 
stability of \eqref{eq:CN-MHD}, and the conservation of three quadratic invariants: 
the energy in the original variables $u$ and $B$, the cross-helicity and the magnetic-helicity.
Finally, we perform error analysis of the full time-space discretization, using the finite element method.

\subsection{Convergence of the subiterates} 
\label{subsec:Parti-Iter} 
To simplify the presentation, we
introduce the following notations related to the errors at each iteration
    \begin{align*}
        &a_{(k)} = z^+_{n+1/2} - z^+_{(k)}, \qquad b_{(k)} = z^-_{n+1/2} - z^-_{(k)},
        \\
        & \gamma_n = \max \big\{\|\nabla z^+_{n+1/2}\|, \|\nabla z^-_{n+1/2}\| \big\}, \qquad
        \delta = \Big( \frac{ 2(d-1) }{d^{3/2}} \Big)^{d/4}.
    \end{align*}
and a weighted $H^1(\Omega)$ norm of the errors $a_{(k)}, b_{(k)}$
    \begin{align*}
        {\cal G}_{(k)} 
        &= 
        \frac{4-d}{8} \big( \delta^2\gamma_n \big)^{\frac{4}{4-d}} 
        \Big( \frac{d}{2} \frac{ \nu+\nu_m }{\nu\nu_m} \Big)^{\frac{d}{4-d}}
        \big( \|a_{(k)}\|^{2} + \|b_{(k)}\|^{2}
        \big)   \\
        & \qquad
        + \frac{\nu^2 - \nu\nu_m + \nu_m^2 }{ 4(\nu+\nu_m) } \big(  \|\nabla a_{(k)} \|^2 
        + \|\nabla b_{(k)}\|^2 \big).
    \end{align*}
The next result shows the $H^1(\Omega)$ strong convergence of the iterates $z_{(k)}^{\pm}$, 
under the following time step restriction 
\begin{align}
\label{time_cond}
\tau_n 
\leq
        \frac{8}{4-d}
        \frac{1}{(  \delta^2 \gamma_n ) ^{4/(4-d)} }
        \frac{\nu^2 - \nu\nu_m + \nu_m^2 }{ \nu^2 + \nu_m^2  }
        \Big(\frac{2\nu\nu_m}{ d( \nu+\nu_m ) } \Big)^{d/(4-d)}.
\end{align}
    \begin{theorem}\label{themreo_1}
Asume that the time step $\tau_n$ satisfies condition \eqref{time_cond}.
Then, the \eqref{eq:BE-iter} sequence of iterates $\{z_{(k)}^{\pm}\}_{k \geq 0}$ 
        converges linearly to $z_{n+1/2}^{\pm}$ in the $H^1(\Omega)$ norm
        \begin{align}
        \label{eq:linear conv}
        {\cal G}_{(k)} 
        \leq \Big( 1 -  \frac{2\nu\nu_m}{\nu^2 + \nu\nu_m + \nu_m^2} \Big)  \, {\cal G}_{(k-1)}.
        \end{align}
    \end{theorem}
\begin{proof}
First, we subtract \eqref{eq:BE-iter} from \eqref{eq:BE-MHD},
\begin{confidential}
            \color{red}
            gives, respectively,
            \begin{align}\label{minus}
                & \frac{z_{n+1/2}^+ - z_{(k)}^+}{\tau_n/2} - (B_0 \cdot \nabla) \left(z_{n+1/2}^+ - z_{(k)}^+\right) + \left(z_{n+1/2}^-\cdot \nabla\right) z_{n+1/2}^+ - \left(z_{(k - 1)}^-\cdot \nabla\right) z_{(k)}^+ \notag \\
                & - \nu^+ \Delta \left(z_{n+1/2}^+ - z_{(k)}^+\right) - \nu^- \Delta \left(z_{n+1/2}^- - z_{(k-1)}^-\right)
                + \nabla \left(p_{n+1/2}^+ - p_{(k)}^+\right) = 0,  \notag \\
                & \frac{z_{n+1/2}^- - z_{(k)}^-}{\tau_n/2} - (B_0 \cdot \nabla) \left(z_{n+1/2}^- - z_{(k)}^-\right) + \left(z_{n+1/2}^+\cdot \nabla\right) z_{n+1/2}^- - \left(z_{(k - 1)}^+\cdot \nabla\right) z_{(k)}^- \notag \\
                & - \nu^+ \Delta \left(z_{n+1/2}^- - z_{(k)}^-\right) - \nu^- \Delta \left(z_{n+1/2}^+ - z_{(k-1)}^+\right)
                + \nabla \left(p_{n+1/2}^- - p_{(k)}^-\right) = 0.
            \end{align}
            Second, multiplying the first equation of \eqref{minus} with $z_{n+1/2}^+ - z_{(k)}^+$ and integrating it over $\Omega$ yields
            \begin{align}\label{integral_1}
                & \frac{2}{\tau_n} \big\|z_{n+1/2}^+ - z_{(k)}^+\big\|^2 + \int_{\Omega} \left(z_{n+1/2}^- - z_{(k - 1)}^-\right) \cdot \nabla z_{n+1/2}^+ \left(z_{n+1/2}^+ - z_{(k)}^+\right) \notag \\
                & + \nu^+ \big\|\nabla (z_{n+1/2}^+ - z_{(k)}^+) \big\|^2 + \nu^- \left \langle \nabla \left(z_{n+1/2}^- - z_{(k-1)}^-\right), \nabla \left(z_{n+1/2}^+ - z_{(k)}^+\right)  \right \rangle = 0, 
            \end{align}
            where $\left \langle \cdot , \cdot \right \rangle$ denotes the $L^2(\Omega)$ inner product and we used the equality
            \begin{align*}
                & \left(z_{n+1/2}^-\cdot \nabla\right) z_{n+1/2}^+ - \left(z_{(k - 1)}^-\cdot \nabla\right) z_{(k)}^+ \\
                = & \left(z_{n+1/2}^- - z_{(k - 1)}^-\right) \cdot \nabla z_{n+1/2}^+ + z_{(k - 1)}^- \cdot \nabla \left(z_{n+1/2}^+ - z_{(k)}^+\right).
            \end{align*}
            Similarly, multiplying the second equation of \eqref{minus} with $z_{n+1/2}^- - z_{(k)}^-$ and integrating it over $\Omega$ yields
            \\
            \normalcolor
\end{confidential}
and test the result with 
            $a_{(k)}$ and $b_{(k)}$, respectively,
to obtain
\begin{confidential}
\color{red}
\begin{align*}\label{integral_2}
    & \frac{2}{\tau_n} \big\|z_{n+1/2}^- - z_{(k)}^-\big\|^2 + \int_{\Omega} \left(z_{n+1/2}^+ - z_{(k - 1)}^+ \right) \cdot \nabla z_{n+1/2}^- \left(z_{n+1/2}^- - z_{(k)}^-\right) \notag \\
    & + \nu^+ \big\|\nabla (z_{n+1/2}^- - z_{(k)}^-) \big\|^2 + \nu^-  \left(\nabla \left(z_{n+1/2}^+ - z_{(k-1)}^+\right), \nabla \left(z_{n+1/2}^- - z_{(k)}^-\right) \right) = 0.
\end{align*}
Adding \eqref{integral_1} and the above we obtain
\begin{align*}
& \frac{2}{\tau_n} \Big(  \big\|z_{n+1/2}^+ - z_{(k)}^+\big\|^2 
+ \big\|z_{n+1/2}^- - z_{(k)}^-\big\|^2
\Big)
\\
& \quad
+ 
\int_{\Omega} \big(z_{n+1/2}^- - z_{(k - 1)}^-\big) \cdot \nabla z_{n+1/2}^+ \big(z_{n+1/2}^+ - z_{(k)}^+\big) 
\notag
\\
& \quad
+ 
\int_{\Omega} \big(z_{n+1/2}^+ - z_{(k - 1)}^+ \big) \cdot \nabla z_{n+1/2}^- \big(z_{n+1/2}^- - z_{(k)}^-\big)
\notag 
\\
& \quad
+ 
\frac{\nu+\nu_m}{2} \Big(  \big\|\nabla (z_{n+1/2}^+ - z_{(k)}^+) \big\|^2 
+ \big\|\nabla (z_{n+1/2}^- - z_{(k)}^-) \big\|^2 \Big)
\notag
\\
& 
\quad
+ 
\frac{\nu - \nu_m}{2} \Big( \nabla \big(z_{n+1/2}^- - z_{(k-1)}^-\big), \nabla \big(z_{n+1/2}^+ - z_{(k)}^+\big)  \Big)
\notag
\\
& 
\quad
+ 
\frac{\nu - \nu_m}{2} \Big(\nabla \big(z_{n+1/2}^+ - z_{(k-1)}^+\big), \nabla \big(z_{n+1/2}^- - z_{(k)}^-\big) \Big)
= 0,
\notag
\end{align*}
i.e.,
\normalcolor
\end{confidential}
        \begin{align*} 
            & \frac{2}{\tau_n} \big(  \| a_{(k)} \|^2 + \| b_{(k)}\|^2
            \big)
            + 
            \frac{\nu+\nu_m}{2} \big(  \|\nabla a_{(k)} \|^2 
            + \|\nabla b_{(k)}\|^2 \big)
            \\
            & 
            \quad
            +
            \frac{\nu - \nu_m}{2} 
            \Big( 
            ( \nabla a_{(k)}  ,  \nabla b_{(k-1)}) +  (\nabla a_{(k-1)}, \nabla b_{(k)})
            \Big)
            \\
            & \quad
            + 
            \int_{\Omega}  b_{(k-1)}  \cdot \nabla z_{n+1/2}^+ \, a_{(k)} 
            + 
            \int_{\Omega} a_{(k)} \cdot \nabla z_{n+1/2}^- \, b_{(k)}
            = 0.
            \notag
        \end{align*}
Then, using the polarization identity on the mixed diffusion terms, we have
        \begin{align*}
            & 
            \frac{\nu - \nu_m}{2} 
            \Big( 
            ( \nabla a_{(k)}  ,  \nabla b_{(k-1)}) +  (\nabla a_{(k-1)}, \nabla b_{(k)})
            \Big)
            \\
            & 
            = 
            - \frac{\nu + \nu_m}{4} \big( \| \nabla a_{(k)} \|^2 + \| \nabla b_{(k)} \|^2  \big)
            - \frac{(\nu - \nu_m)^2}{4(\nu+\nu_m)} \big( \| \nabla a_{(k-1)} \|^2 + \| \nabla b_{(k-1)} \|^2 \big) 
            \\
            & \qquad
            + \frac{|\nu - \nu_m|}{4} \bigg\| \sqrt{ \frac{ \nu + \nu_m}{ |\nu-\nu_m| }} \nabla a_{(k)}  + \mbox{\rm sign} (\nu-\nu_m)  \sqrt{\frac{|\nu-\nu_m|}{\nu+\nu_m}} \nabla b_{(k-1)}
            \bigg\|^2
            \\
            & \qquad
            + \frac{|\nu - \nu_m|}{4} \bigg\| \sqrt{ \frac{ \nu + \nu_m}{ |\nu-\nu_m| }} \nabla b_{(k)} + \mbox{\rm sign} (\nu-\nu_m)  \sqrt{\frac{|\nu-\nu_m|}{\nu+\nu_m}} \nabla a_{(k-1)}
            \bigg\|^2,
        \end{align*}
        where  
        \begin{align*}
            \text{sign}(x) \coloneqq \begin{cases}
                -1, & \text{if} \,\, x < 0,\\
                0, & \text{if} \,\, x = 0,\\
                1, & \text{if} \,\, x > 0,
            \end{cases}
        \end{align*}
and therefore
\begin{confidential}
        \color{red}
        \begin{align*}
        & \frac{2}{\tau_n} \big(  \| a_{(k)} \|^2 + \| b_{(k)}\|^2
        \big)
        + 
        \frac{\nu+\nu_m}{2} \big(  \|\nabla a_{(k)} \|^2 
        + \|\nabla b_{(k)}\|^2 \big)
        \\
        & 
        \qquad
        - 
        \frac{\nu + \nu_m}{4} \big( \| \nabla a_{(k)} \|^2 + \| \nabla b_{(k)} \|^2  \big)
        - \frac{(\nu - \nu_m)^2}{4(\nu+\nu_m)} \big( \| \nabla a_{(k-1)} \|^2 + \| \nabla b_{(k-1)} \|^2 \big) 
        \\
        & \qquad
        + \frac{|\nu - \nu_m|}{4} \bigg\| \sqrt{ \frac{ \nu + \nu_m}{ |\nu-\nu_m| }} \nabla a_{(k)}  + \mbox{\rm sign} (\nu-\nu_m)  \sqrt{\frac{|\nu-\nu_m|}{\nu+\nu_m}} \nabla b_{(k-1)}
        \bigg\|^2
        \\
        & \qquad
        + \frac{|\nu - \nu_m|}{4} \bigg\| \sqrt{ \frac{ \nu + \nu_m}{ |\nu-\nu_m| }} \nabla b_{(k)} + \mbox{\rm sign} (\nu-\nu_m)  \sqrt{\frac{|\nu-\nu_m|}{\nu+\nu_m}} \nabla a_{(k-1)}
        \bigg\|^2
        \\
        & \quad
        + 
        \int_{\Omega}  b_{(k-1)}  \cdot \nabla z_{n+1/2}^+ \, a_{(k)} 
        + 
        \int_{\Omega} a_{(k)} \cdot \nabla z_{n+1/2}^- \, b_{(k)}
        = 0
        ,
        \notag
        \end{align*}
        or,
        \normalcolor
\end{confidential}
        \begin{align*} 
            & \frac{2}{\tau_n} \big(  \| a_{(k)} \|^2 + \| b_{(k)}\|^2
            \big)
            +
            \frac{\nu+\nu_m}{4} \big(  \|\nabla a_{(k)} \|^2 
            + \|\nabla b_{(k)}\|^2 \big)
            \\
            & \qquad
            + \frac{|\nu - \nu_m|}{4} \bigg\| \sqrt{ \frac{ \nu + \nu_m}{ |\nu-\nu_m| }} \nabla a_{(k)}  + \mbox{\rm sign} (\nu-\nu_m)  \sqrt{\frac{|\nu-\nu_m|}{\nu+\nu_m}} \nabla b_{(k-1)}
            \bigg\|^2
            \\
            & \qquad
            + \frac{|\nu - \nu_m|}{4} \bigg\| \sqrt{ \frac{ \nu + \nu_m}{ |\nu-\nu_m| }} \nabla b_{(k)} + \mbox{\rm sign} (\nu-\nu_m)  \sqrt{\frac{|\nu-\nu_m|}{\nu+\nu_m}} \nabla a_{(k-1)}
            \bigg\|^2
            \\
            & \qquad
            + 
            \int_{\Omega}  b_{(k-1)}  \cdot \nabla z_{n+1/2}^+ \, a_{(k)} 
            + 
            \int_{\Omega} a_{(k)} \cdot \nabla z_{n+1/2}^- \, b_{(k)}
            \\
            &
            =
            \frac{(\nu - \nu_m)^2}{4(\nu+\nu_m)} \big( \| \nabla a_{(k-1)} \|^2 + \| \nabla b_{(k-1)} \|^2 \big).
            \notag
        \end{align*}
Now we apply the Ladyzhenskaya inequality \eqref{eq:Ladyzh} to the convective terms
\begin{confidential}
            \color{red}
            Next, we estimate the following term
            \begin{align*}
                \int_{\Omega} \left(z_{n+1/2}^{\mp} - z_{(k - 1)}^{\mp} \right) \cdot \nabla z_{n+1/2}^{\pm} \left(z_{n+1/2}^{\pm} - z_{(k)}^{\pm}\right).
            \end{align*}
            \normalcolor
            \normalcolor 
\end{confidential}
\begin{confidential}
        \color{blue}
        More specifically
        \begin{align*}
        \tag{2D}
        \|u\|_{L^4(\Omega)}
        \leq 
        2^{-1/4} \|u\|^{1/2} \|\nabla u \|^{1/2},
        \end{align*}
        (or for a weaker result Lemma 6.1 in \cite{MR0259693} and the Ladyzhenskaya inequalities in \cite{MR108962}, Temam \cite[p. 197]{MR1846644}, where the constant is $2^{1/4}\approx 1.1892$, while Galdi's constant is $2^{-1/4}\approx 0.8409$), 
        \\
        \color{darkgreen}
        while in 3D, Galdi's result is
        \begin{align*}
        \tag{3D}
        \|u\|_{L^4(\Omega)}
        \leq 
        \Big(\frac{4}{3\sqrt{3}}\Big)^{3/4} \|u\|^{1/4} \|\nabla u \|^{3/4},
        \end{align*}
        hence the Galdi constant is $(\frac{4}{3\sqrt{3}})^{3/4}\approx 0.8218$,
        with weaker results in Temam \cite[p. 200]{MR1846644}
        where the Temam constant is $2^{1/2}\approx 1.4142$.
        Lions has it with just an unspecified constant \cite[p. 220]{MR0259693}.
        \\
        Layton \cite[p.11]{MR2442411} gives the constant $\frac{4}{3\sqrt{3}} \approx 0.7698$, citing Ladyzhenskaya \cite[p. 9]{MR254401}, but the constant in Ladyzhenskaya's book is the same as the one given by Temam, namely $2^{1/2}$.
        \normalcolor
        \\
\end{confidential}
        to obtain
\begin{confidential}
            \color{darkblue}
            \begin{align*}
            E_1 
            & 
            \coloneqq \int_{\Omega} b_{(k - 1)}  \cdot \nabla z_{n+1/2}^{+} \, a_{(k)}
            \leq \|b_{(k - 1)}\|_{L^4(\Omega)} \|\nabla z_{n+1/2}^+ \| \| a_{(k)}\|_{L^4(\Omega)} 
            \\
            & 
            \leq 
            \Big( \frac{2(d-1)}{d^{3/2}} \Big)^{d/4} \|b_{(k - 1)}\|^{1-d/4} \|\nabla b_{(k - 1)} \|^{d/4}
            \|\nabla z_{n+1/2}^+\| 
            \Big( \frac{2(d-1)}{d^{3/2}} \Big)^{d/4} \|a_{(k)}\|^{1-d/4} \|\nabla a_{(k)} \|^{d/4}
            \\
            & 
            =
            \Big( \frac{2(d-1)}{d^{3/2}} \Big)^{d/2} \|b_{(k - 1)}\|^{1-d/4} \|\nabla b_{(k - 1)} \|^{d/4}
            \|\nabla z_{n+1/2}^+\| 
            \|a_{(k)}\|^{1-d/4} \|\nabla a_{(k)} \|^{d/4}
            \\
            & 
            \tag{Cauchy-Schwarz}
            \leq
            \frac{1}{2}\Big( \frac{2(d-1)}{d^{3/2}} \Big)^{d/2} 
            \|\nabla z_{n+1/2}^+\|
            \|b_{(k - 1)}\|^{2-d/2} \|\nabla b_{(k - 1)} \|^{d/2}
            \\
            & 
            \qquad
            + 
            \frac{1}{2}\Big( \frac{2(d-1)}{d^{3/2}} \Big)^{d/2}
            \|\nabla z_{n+1/2}^+\|
            \underbrace{\frac{1}{\varepsilon_n^{d/4}}\|a_{(k)}\|^{2-d/2}}_{L^{4/(4-d)}}
            \underbrace{\varepsilon_n^{d/4}\|\nabla a_{(k)} \|^{d/2}}_{L^{4/d}}
            \\
            & 
            \tag{Young}
            \leq
            \frac{1}{2}\Big( \frac{2(d-1)}{d^{3/2}} \Big)^{d/2} 
            \|\nabla z_{n+1/2}^+\|
            \frac{1}{\varepsilon_n^{d/4}}\|b_{(k - 1)}\|^{2-d/2} \varepsilon_n^{d/4}\|\nabla b_{(k - 1)} \|^{d/2}
            \\
            & 
            \qquad
            + 
            \frac{1}{2}\Big( \frac{2(d-1)}{d^{3/2}} \Big)^{d/2}
            \|\nabla z_{n+1/2}^+\|
            \Big( 
            \frac{1}{(\varepsilon_n^{d/4})^{4/(4-d)}} \frac{4-d}{4} \|a_{(k)}\|^{(2-d/2)(4/(4-d))}
            + \frac{d}{4}(\varepsilon_n^{d/4})^{4/d} \|\nabla a_{(k)} \|^{\frac{d}{2}\frac{4}{d}}
            \Big)
            \\
            & 
            = 
            \frac{1}{2}\Big( \frac{2(d-1)}{d^{3/2}} \Big)^{d/2} 
            \|\nabla z_{n+1/2}^+\|
            \Big( 
            \frac{1}{\varepsilon_n^{d/(4-d)}} \frac{4-d}{4} \|b_{(k-1)}\|^{2} 
            + \frac{d}{4}\varepsilon_n\|\nabla b_{(k-1)} \|^{2}
            \Big)
            \\
            & 
            \qquad
            + 
            \frac{1}{2}\Big( \frac{2(d-1)}{d^{3/2}} \Big)^{d/2}
            \|\nabla z_{n+1/2}^+\|
            \Big( 
            \frac{1}{\varepsilon_n^{d/(4-d)}} \frac{4-d}{4} \|a_{(k)}\|^{2}
            + \frac{d}{4}\varepsilon_n \|\nabla a_{(k)} \|^{2}
            \Big)
            \\
            & 
            =
            \frac{1}{2}\Big( \frac{2(d-1)}{d^{3/2}} \Big)^{d/2} 
            \frac{1}{\varepsilon_n^{d/(4-d)}}
            \frac{4-d}{4}
            \|\nabla z_{n+1/2}^+\|
            \Big( 
            \|a_{(k)}\|^{2} + \|b_{(k-1)}\|^{2}
            \Big)
            \\
            & 
            \qquad
            + 
            \frac{1}{2}\Big( \frac{2(d-1)}{d^{3/2}} \Big)^{d/2}
            \frac{d}{4}\varepsilon_n 
            \|\nabla z_{n+1/2}^+\|
            \Big( 
            \|\nabla a_{(k)} \|^{2} + \|\nabla b_{(k-1)} \|^{2}
            \Big)
            ,
            \end{align*}
            so similarly
            \begin{align*}
            E_1  + E_2
            & 
            \geq
            - 
            \frac{4-d}{8}\Big( \frac{2(d-1)}{d^{3/2}} \Big)^{d/2} 
            \frac{1}{\varepsilon_n^{d/(4-d)}}
            \|\nabla z_{n+1/2}^+\|
            \Big( 
            \|a_{(k)}\|^{2} + \|b_{(k)}\|^{2}
            \Big)
            \\
            & 
            \qquad
            - 
            \frac{4-d}{8}\Big( \frac{2(d-1)}{d^{3/2}} \Big)^{d/2} 
            \frac{1}{\varepsilon_n^{d/(4-d)}}
            \|\nabla z_{n+1/2}^+\|
            \Big( 
            \|a_{(k-1)}\|^{2} + \|b_{(k-1)}\|^{2}
            \Big)
            \\
            & 
            \qquad
            - 
            \frac{d}{8}\Big( \frac{2(d-1)}{d^{3/2}} \Big)^{d/2}
            \varepsilon_n 
            \|\nabla z_{n+1/2}^+\|
            \Big( 
            \|\nabla a_{(k)} \|^{2} + \|\nabla b_{(k)} \|^{2}
            \Big)
            \\
            & 
            \qquad
            - 
            \frac{d}{8}\Big( \frac{2(d-1)}{d^{3/2}} \Big)^{d/2}
            \varepsilon_n 
            \|\nabla z_{n+1/2}^+\|
            \Big( 
            \|\nabla a_{(k-1)} \|^{2} + \|\nabla b_{(k-1)} \|^{2}
            \Big)
            .
            \end{align*}
            The idea is to find a positive 
            $$
            z>0$$
            and $\varepsilon_n$
            such that 
            \begin{align}
            \label{linear-conv-goal1}
            & 
            \frac{\nu+\nu_m}{4} 
            - 
            \frac{d}{8}\Big( \frac{2(d-1)}{d^{3/2}} \Big)^{d/2}
            \varepsilon_n 
            \|\nabla z_{n+1/2}^+\|
            \\
            & 
            \geq
            (1+z)
            \Big(
            \frac{(\nu-\nu_m)^2}{4(\nu+\nu_m)}
            + 
            \frac{d}{8}\Big( \frac{2(d-1)}{d^{3/2}} \Big)^{d/2}
            \varepsilon_n 
            \|\nabla z_{n+1/2}^+\|
            \Big).
            \notag
            \end{align}
            Choosing 
            \begin{align}
            \label{eq:z}
            & 
            z 
            = \frac{2 \nu\nu_m}{\nu^2 - \nu\nu_m + \nu_m^2},
            \qquad
            1 + z 
            = \frac{\nu^2 + \nu\nu_m + \nu_m^2}{\nu^2 + \nu_m^2 - \nu\nu_m}
            ,
            \qquad
            2 + z 
            = \frac{ 2 ( \nu^2 + \nu_m^2 ) }{\nu^2 - \nu\nu_m + \nu_m^2 }
            ,
            \end{align}
            also gives 
            \begin{align*}
            \varepsilon_n 
            = 
            \frac{2}{ d \gamma_n}
            \Big( \frac{ 2 (d-1) }{d^{3/2}} \Big) ^{-d/2}
            \frac{\nu\nu_m}{ \nu+\nu_m }
            .
            \end{align*}
            Now we evaluate
            \begin{align*}
            & 
            \frac{4-d}{8}\Big( \frac{2(d-1)}{d^{3/2}} \Big)^{d/2} 
            \frac{1}{\varepsilon_n^{d/(4-d)}}
            \|\nabla z_{n+1/2}^+\|
            =
            \frac{4-d}{8}\Big( \frac{2(d-1)}{d^{3/2}} \Big)^{d/2} 
            \varepsilon_n^{-d/(4-d)}
            \gamma_n
            \\
            & 
            =
            \frac{4-d}{8}\Big( \frac{2(d-1)}{d^{3/2}} \Big)^{d/2} 
            \Big( 
            \frac{2}{ d \gamma_n}
            \Big( \frac{ 2 (d-1) }{d^{3/2}} \Big) ^{-d/2}
            \frac{\nu\nu_m}{ \nu+\nu_m }
            \Big)^{-d/(4-d)}
            \gamma_n
            \\
            & 
            =
            \frac{4-d}{8}\Big( \frac{2(d-1)}{d^{3/2}} \Big)^{d/2} 
            \Big( 
            \frac{2}{ d \gamma_n} \Big)^{-d/(4-d)}
            \Big( \frac{ 2 (d-1) }{d^{3/2}} \Big) ^{\frac{d^2}{2(4-d)}}
            \Big( \frac{\nu\nu_m}{ \nu+\nu_m }\Big)^{-d/(4-d)}
            \gamma_n
            \\
            & 
            =
            \frac{4-d}{8}\Big( \frac{2(d-1)}{d^{3/2}} \Big)^{\frac{d}{2}+\frac{d^2}{2(4-d)}} 
            \Big( 
            \frac{ d \gamma_n}{2} \Big)^{d/(4-d)}
            \Big( \frac{ \nu+\nu_m }{\nu\nu_m} \Big)^{d/(4-d)}
            \gamma_n
            \\
            & 
            =
            \frac{4-d}{8}\Big( \frac{2(d-1)}{d^{3/2}} \Big)^{\frac{2d}{4-d}} 
            \Big( \frac{d}{2} \Big)^{d/(4-d)}
            \Big( \frac{ \nu+\nu_m }{\nu\nu_m} \Big)^{\frac{d}{4-d}}
            \gamma_n^{\frac{4}{4-d}}
            \end{align*}
            and
            \begin{align*}
            &
            \frac{d}{8}\Big( \frac{2(d-1)}{d^{3/2}} \Big)^{d/2}
            \varepsilon_n 
            \|\nabla z_{n+1/2}^+\|
            =
            \frac{\bcancel{d}}{8}\cancel{\Big( \frac{2(d-1)}{d^{3/2}} \Big)^{d/2}}
            \frac{2}{ \bcancel{d} \xcancel{\gamma_n}}
            \cancel{\Big( \frac{ 2 (d-1) }{d^{3/2}} \Big) ^{-d/2}}
            \frac{\nu\nu_m}{ \nu+\nu_m }
            \xcancel{\gamma_n}
            \\
            & 
            =
            \frac{1}{4}
            \frac{\nu\nu_m}{ \nu+\nu_m }
            .
            \end{align*}
            Therefore 
            \begin{align*}
            E_1  + E_2
            & 
            \geq
            - 
            \frac{4-d}{8}\Big( \frac{2(d-1)}{d^{3/2}} \Big)^{d/2} 
            \frac{1}{\varepsilon_n^{d/(4-d)}}
            \|\nabla z_{n+1/2}^+\|
            \Big( 
            \|a_{(k)}\|^{2} + \|b_{(k)}\|^{2}
            \Big)
            \\
            & 
            \qquad
            - 
            \frac{4-d}{8}\Big( \frac{2(d-1)}{d^{3/2}} \Big)^{d/2} 
            \frac{1}{\varepsilon_n^{d/(4-d)}}
            \|\nabla z_{n+1/2}^+\|
            \Big( 
            \|a_{(k-1)}\|^{2} + \|b_{(k-1)}\|^{2}
            \Big)
            \\
            & 
            \qquad
            - 
            \frac{d}{8}\Big( \frac{2(d-1)}{d^{3/2}} \Big)^{d/2}
            \varepsilon_n 
            \|\nabla z_{n+1/2}^+\|
            \Big( 
            \|\nabla a_{(k)} \|^{2} + \|\nabla b_{(k)} \|^{2}
            \Big)
            \\
            & 
            \qquad
            - 
            \frac{d}{8}\Big( \frac{2(d-1)}{d^{3/2}} \Big)^{d/2}
            \varepsilon_n 
            \|\nabla z_{n+1/2}^+\|
            \Big( 
            \|\nabla a_{(k-1)} \|^{2} + \|\nabla b_{(k-1)} \|^{2}
            \Big)
            \\
            &
            =
            - 
            \frac{4-d}{8}\Big( \frac{2(d-1)}{d^{3/2}} \Big)^{\frac{2d}{4-d}} 
            \Big( \frac{d}{2} \Big)^{d/(4-d)}
            \Big( \frac{ \nu+\nu_m }{\nu\nu_m} \Big)^{\frac{d}{4-d}}
            \gamma_n^{\frac{4}{4-d}}
            \Big( 
            \|a_{(k)}\|^{2} + \|b_{(k)}\|^{2}
            \Big)
            \\
            & 
            \qquad
            - 
            \frac{4-d}{8}\Big( \frac{2(d-1)}{d^{3/2}} \Big)^{\frac{2d}{4-d}} 
            \Big( \frac{d}{2} \Big)^{d/(4-d)}
            \Big( \frac{ \nu+\nu_m }{\nu\nu_m} \Big)^{\frac{d}{4-d}}
            \gamma_n^{\frac{4}{4-d}}
            \Big( 
            \|a_{(k-1)}\|^{2} + \|b_{(k-1)}\|^{2}
            \Big)
            \\
            & 
            \qquad
            - 
            \frac{1}{4}
            \frac{\nu\nu_m}{ \nu+\nu_m }
            \Big( 
            \|\nabla a_{(k)} \|^{2} + \|\nabla b_{(k)} \|^{2}
            \Big)
            \\
            & 
            \qquad
            - 
            \frac{1}{4}
            \frac{\nu\nu_m}{ \nu+\nu_m }
            \Big( 
            \|\nabla a_{(k-1)} \|^{2} + \|\nabla b_{(k-1)} \|^{2}
            \Big)
            .
            \end{align*}
            hence
            \normalcolor
\end{confidential}
        \begin{align*}
            & 
            \int_{\Omega}  b_{(k-1)}  \cdot \nabla z_{n+1/2}^+ \, a_{(k)} 
            + 
            \int_{\Omega} a_{(k-1)} \cdot \nabla z_{n+1/2}^- \, b_{(k)}
            \\
            &
            \geq
            - 
            \frac{4-d}{8}\Big( \frac{2(d-1)}{d^{3/2}} \Big)^{\frac{2d}{4-d}} 
            \Big( \frac{d}{2} \Big)^{d/(4-d)}
            \Big( \frac{ \nu+\nu_m }{\nu\nu_m} \Big)^{\frac{d}{4-d}}
            \gamma_n^{\frac{4}{4-d}}
            \Big( 
            \|a_{(k)}\|^{2} + \|b_{(k)}\|^{2}
            \Big)
            \\
            & 
            \qquad
            - 
            \frac{4-d}{8}\Big( \frac{2(d-1)}{d^{3/2}} \Big)^{\frac{2d}{4-d}} 
            \Big( \frac{d}{2} \Big)^{d/(4-d)}
            \Big( \frac{ \nu+\nu_m }{\nu\nu_m} \Big)^{\frac{d}{4-d}}
            \gamma_n^{\frac{4}{4-d}}
            \Big( 
            \|a_{(k-1)}\|^{2} + \|b_{(k-1)}\|^{2}
            \Big)
            \\
            & 
            \qquad
            - 
            \frac{1}{4}
            \frac{\nu\nu_m}{ \nu+\nu_m }
            \Big( 
            \|\nabla a_{(k)} \|^{2} + \|\nabla b_{(k)} \|^{2}
            \Big)
            - 
            \frac{1}{4}
            \frac{\nu\nu_m}{ \nu+\nu_m }
            \Big( 
            \|\nabla a_{(k-1)} \|^{2} + \|\nabla b_{(k-1)} \|^{2}
            \Big).
        \end{align*}
Finally, using this inequality in the relation above implies
\begin{confidential}
\color{red}
\begin{align*} 
& \frac{2}{\tau_n} \big(  \| a_{(k)} \|^2 + \| b_{(k)}\|^2
\big)
+
\frac{\nu+\nu_m}{4} \big(  \|\nabla a_{(k)} \|^2 
+ \|\nabla b_{(k)}\|^2 \big)
\\
& \qquad
+ \frac{|\nu - \nu_m|}{4} \bigg\| \sqrt{ \frac{ \nu + \nu_m}{ |\nu-\nu_m| }} \nabla a_{(k)}  + \mbox{\rm sign} (\nu-\nu_m)  \sqrt{\frac{|\nu-\nu_m|}{\nu+\nu_m}} \nabla b_{(k-1)}
\bigg\|^2
\\
& \qquad
+ \frac{|\nu - \nu_m|}{4} \bigg\| \sqrt{ \frac{ \nu + \nu_m}{ |\nu-\nu_m| }} \nabla b_{(k)} + \mbox{\rm sign} (\nu-\nu_m)  \sqrt{\frac{|\nu-\nu_m|}{\nu+\nu_m}} \nabla a_{(k-1)}
\bigg\|^2
\\
& \qquad 
-
\frac{4-d}{8}\Big( \frac{2(d-1)}{d^{3/2}} \Big)^{\frac{2d}{4-d}} 
\Big( \frac{d}{2} \Big)^{d/(4-d)}
\Big( \frac{ \nu+\nu_m }{\nu\nu_m} \Big)^{\frac{d}{4-d}}
\gamma_n^{\frac{4}{4-d}}
\Big( 
\|a_{(k)}\|^{2} + \|b_{(k)}\|^{2}
\Big)
\\
& 
\qquad
- 
\frac{4-d}{8}\Big( \frac{2(d-1)}{d^{3/2}} \Big)^{\frac{2d}{4-d}} 
\Big( \frac{d}{2} \Big)^{d/(4-d)}
\Big( \frac{ \nu+\nu_m }{\nu\nu_m} \Big)^{\frac{d}{4-d}}
\gamma_n^{\frac{4}{4-d}}
\Big( 
\|a_{(k-1)}\|^{2} + \|b_{(k-1)}\|^{2}
\Big)
\\
& 
\qquad
- 
\frac{1}{4}
\frac{\nu\nu_m}{ \nu+\nu_m }
\Big( 
\|\nabla a_{(k)} \|^{2} + \|\nabla b_{(k)} \|^{2}
\Big)
- 
\frac{1}{4}
\frac{\nu\nu_m}{ \nu+\nu_m }
\Big( 
\|\nabla a_{(k-1)} \|^{2} + \|\nabla b_{(k-1)} \|^{2}
\Big)
\\
&
\leq
\frac{(\nu - \nu_m)^2}{4(\nu+\nu_m)} \big( \| \nabla a_{(k-1)} \|^2 + \| \nabla b_{(k-1)} \|^2 \big) 
,
\notag
\end{align*}
\color{red}
which can be arranged as
\begin{align*} 
& 
\bigg(
\frac{2}{\tau_n} 
-
\frac{4-d}{8}\Big( \frac{2(d-1)}{d^{3/2}} \Big)^{\frac{2d}{4-d}} 
\Big( \frac{d}{2} \Big)^{d/(4-d)}
\Big( \frac{ \nu+\nu_m }{\nu\nu_m} \Big)^{\frac{d}{4-d}}
\gamma_n^{\frac{4}{4-d}}
\bigg)
\Big( 
\|a_{(k)}\|^{2} + \|b_{(k)}\|^{2}
\Big)
\\
& 
\qquad
+
\bigg(
\frac{\nu+\nu_m}{4} 
- 
\frac{1}{4}
\frac{\nu\nu_m}{ \nu+\nu_m }
\bigg)
\big(  \|\nabla a_{(k)} \|^2 
+ \|\nabla b_{(k)}\|^2 \big)
\\
& \qquad
+ \frac{|\nu - \nu_m|}{4} \bigg\| \sqrt{ \frac{ \nu + \nu_m}{ |\nu-\nu_m| }} \nabla a_{(k)}  + \mbox{\rm sign} (\nu-\nu_m)  \sqrt{\frac{|\nu-\nu_m|}{\nu+\nu_m}} \nabla b_{(k-1)}
\bigg\|^2
\\
& \qquad
+ \frac{|\nu - \nu_m|}{4} \bigg\| \sqrt{ \frac{ \nu + \nu_m}{ |\nu-\nu_m| }} \nabla b_{(k)} + \mbox{\rm sign} (\nu-\nu_m)  \sqrt{\frac{|\nu-\nu_m|}{\nu+\nu_m}} \nabla a_{(k-1)}
\bigg\|^2
\\
&
\leq
\frac{4-d}{8}\Big( \frac{2(d-1)}{d^{3/2}} \Big)^{\frac{2d}{4-d}} 
\Big( \frac{d}{2} \Big)^{d/(4-d)}
\Big( \frac{ \nu+\nu_m }{\nu\nu_m} \Big)^{\frac{d}{4-d}}
\gamma_n^{\frac{4}{4-d}}
\Big( 
\|a_{(k-1)}\|^{2} + \|b_{(k-1)}\|^{2}
\Big)
\\
& 
\qquad
+
\bigg(
\frac{(\nu - \nu_m)^2}{4(\nu+\nu_m)} 
+
\frac{1}{4}
\frac{\nu\nu_m}{ \nu+\nu_m }
\bigg)
\big( \| \nabla a_{(k-1)} \|^2 + \| \nabla b_{(k-1)} \|^2 \big) 
,
\notag
\end{align*}
\color{red}
or
\begin{align*} 
& 
\bigg(
\frac{2}{\tau_n} 
-
\frac{4-d}{8}\Big( \frac{2(d-1)}{d^{3/2}} \Big)^{\frac{2d}{4-d}} 
\Big( \frac{d}{2} \Big)^{d/(4-d)}
\Big( \frac{ \nu+\nu_m }{\nu\nu_m} \Big)^{\frac{d}{4-d}}
\gamma_n^{\frac{4}{4-d}}
\bigg)
\Big( 
\|a_{(k)}\|^{2} + \|b_{(k)}\|^{2}
\Big)
\\
& 
\qquad
+
\frac{\nu^2 + \nu\nu_m + \nu_m^2 }{ 4(\nu+\nu_m) }
\big(  \|\nabla a_{(k)} \|^2 
+ \|\nabla b_{(k)}\|^2 \big)
\\
& \qquad
+ \frac{|\nu - \nu_m|}{4} \bigg\| \sqrt{ \frac{ \nu + \nu_m}{ |\nu-\nu_m| }} \nabla a_{(k)}  + \mbox{\rm sign} (\nu-\nu_m)  \sqrt{\frac{|\nu-\nu_m|}{\nu+\nu_m}} \nabla b_{(k-1)}
\bigg\|^2
\\
& \qquad
+ \frac{|\nu - \nu_m|}{4} \bigg\| \sqrt{ \frac{ \nu + \nu_m}{ |\nu-\nu_m| }} \nabla b_{(k)} + \mbox{\rm sign} (\nu-\nu_m)  \sqrt{\frac{|\nu-\nu_m|}{\nu+\nu_m}} \nabla a_{(k-1)}
\bigg\|^2
\\
&
\leq
\frac{4-d}{8}\Big( \frac{2(d-1)}{d^{3/2}} \Big)^{\frac{2d}{4-d}} 
\Big( \frac{d}{2} \Big)^{d/(4-d)}
\Big( \frac{ \nu+\nu_m }{\nu\nu_m} \Big)^{\frac{d}{4-d}}
\gamma_n^{\frac{4}{4-d}}
\Big( 
\|a_{(k-1)}\|^{2} + \|b_{(k-1)}\|^{2}
\Big)
\\
& 
\qquad
+
\frac{\nu^2 - \nu\nu_m + \nu_m^2}{4(\nu+\nu_m)} 
\big( \| \nabla a_{(k-1)} \|^2 + \| \nabla b_{(k-1)} \|^2 \big) 
,
\notag
\end{align*}
\color{red}
so, using Wenlong's notation $\delta = \Big( \frac{2(d-1)}{d^{3/2}} \Big)^{\frac{d}{4}}$, we have
\begin{align*} 
& 
\bigg(
\frac{2}{\tau_n} 
-
\frac{4-d}{8} \big( \delta^2\gamma_n \big)^{\frac{4}{4-d}} 
\Big( \frac{d}{2} \frac{ \nu+\nu_m }{\nu\nu_m} \Big)^{\frac{d}{4-d}}
\bigg)
\Big( 
\|a_{(k)}\|^{2} + \|b_{(k)}\|^{2}
\Big)
\\
& 
\qquad
+
\frac{\nu^2 + \nu\nu_m + \nu_m^2 }{ 4(\nu+\nu_m) }
\big(  \|\nabla a_{(k)} \|^2 
+ \|\nabla b_{(k)}\|^2 \big)
\\
& \qquad
+ \frac{|\nu - \nu_m|}{4} \bigg\| \sqrt{ \frac{ \nu + \nu_m}{ |\nu-\nu_m| }} \nabla a_{(k)}  + \mbox{\rm sign} (\nu-\nu_m)  \sqrt{\frac{|\nu-\nu_m|}{\nu+\nu_m}} \nabla b_{(k-1)}
\bigg\|^2
\\
& \qquad
+ \frac{|\nu - \nu_m|}{4} \bigg\| \sqrt{ \frac{ \nu + \nu_m}{ |\nu-\nu_m| }} \nabla b_{(k)} + \mbox{\rm sign} (\nu-\nu_m)  \sqrt{\frac{|\nu-\nu_m|}{\nu+\nu_m}} \nabla a_{(k-1)}
\bigg\|^2
\\
&
\leq
\frac{4-d}{8} \big(  \delta^2 \gamma_n \big) ^{\frac{4}{4-d}} 
\Big( \frac{d}{2} \frac{ \nu+\nu_m }{\nu\nu_m} \Big)^{\frac{d}{4-d}}
\Big( 
\|a_{(k-1)}\|^{2} + \|b_{(k-1)}\|^{2}
\Big)
\\
& 
\qquad
+
\frac{\nu^2 - \nu\nu_m + \nu_m^2}{4(\nu+\nu_m)} 
\big( \| \nabla a_{(k-1)} \|^2 + \| \nabla b_{(k-1)} \|^2 \big) 
,
\notag
\end{align*}
hence dropping the positive terms in the LHS yields
\normalcolor
\end{confidential}
        \begin{align*} 
            & 
            \bigg(
            \frac{2}{\tau_n} 
            -
            \frac{4-d}{8} \big( \delta^2\gamma_n \big)^{\frac{4}{4-d}} 
            \Big( \frac{d}{2} \frac{ \nu+\nu_m }{\nu\nu_m} \Big)^{\frac{d}{4-d}}
            \bigg)
            \big( 
            \|a_{(k)}\|^{2} + \|b_{(k)}\|^{2}
            \big)
            \\
            & 
            \qquad
            +
            \frac{\nu^2 + \nu\nu_m + \nu_m^2 }{ 4(\nu+\nu_m) }
            \big(  \|\nabla a_{(k)} \|^2 
            + \|\nabla b_{(k)}\|^2 \big)
            \\
            &
            \leq
            \frac{4-d}{8} \big(  \delta^2 \gamma_n \big) ^{\frac{4}{4-d}} 
            \Big( \frac{d}{2} \frac{ \nu+\nu_m }{\nu\nu_m} \Big)^{\frac{d}{4-d}}
            \big( 
            \|a_{(k-1)}\|^{2} + \|b_{(k-1)}\|^{2}
            \big)
            \\
            & 
            \qquad
            +
            \frac{\nu^2 - \nu\nu_m + \nu_m^2}{4(\nu+\nu_m)} 
            \big( \| \nabla a_{(k-1)} \|^2 + \| \nabla b_{(k-1)} \|^2 \big) .
            \notag
        \end{align*}
\begin{confidential}
            \color{red}
            As above, see \eqref{linear-conv-goal1}, we would like the time step $\tau_n$ to be chosen in a way to satisfy
            \begin{align*}
            & 
            \frac{2}{\tau_n} 
            -
            \frac{4-d}{8} \big( \delta^2\gamma_n \big)^{\frac{4}{4-d}} 
            \Big( \frac{d}{2} \frac{ \nu+\nu_m }{\nu\nu_m} \Big)^{\frac{d}{4-d}}
            \\
            & 
            \geq
            (1+z)
            \frac{4-d}{8} \big(  \delta^2 \gamma_n \big) ^{\frac{4}{4-d}} 
            \Big( \frac{d}{2} \frac{ \nu+\nu_m }{\nu\nu_m} \Big)^{\frac{d}{4-d}}
            \end{align*}
            or
            \begin{align*}
            & 
            \frac{2}{\tau_n}  
            \geq
            (2+z)
            \frac{4-d}{8} \big(  \delta^2 \gamma_n \big) ^{\frac{4}{4-d}} 
            \Big( \frac{d}{2} \frac{ \nu+\nu_m }{\nu\nu_m} \Big)^{\frac{d}{4-d}},
            \end{align*}
            where $z$ was defined in \eqref{eq:z}.
            Hence
            \begin{align*}
            & 
            \frac{\cancel{2}}{\tau_n}  
            \geq
            \frac{ \cancel{2} ( \nu^2 + \nu_m^2 ) }{\nu^2 - \nu\nu_m + \nu_m^2 }
            \frac{4-d}{8} \big(  \delta^2 \gamma_n \big) ^{\frac{4}{4-d}} 
            \Big( \frac{d}{2} \frac{ \nu+\nu_m }{\nu\nu_m} \Big)^{\frac{d}{4-d}}
            ,
            \end{align*}
            i.e., 
            \begin{align*}
            & 
            \tau_n
            \leq
            \frac{8}{4-d}
            \big(  \delta^2 \gamma_n \big) ^{\frac{-4}{4-d}} 
            \frac{\nu^2 - \nu\nu_m + \nu_m^2 }{ \nu^2 + \nu_m^2  }
            \Big( \frac{d}{2} \frac{ \nu+\nu_m }{\nu\nu_m} \Big)^{-\frac{d}{4-d}}
            \\
            & 
            = 
            \frac{8}{4-d}
            \frac{1}{(  \delta^2 \gamma_n ) ^{4/(4-d)} }
            \frac{\nu^2 - \nu\nu_m + \nu_m^2 }{ \nu^2 + \nu_m^2  }
            \Big(\frac{2\nu\nu_m}{ d( \nu+\nu_m ) } \Big)^{d/(4-d)}
            ,
            \end{align*}
            which is the time restriction \eqref{time_cond}.
            \\
            This then gives
            \begin{align*} 
            & 
            \frac{\nu^2 + \nu\nu_m + \nu_m^2}{\nu^2 + \nu_m^2 - \nu\nu_m}
            \frac{4-d}{8} \big( \delta^2\gamma_n \big)^{\frac{4}{4-d}} 
            \Big( \frac{d}{2} \frac{ \nu+\nu_m }{\nu\nu_m} \Big)^{\frac{d}{4-d}}
            \big( 
            \|a_{(k)}\|^{2} + \|b_{(k)}\|^{2}
            \big)
            \\
            & 
            \qquad
            +
            \frac{\nu^2 + \nu\nu_m + \nu_m^2 }{ 4(\nu+\nu_m) }
            \big(  \|\nabla a_{(k)} \|^2 
            + \|\nabla b_{(k)}\|^2 \big)
            \\
            &
            \leq
            \frac{4-d}{8} \big(  \delta^2 \gamma_n \big) ^{\frac{4}{4-d}} 
            \Big( \frac{d}{2} \frac{ \nu+\nu_m }{\nu\nu_m} \Big)^{\frac{d}{4-d}}
            \big( 
            \|a_{(k-1)}\|^{2} + \|b_{(k-1)}\|^{2}
            \big)
            \\
            & 
            \qquad
            +
            \frac{\nu^2 - \nu\nu_m + \nu_m^2}{4(\nu+\nu_m)} 
            \big( \| \nabla a_{(k-1)} \|^2 + \| \nabla b_{(k-1)} \|^2 \big)
            \notag
            ,
            \notag
            \end{align*}
            i.e., 
            \begin{align*} 
            & 
            \frac{\nu^2 + \nu\nu_m + \nu_m^2}{\nu^2 + \nu_m^2 - \nu\nu_m}
            \bigg(
            \frac{4-d}{8} \big( \delta^2\gamma_n \big)^{\frac{4}{4-d}} 
            \Big( \frac{d}{2} \frac{ \nu+\nu_m }{\nu\nu_m} \Big)^{\frac{d}{4-d}}
            \big( 
            \|a_{(k)}\|^{2} + \|b_{(k)}\|^{2}
            \big)
            \\
            & 
            \qquad
            +
            \frac{\nu^2 - \nu\nu_m + \nu_m^2 }{ 4(\nu+\nu_m) }
            \big(  \|\nabla a_{(k)} \|^2 
            + \|\nabla b_{(k)}\|^2 \big)
            \bigg)
            \\
            &
            \leq
            \frac{4-d}{8} \big(  \delta^2 \gamma_n \big) ^{\frac{4}{4-d}} 
            \Big( \frac{d}{2} \frac{ \nu+\nu_m }{\nu\nu_m} \Big)^{\frac{d}{4-d}}
            \big( 
            \|a_{(k-1)}\|^{2} + \|b_{(k-1)}\|^{2}
            \big)
            \\
            & 
            \qquad
            +
            \frac{\nu^2 - \nu\nu_m + \nu_m^2}{4(\nu+\nu_m)} 
            \big( \| \nabla a_{(k-1)} \|^2 + \| \nabla b_{(k-1)} \|^2 \big)
            \notag
            ,
            \notag
            \end{align*}
            or 
            \begin{align*} 
            & 
            {\cal G}_{(k)}
            \leq
            \frac{\nu^2 - \nu\nu_m + \nu_m^2}{\nu^2 + \nu\nu_m + \nu_m^2}
            {\cal G}_{(k-1)}
            =
            \Big( 1 - \frac{ 2\nu\nu_m }{\nu^2 + \nu\nu_m + \nu_m^2} \Big)
            {\cal G}_{(k-1)}
            \notag
            \end{align*}
            \normalcolor
            \\
\end{confidential}
Hence, under the time step restriction \eqref{time_cond}, the above estimate yields the linear convergence relation \eqref{eq:linear conv}.
    \end{proof}

\subsection{Stability and a balance of energy}
\label{subsec:Conservation}
Now we consider the stability of the algorithm \eqref{eq:CN-MHD}. 
We denote the discrete kinetic energy of the system by $\mathcal{E}^N$, the viscous dissipation rate by $\mathcal{D}^N$:
    \begin{equation}
        \label{energy_def}
        \begin{split}
            \mathcal{E}^N =& \frac{1}{2} \big( \|z_N^+ \|^2 + \|z_N^- \|^2 \big), \\
        \mathcal{D}^N =& \min\{\nu, \nu_m\} \sum^{N-1}_{n=1} \tau_n \big(\|\nabla z_{n+1/2}^+ \|^2 + \|\nabla z_{n+1/2}^- \|^2 \big) \\
        & + |v^{-}| \sum^{N-1}_{n=1} \tau_n \big\|\nabla z_{n+1/2}^+ + \text{sign}(v^{-}) \nabla z_{n+1/2}^-\big\|^2
        .
        \end{split}
    \end{equation}

    \begin{theorem}
        The algorithm in \eqref{eq:CN-MHD} is unconditionally stable and satisfies 
        \begin{align}
            &\mathcal{E}^N + \frac{1}{2}\mathcal{D}^N
            + \frac{|v^{-}|}{2} \sum^{N-1}_{n=1} \tau_n \big\|\nabla z_{n+1/2}^+ 
            + {\rm{sign}}(v^{-}) \nabla z_{n+1/2}^-\big\|^2 \notag \\
            &\leq \sum_{n=1}^{N-1} \frac{\tau_{n}}{\min\{ \nu, \nu_{m} \}} \| f(t_{n+1/2}) \|_{-1}^{2} + \mathcal{E}^0.
            \label{eq:zStab-conclusion}
        \end{align}
        Moreover, 
        in the absence of external force $f$, the following energy 
        balance
        holds
        \begin{gather}
            \mathcal{E}^N + \mathcal{D}^N = \mathcal{E}^0.
            \label{eq:Energy-eq}
        \end{gather}
    \end{theorem}
\begin{proof}
We first take inner product of \eqref{eq:CN-MHD} with $2 z_{n+1/2}^{\pm}$, and use the skew-symmetry properties \eqref{eq:N-form}, \eqref{eq:N-form-equi} to obtain
        \begin{gather}
            \frac{1}{\tau_n} \big( \|z_{n+1}^{\pm} \|^2 - \big\|z_{n}^{\pm} \|^2 \big)
            + 2 \nu^{+} \|\nabla z_{n+1/2}^{\pm} \|^2 
            + 2 \nu^{-} \big( z_{n+1/2}^{\mp}, z_{n+1/2}^{\pm} \big) \notag \\
            = 2 ( f(t_{n+1/2}), z_{n+1/2}^{\pm}). 
            \label{eq:zStab-eq1}
        \end{gather}
Then we use the following polarization identity 
        \begin{align*}
& 
\nu^{-} \big( \nabla z_{n+1/2}^{\mp}, \nabla z_{n+1/2}^{\pm} \big) 
\\
& 
=
 \frac{|\nu^{-}|}{2} \big\|\nabla z_{n+1/2}^{\mp} + \text{sign}(\nu^{-}) \nabla z_{n+1/2}^{\pm} \big\|^2 
 - \frac{|\nu^{-}|}{2} \big( \|\nabla z_{n+1/2}^{\mp} \|^2 + \|\nabla z_{n+1/2}^{\pm} \|^2 \big),
        \end{align*}
to obtain
\begin{confidential}
        \color{darkblue}
        \begin{align*}
            &\frac{1}{\tau_n} \big( \|z_{n+1}^+ \|^2 + \|z_{n+1}^- \|^2 
            - \|z_{n}^+ \|^2 - \|z_{n}^- \|^2 \big) 
            + 2(\nu^{+} - |\nu^{-}|) \big( \|\nabla z_{n+1/2}^+ \|^2 + \|\nabla z_{n+1/2}^- \|^2 \big)  \\
            & + 2|\nu^{-}| \big\|\nabla z_{n+1/2}^+  + \text{sign}(\nu^{-}) \nabla z_{n+1/2}^-\big\|^2 \\
            &= 2 \big( ( f(t_{n+1/2}), z_{n+1/2}^{+}) + ( f(t_{n+1/2}), z_{n+1/2}^{-}) \big).
        \end{align*}
        \normalcolor
\end{confidential}
 \begin{align}
&
\label{eq:zStab-eq2}
\frac{1}{2\tau_n} \big( \|z_{n+1}^+ \|^2 + \|z_{n+1}^- \|^2 
            - \|z_{n}^+ \|^2 - \|z_{n}^- \|^2 \big) 
\\
& 
\quad
+ (\nu^{+} - |\nu^{-}|) \big( \|\nabla z_{n+1/2}^+ \|^2 + \|\nabla z_{n+1/2}^- \|^2 \big)  
+ |\nu^{-}| \big\|\nabla z_{n+1/2}^+  + \text{sign}(\nu^{-}) \nabla z_{n+1/2}^-\big\|^2 
\notag 
\\
&
= ( f(t_{n+1/2}), z_{n+1/2}^{+}) + ( f(t_{n+1/2}), z_{n+1/2}^{-})
.
\notag
\end{align}
Using 
$\nu^{+} - |\nu^{-}| = \min\{ \nu, \nu_{m} \}$, the definition of dual norm \eqref{eq:dual-norm}, and Young's inequality 
\begin{confidential}
            \color{darkblue}
            \begin{align*}
                &( f(t_{n+1/2}), z_{n+1/2}^{+}) + ( f(t_{n+1/2}), z_{n+1/2}^{-}) \\
                \leq& \| f(t_{n+1/2}) \|_{-1} \| \nabla z_{n+1/2}^{+} \| 
                + \| f(t_{n+1/2}) \|_{-1} \| z_{n+1/2}^{-} \|  \\
                \leq& \frac{1}{2 \min\{ \nu, \nu_{m} \}} \| f(t_{n+1/2}) \|_{-1}^{2}
                + \frac{\min\{ \nu, \nu_{m} \}}{2} \| \nabla z_{n+1/2}^{+} \|^{2} 
                + \frac{1}{2 \min\{ \nu, \nu_{m} \}} \| f(t_{n+1/2}) \|_{-1}^{2}
                + \frac{\min\{ \nu, \nu_{m} \}}{2} \| \nabla z_{n+1/2}^{-} \|^{2}
            \end{align*}
            \normalcolor
\end{confidential}
we have
\begin{align}
&
\frac{1}{2\tau_n} \big( \|z_{n+1}^+ \|^2 + \|z_{n+1}^- \|^2 
            - \|z_{n}^+ \|^2 - \|z_{n}^- \|^2 \big) 
\label{eq:zStab-eq3}
\\
& 
+ \frac{1}{2}\min\{ \nu, \nu_{m} \} \big( \|\nabla z_{n+1/2}^+ \|^2 + \|\nabla z_{n+1/2}^- \|^2 \big)  
+ |\nu^{-}| \big\|\nabla z_{n+1/2}^+  + \text{sign}(\nu^{-}) \nabla z_{n+1/2}^-\big\|^2 
\notag 
\\
&
\leq 
\frac{1}{\min\{ \nu, \nu_{m} \}} \| f(t_{n+1/2}) \|_{-1}^{2}.
\notag
        \end{align}
\begin{confidential}
            \color{darkblue}
            \begin{align*}
                &\frac{1}{2} \big( \|z_{n+1}^+ \|^2 + \|z_{n+1}^- \|^2 
                - \|z_{n}^+ \|^2 - \|z_{n}^- \|^2 \big) 
                + \frac{\min\{ \nu, \nu_{m} \}\tau_{n}}{2} \big( \|\nabla z_{n+1/2}^+ \|^2 + \|\nabla z_{n+1/2}^- \|^2 \big)   \\
                & + |\nu^{-}| \tau_{n} \big\|\nabla z_{n+1/2}^+  + \text{sign}(\nu^{-}) \nabla z_{n+1/2}^-\big\|^2  \\
                &\leq \frac{\tau_{n}}{\min\{ \nu, \nu_{m} \}} \| f(t_{n+1/2}) \|_{-1}^{2}. 
            \end{align*}
            \begin{align*}
                &\frac{1}{2} \big( \|z_{N}^+ \|^2 + \|z_{N}^- \big)
                + \frac{\min\{ \nu, \nu_{m} \}}{2} \sum_{n=1}^{N-1} \tau_{n}
                \big( \|\nabla z_{n+1/2}^+ \|^2 + \|\nabla z_{n+1/2}^- \|^2 \big) \\
                &+ |\nu^{-}| \sum_{n=1}^{N-1} \tau_{n} \big\|\nabla z_{n+1/2}^+ + \text{sign}(\nu^{-}) \nabla z_{n+1/2}^-\big\|^2 \\
                &\leq \sum_{n=1}^{N-1} \frac{\tau_{n}}{\min\{ \nu, \nu_{m} \}} \| f(t_{n+1/2}) \|_{-1}^{2} + \frac{1}{2} \big( \|z_{0}^+ \|^2 + \|z_{0}^- \big)
            \end{align*}
            \normalcolor
\end{confidential}
Finally, if we multiply \eqref{eq:zStab-eq3} by $\tau_{n}$ and sum 
over $n$ from $1$ to $N-1$, we obtain \eqref{eq:zStab-conclusion}.
In the case when the source function $f$ vanishes, summation of \eqref{eq:zStab-eq2} over $n$ from $1$ to $N-1$ 
yields \eqref{eq:Energy-eq}.
    \end{proof}

    \subsection{Conservation of quadratic invariants
    : energy, cross-helicity and magnetic-helicity}
    \label{subsec:Quad-Invar} 
We note that the fully coupled system \eqref{eq:CN-MHD}
can be equivalently written as
    \begin{align}
    	\begin{cases}
    		\displaystyle \frac{u_{n+1} - u_{n}}{\tau_{n}} - \big( B_{n+1/2} \cdot \nabla \big) B_{n+1/2} + \big( u_{n+1/2} \cdot \nabla \big) u_{n+1/2}  
    		- \nu \Delta u_{n+1/2} 
\\ 
\displaystyle 
\qquad 
+ \nabla \Big( \frac{p_{n+1/2}^+ + p_{n+1/2}^-}{2} \Big) = f(t_{n+1/2}),  
\vspace{0.2cm}
\\
            \displaystyle \frac{B_{n+1} - B_{n}}{\tau_n} + \big( u_{n+1/2} \cdot \nabla \big) B_{n+1/2}
            - \big( B_{n+1/2} \cdot \nabla \big) u_{n+1/2}  - \nu_m \Delta B_{n+1/2} \\
            \displaystyle \qquad \qquad \qquad \qquad \qquad \quad \ + \nabla \Big( \frac{p_{n+1/2}^+ - p_{n+1/2}^-}{2} \Big) = 0, 
\\
\nabla \cdot u_{n+1/2} = 0, \qquad 
\nabla \cdot B_{n+1/2} = 0,
\end{cases}
\label{eq:CN-MHD-equi}
\end{align}
where by {\it Step 2} in \eqref{eq:BE-MHD} we have
$u_{n+1/2} = (u_{n+1} + u_{n})/2$ and $B_{n+1/2} = (B_{n+1} + B_{n})/2$. 
\\
We denote the energy, the cross-helicity and the magnetic-helicity corresponding to the solution at time $t_{n}$ of algorithm \eqref{eq:CN-MHD-equi} by
\begin{gather*}
        \mathcal{E}_{n} = \frac{1}{2} \int_{\Omega} (|u_{n}|^{2} + |B_{n}|^{2}) dx, \quad 
        \mathcal{H}_{C_{n}} = \frac{1}{2} \int_{\Omega} u_{n} \cdot B_{n} dx, \quad 
        \mathcal{H}_{M_{n}} = \frac{1}{2} \int_{\Omega} \mathbb{A}_{n} \cdot B_{n} dx. 
\end{gather*}

    \begin{theorem}\label{energy_cross_magnetic}
The algorithm in \eqref{eq:CN-MHD} or \eqref{eq:CN-MHD-equi}, conserves the energy, cross-helicity and magnetic-helicity 
\begin{align*}
{\cal E}_n = {\cal E}_0, 
\qquad
{\cal H}_{C_n} = {\cal H}_{C_0}, \qquad
{\cal H}_{M_n} = {\cal H}_{M_0},
\end{align*}
in the ideal case, i.e., in the absence of external forcing terms, and for zero kinematic viscosity and magnetic diffusivity.
    \end{theorem}
    \begin{proof}
We assume that $f = 0$, and $\nu = \nu_{m} = 0$.
First, we take inner product of first equation of \eqref{eq:CN-MHD-equi} with $u_{n+1/2}$, 
the second equation of \eqref{eq:CN-MHD-equi} with $B_{n+1/2}$ 
and use the skew-symmetry properties \eqref{eq:N-form}, \eqref{eq:N-form-equi} 
to obtain
        \begin{gather}
            \frac{1}{2} \big( \|u_{n+1} \|^2 + \|B_{n+1} \|^2 \big)
            - \frac{1}{2} \big( \|u_{n} \|^2 + \|B_{n} \|^2 \big) = 0. 
            \label{eq:Stab-eq2}
        \end{gather}
Summation of \eqref{eq:Stab-eq2} over $n$ from $1$ to $N-1$ and 
proves the conservation of the model energy, i.e., $\mathcal{E}_{N} = \mathcal{E}_{0}$.
 
Secondly, 
we take the inner product of the first equation in \eqref{eq:CN-MHD-equi} with $B_{n+1/2}$, 
the second equation of \eqref{eq:CN-MHD-equi} with $u_{n+1/2}$, to obtain
\begin{align*}
& \frac{1}{2 \tau_{n}} \big( u_{n+1} - u_{n}, B_{n+1} + B_{n} \big)
	+ \big( (u_{n+1/2} \cdot \nabla) u_{n+1/2}, B_{n+1/2} \big) = 0, 
\notag 
\\
& 
\frac{1}{2 \tau_{n}} \big( B_{n+1} - B_{n}, u_{n+1} + u_{n} \big)
            + \big( (u_{n+1/2} \cdot \nabla) B_{n+1/2}, u_{n+1/2} \big) = 0
            ,
\end{align*}
and using \eqref{eq:N-form}, \eqref{eq:N-form-equi},
these yield
\begin{align}
            \frac{1}{\tau_n} \int_{\Omega} u_{n+1} \cdot B_{n+1} dx 
            - \frac{1}{\tau_n} \int_{\Omega} u_{n} \cdot B_{n} dx = 0,
            \label{eq:Stab-eq4}
\end{align}
which proves the conservation of cross helicity, i.e.,
$\mathcal{H}_{C_{N}} = \mathcal{H}_{C_{0}}$.

For the result on the magnetic helicity, we  begin by recalling 
(see e.g., \cite[Theorem 3.2, Chapter I]{MR851383})
that,  since $\nabla \cdot B_{n} = 0$, there exists a function $\mathbb{A}_{n} \in \big(H^1(\Omega) \big)^{d}$ such that $B_{n} = \nabla \times \mathbb{A}_{n}$ for all $n$. Moreover, $\nabla \cdot \mathbb{A}_{n} = 0$ and $\mathbb{A}_{n} = 0$ on the boundary.
Next, 
we take the inner product of the second equation 
in
\eqref{eq:CN-MHD-equi} with $\mathbb{A}_{n+1/2} = (\mathbb{A}_{n+1} + \mathbb{A}_{n})/2$
        \begin{align}
            &\frac{1}{\tau_n} \big( B_{n+1} - B_{n}, \mathbb{A}_{n+1/2} \big) \notag \\
            &+ \big( (u_{n+1/2} \cdot \nabla) B_{n+1/2}, \mathbb{A}_{n+1/2} \big)
            - \big( (B_{n+1/2} \cdot \nabla) u_{n+1/2}, \mathbb{A}_{n+1/2} \big) = 0.
            \label{eq:Stab-eq5}
        \end{align}
Then, 
we use the identity 
        \begin{align*}
            \nabla \times ( B \times u ) 
            = (\nabla \cdot u)B + (u \cdot \nabla)B - (\nabla \cdot B)u - (B \cdot \nabla)u, 
        \end{align*}
to write the convective terms as follows
\begin{align}
&
\big( (u_{n+1/2} \cdot \nabla) B_{n+1/2}, \mathbb{A}_{n+1/2} \big)
            - \big( (B_{n+1/2} \cdot \nabla) u_{n+1/2}, \mathbb{A}_{n+1/2} \big) 
\label{eq:Stab-eq6} 
\\
&
=
 \big( \nabla \times ( B_{n+1/2} \times u_{n+1/2} ), \mathbb{A}_{n+1/2} \big)
            - \big( (\nabla \cdot u_{n+1/2}) B_{n+1/2}, \mathbb{A}_{n+1/2} \big) 
\notag 
\\
&
\qquad
+ \big( (\nabla \cdot B_{n+1/2} ) u_{n+1/2}, \mathbb{A}_{n+1/2} \big) 
\notag 
\\
& 
= 
\big( \nabla \times ( (\nabla \times \mathbb{A}_{n+1/2} ) \times u_{n+1/2} ), \mathbb{A}_{n+1/2} \big) 
\notag 
\\
&
= \big( (\nabla \!\times\! \mathbb{A}_{n+1/2} ) \!\times\! u_{n+1/2}, \nabla \!\times\! \mathbb{A}_{n+1/2} \big) 
 - \int_{\partial{\Omega}} \!\!\! \big( (\nabla \!\times\! \mathbb{A}_{n+1/2} ) \!\times\! u_{n+1/2} \big) \!\times\! \mathbb{A}_{n+1/2} 
= 0, 
\notag 
\end{align}
where the last equality
comes from 
the homogeneous Dirichlet boundary condition on the velocity
$u|_{\partial{\Omega}} = 0$,
and the orthogonality of $\nabla \times \mathbb{A}_{n+1/2}$ and $(\nabla \times \mathbb{A}_{n+1/2}) \times u_{n+1/2}$. 
Using \eqref{eq:Stab-eq6} in \eqref{eq:Stab-eq5} 
gives
\begin{align}
            \int_{\Omega} \mathbb{A}_{n+1} \cdot B_{n+1} dx 
            + \int_{\Omega} \mathbb{A}_{n} \cdot B_{n+1} dx 
            - \int_{\Omega} \mathbb{A}_{n+1} \cdot B_{n} dx 
            = \int_{\Omega} \mathbb{A}_{n} \cdot B_{n} dx.
\label{eq:Stab-eq7}
\end{align}
Finally, 
by the Gauss' divergence theorem we have
\begin{gather*}
            \int_{\Omega} \mathbb{A}_{n} \cdot B_{n+1} dx 
            - \int_{\Omega} \mathbb{A}_{n+1} \cdot B_{n} dx
            = \int_{\Omega} \nabla \cdot \big( \mathbb{A}_{n+1} \times \mathbb{A}_{n} \big) dx 
            = 0,
        \end{gather*}
and therefore
\eqref{eq:Stab-eq7} 
implies the conservation of magnetic-helicity 
$\mathcal{H}_{M_{N}} = \mathcal{H}_{M_{0}}$.
    \end{proof}

    \subsection{Error Analysis}
    \label{subsec:Err-analy}
In this section, 
\color{red}$z^{\pm}(t_n)$ \normalcolor 
represent the true solutions 
of \eqref{eq:MHD-Elsa} at time $t_{n}$.
For the fully space-time discretization, $z_{n}^{\pm,h} \in X^{h}$ and $p_{n+1/2}^{\pm,h} \in Q^{h}$ are numerical approximations of $\color{red}z^{\pm}(t_{n})$ and $p(t_{n+1/2})$ respectively. 

\begin{confidential}
We denote 
\color{red}$z_{n+1/2}^{\pm} = (z^{\pm}(t_{n+1}) + z^{\pm}(t_{n}))/2$ : this is bad notation !!! \normalcolor
and 
{\color{blue}$z_{n+1/2}^{\pm,h} = (z_{n+1}^{\pm,h} + z_{n}^{\pm,h})/2$ this should not be a notation, it is a value}.
\end{confidential}

    The variational formulation of the fully-decoupled algorithm in \eqref{eq:CN-MHD} is: given $z_{n}^{\pm,h} \in X^h$, we solve $z_{n+1}^{\pm,h} \in X^h, p_{n+1/2}^{\pm,h} \in Q^h$ such that for all $(v^h, q^h) \in X^h \times Q^h$,
\begin{align}
\begin{cases}
\displaystyle \Big( \frac{z_{n+1}^{\pm,h} - z_n^{\pm,h}}{\tau_n}, v^h \Big)
            \mp \mathcal{N} \big(B_{\circ}, z_{n+1/2}^{\pm,h}, v^h \big) 
            + \mathcal{N} \big( z_{n+1/2}^{\mp,h}, z_{n+1/2}^{\pm,h}, v^h \big) 
\\ 
\displaystyle + \nu^+ \big( \nabla z_{n+1/2}^{\pm,h}, \nabla v^h \big)
            + \nu^- \big( \nabla z_{n+1/2}^{\mp,h}, \nabla v^h \big)
            - \big( p_{n+1/2}^{\pm,h}, \nabla \cdot v^h \big) 
            = \big( f(t_{n+1/2}), v^h \big),  
\\
 \displaystyle \big( \nabla \cdot z_{n+1/2}^{\pm,h}, q^{\pm,h} \big) = 0.
        \end{cases}
        \label{eq:CN-MHD-Weak}
    \end{align}
The following result 
recalls the consistency of the midpoint finite difference 
method.
\begin{lemma}
\label{lemma_taylor}
For any given 
nodes $\{t_n\}_{n=0}^{N}$ on the time interval $[0,T]$,  
if the mapping $u: [0,T] \to H^{\ell}(\Omega)$ is smooth enough, 
then
        \begin{align}
            \Big\| \big( u(\cdot,t_{n+1}) + u(\cdot,t_{n}) \big)/2 - u(\cdot,t_{n+1/2}) \Big\|_{\ell}
            \leq& C \tau_{n}^{3} \int_{t_n}^{t_{n+1}} \| u_{tt}(\cdot,t) \|_{\ell}^{2} dt, 
            \label{eq:consistency-eq1} \\
            \Big\| \frac{u(\cdot, t_{n+1}) - u(\cdot, t_{n})}{\tau_{n}} - u(\cdot,t_{n+1/2}) \Big\|_{\ell} 
            \leq& C \tau_{n}^{3} \int_{t_n}^{t_{n+1}} \| u_{ttt}(\cdot,t) \|_{\ell}^{2} dt.
            \label{eq:consistency-eq2}
        \end{align}
    \end{lemma}
\begin{proof}
A direct calculation involving the
Taylor 
expansion with integral remainder 
about $t_{n+1/2}$. 
    \end{proof}

    We also need the following Bochner spaces on the time interval $[0,T]$ 
    \begin{align*}
        L^{p} \big(0,T; (H^{\ell}(\Omega))^{d} \big)
        =& \Big\{ v(\cdot,t) \in (H^{\ell}(\Omega))^{d}: \| v \|_{p,\ell} 
        = \Big( \int_{0}^{T} \|v(\cdot,t) \|_{\ell}^{p} dt \Big)^{1/p}  < \infty \Big\},  \\
        L^{\infty} \big(0,T; (H^{\ell}(\Omega))^{d} \big)
        =& \Big\{ v(\cdot,t) \in (H^{\ell}(\Omega))^{d}: \| v \|_{\infty,\ell} 
        = \sup_{0 < t < T} \|v(\cdot,t) \|_{\ell} < \infty \Big\},  \\
        L^{p} \big(0,T; X' \big)
        =& \Big\{ v(\cdot,t) \in X': \| v \|_{p,-1} 
        = \Big( \int_{0}^{T} \|v(\cdot,t) \|_{-1}^{p} dt \Big)^{1/p}  < \infty \Big\},
    \end{align*}
    and the corresponding discrete Bochner spaces
    \begin{align*}
        &\ell^{\infty} \big( \{ t_{n} \}_{n=0}^{N}; (H^{\ell}(\Omega))^{d} \big) \\
        &:= \Big\{ v(\cdot,t_{n}) \in (H^{\ell}(\Omega))^{d}: \| |v| \|_{\infty,\ell} 
        := \max_{0 \leq n \leq N} \| v (\cdot,t_{n}) \|_{\ell} < \infty \Big\}, \\
        &\ell^{\infty,1/2} \big( \{ t_{n} \}_{n=0}^{N}; (H^{\ell}(\Omega))^{d} \big) \\
        &:= \Big\{ v(\cdot,t_{n+1/2}) \in (H^{\ell}(\Omega))^{d}: \| |v| \|_{\infty,\ell,1/2} 
        := \max_{1 \leq n \leq N-1} \| v (\cdot,t_{n+1/2}) \|_{\ell} < \infty \Big\}, \\
        &\ell^{p,1/2} \big( \{ t_{n} \}_{n=0}^{N}; (H^{\ell}(\Omega))^{d} \big) \\
        &:= \Big\{ v(\cdot,t_{n+1/2}) \in (H^{\ell}(\Omega))^{d}: \| |v| \|_{p,\ell,1/2} 
        := \Big( \sum_{n=1}^{N-1} \tau_{n} \| v (\cdot,t_{n+1/2}) \|_{\ell}^{p} \Big)^{1/p} < \infty \Big\}, \\
        &\ell^{p,1/2} \big( \{ t_{n} \}_{n=0}^{N}; X' \big) \\
        &:= \Big\{ v(\cdot,t_{n+1/2}) \in X': \| |v| \|_{p,\ell,-1} 
        := \Big( \sum_{n=1}^{N-1} \tau_{n} \| v (\cdot,t_{n+1/2}) \|_{-1}^{p} \Big)^{1/p} < \infty \Big\}. \\
    \end{align*}
We denote the errors 
in the algorithm 
\eqref{eq:CN-MHD-Weak} at time $t_{n}$ by 
$e_{n}^{\pm} = z_{n}^{\pm,h} - z^{\pm} (t_{n}) $.
\begin{confidential}
{\\ \color{blue} This should be $e_{n}^{\pm,h}$ !!!}
\end{confidential}
    \begin{theorem}
        \label{error_theorem_1}
        We assume that $(z^{\pm}, p)$ in \eqref{eq:MHD-Elsa} satisfies the 
regularity assumptions        
        \begin{align}
            &z^{\pm} \in \ell^{\infty} \big( \{ t_{n} \}_{n\!=\!0}^{N}; (H^{\ell})^{d} \big)
            \cap \ell^{\infty,1/2} \big( \{ t_{n} \}_{n\!=\!0}^{N}; (H^{\ell})^{d} \big)
            \cap \ell^{2,1/2} \big( \{ t_{n} \}_{n\!=\!0}^{N}; (H^{r\!+\!1})^{d} \cap (H^{2})^{d} \big)
            , \notag \\
            &z^{\pm} \in \ell^{4,1/2} \big( \{ t_{n} \}_{n\!=\!0}^{N}; (H^{1})^{d} \big)
            \cap L^{4} \big(0,T; (H^{1})^{d} \big) \notag \\
            &z_t^{\pm} \in L^{2}\big(0,T;(H^{r+1})^{d} \big), \quad 
            z_{tt}^{\pm} \in L^{2}\big(0,T;(H^{r+1})^{d} \cap (H^1)^{d} \big), \notag \\
            &z_{ttt}^{\pm} \in L^{2}\big( 0,T;X' \big), \qquad \qquad 
            p \in \ell^{2,1/2} \big( \{ t_{n} \}_{n\!=\!0}^{N}; H^{s+1} \big).
            \label{eq:regu-L2}
        \end{align} 
        and time step size $\tau_{n}$ satisfies
        \begin{align}
            &\tau_{n} < \frac{1}{16 \big( C_{1}^{\ast}h Z_{n}^{(1)} + 256 C_{2}^{\ast} Z_{n}^{(2)} \big)}, \quad \forall n = 1,2, \cdots, N-1, 
            \label{eq:step-res-Error-L2} 
        \end{align}
where
        \begin{align*}
            &C_{1}^{\ast} = \frac{2(3(d\!-\!1))^{2d/3}C_{P}^{3-d} C_{I} (C_{A}^{1,0})^{2} }{\epsilon \nu^{\ast} d^{7d/6}}\big(1 + \frac{1}{\beta_{\tt{is}}} \big)^{2}, \ \ 
            C_{2}^{\ast} = \frac{(3(d\!-\!1))^{4d/3} C_{P}^{2(3-d)}}{32 (\nu^{\ast})^{3} d^{7d/3}}, 
\notag 
\\
&
Z_{n}^{(1)} 
= 
\mbox{$\frac{1}{2}$} \max \big\{ | z^{+}(t_{n+1}) + z^{+}(t_{n} ) |_{2} ,  | z^{-}(t_{n+1}) + z^{-}(t_{n} )  |_{2} 
 \big\}, 
\\
& 
Z_{n}^{(2)} = \frac{1}{2} \max \big\{  \| \nabla \big( z^{+}(t_{n+1}) + z^{+}(t_{n}) \big) \|^{4} , \| \nabla \big( z^{-}(t_{n+1}) + z^{-}(t_{n}) \big) \|^{4} \big\},  
\notag 
        \end{align*}
$\nu^{\ast} = \nu^{+} - |\nu^{-}| = \min\{ \nu, \nu_{m} \}$, 
and 
$C_{P}$ represents the 
constant in the  Poincar\'e inequality. 
Then the following error estimates hold 
\begin{align}
\|e_{N}^{+} \|^2 + \|e_{N}^{-} \|^2 
& 
\leq 
Ch^{2r} \big( \| |z^{+}| \|_{\infty,r+1}^{2} + \| |z^{-}| \|_{\infty,r+1}^{2} \big) 
\label{eq:error-L2-final} 
\\
& 
\quad
+ \exp \Big( C \sum_{n=1}^{N-1} \big( C_{1}^{\ast}h Z_{n}^{(1)} + C_{2}^{\ast} Z_{n}^{(2)} \big) \tau_{n} \Big) F \big( h^{2r}, h^{2s+2}, \tau_{\rm{max}}^{4}, \| B_{\circ} \|^{2} \big),    
\notag 
\end{align}
and
\begin{align}
&
\nu^{\ast} \sum^{N-1}_{n=0} \!\! \tau_n  
\big( \|\nabla (e^+_{n+1} +  e^+_{n} )/2 \|^2  + \|\nabla ( e^-_{n+1} + e^-_{n} ) /2 \|^2 \big)
\notag
\\
& 
\leq
C \nu^{\ast} h^{2r} \big(\tau_{\rm{max}}^{4} \| z^{\pm} \|_{2,r+1}^{2} 
+  \| |z^{\pm}| \|_{2,r+1,1/2}^{2} \big) 
\notag 
\\
& \quad
+ \exp \Big( C \sum_{n=1}^{N-1} \big( C_{1}^{\ast}h Z_{n}^{(1)} + C_{2}^{\ast} Z_{n}^{(2)} \big) \tau_{n} \Big) F \big( h^{2r}, h^{2s+2}, \tau_{\rm{max}}^{4}, \| B_{\circ} \|^{2} \big),   
\notag
\end{align}
where 
$\tau_{\max} = \max_{0 \leq n \leq N} \{\tau_n\}$ and 
        \begin{align*}
            &F \big( h^{2r}, h^{2s+2}, \tau_{\rm{max}}^{4}, \| B_{\circ} \|^{2} \big) \\
            &= \frac{h^{2r\!+\!2}}{\nu^{\ast}}\! \big( \| z^{+} \|_{2,r\!+\!1}^{2} \!+\! \| z^{-} \|_{2,r\!+\!1}^{2} \big)
            +\! \frac{\big(\!(\nu^{+})^{2} \!\!+\!\! (\nu^{-})^{2} \!\big) h^{2r}\tau_{\rm{max}}^{4}}{ \nu^{\ast}}\! 
            \big( \| z_{tt}^{+} \|_{2,r\!+\!1}^{2} \!\!+\!\! \| z_{tt}^{-} \|_{2,r\!+\!1}^{2} \big) \\
            &+ \frac{\big((\nu^{+})^{2} + (\nu^{-})^{2} \big) h^{2r}}{ \nu^{\ast}}
            \big( \| |z^{+}| \|_{2,r+1,1/2}^{2} + \| |z^{-}| \|_{2,r+1,1/2}^{2} \big) \notag \\
            &+\! \frac{ (\nu^{-})^{2} \tau_{\rm{max}}^{4}}{\nu^{\ast}} 
            \big( \| \nabla z_{tt}^{+} \|_{2,0}^{2} \!+\! \| \nabla z_{tt}^{-} \|_{2,0}^{2} \big)
            \!+\! \frac{h^{2s\!+\!2}}{\nu^{\ast}} \big( \| |p^{+}| \|_{2,s\!+\!1,1\!/\!2}^{2} 
            \!+\! \| |p^{-}| \|_{2,s\!+\!1,1\!/\!2}^{2} \big) \notag \\
            &+ \frac{ \tau_{\rm{max}}^{4}}{\nu^{\ast}} \big( \| \nabla z_{ttt}^{+} \|_{2,0}^{2} + \| \nabla z_{ttt}^{-} \|_{2,0}^{2} \big)
            + \frac{\tau_{\rm{max}}^{4}}{\nu^{\ast}} \| B_{\circ} \|^{2} \big( \| \nabla z_{tt}^{+} \|_{2,0}^{2} + \| \nabla z_{tt}^{-} \|_{2,0}^{2} \big)
            \notag \\
            &+ \frac{h^{2r} \tau_{\rm{max}}^{4}}{\nu^{\ast}} \| B_{\circ} \|^{2}
            \big( \| z_{tt}^{\pm} \|_{2,r+1}^{2} + \| z_{tt}^{\pm} \|_{2,r+1}^{2} \big) \notag \\
            &+ \frac{h^{2r} }{\nu^{\ast}} \| B_{\circ} \|^{2} \big( \| |z^{+}| \|_{2,r+1,1/2}^{2} + \| |z^{-}| \|_{2,r+1,1/2}^{2} \big) 
            \notag \\
            &+ \frac{ \tau_{\rm{max}}^{4} }{\nu^{\ast}} \big( \| |z^{-}| \|_{\infty,1}^{2} \| \nabla z_{tt}^{+} \|_{2,0}^{2} + \| |z^{+}| \|_{\infty,1}^{2} \| \nabla z_{tt}^{-} \|_{2,0}^{2} 
            + \| |z^{+}| \|_{\infty,1,1/2}^{2} \| \nabla z_{tt}^{-} \|_{2,0}^{2} \notag \\
            &\qquad \qquad \qquad \qquad + \| |z^{-}| \|_{\infty,1,1/2}^{2} \| \nabla z_{tt}^{+} \|_{2,0}^{2} \big) 
            \notag \\
            &+ \frac{ h^{2r}\tau_{\rm{max}}^{4}}{\nu^{\ast}} \big( \| | \nabla z^{-} | \|_{\infty,0}^{2} \| z^{+}_{tt} \|_{2,r+1}^{2} + \| | \nabla z^{+} | \|_{\infty,0}^{2} \| z^{-}_{tt} \|_{2,r+1}^{2} \big)  \notag \\
            & + \frac{ h^{2r}}{\nu^{\ast}} \big( \| |z^{+}| \|_{2,r+1,1/2}^{2} +  \| |z^{-}| \|_{2,r+1,1/2}^{2} \big)
            + \big( \| \xi_{0}^{+,h} \|^{2} +  \| \xi_{0}^{-,h} \|^{2} \big).
        \end{align*}
    \end{theorem}
\begin{proof}
The true solutions of \eqref{eq:MHD-Elsa} at time $t_{n+1/2}$ satisfies
        \begin{align}
            & \Big( \frac{ \color{red} z^{\pm}(t_{n+1}) - z^{\pm}(t_n) }{\tau_n} , v^h \Big) 
            \mp \mathcal{N} \big( B_{\circ}, z^{\pm}(t_{n+1/2}), v^h \big) 
            + \mathcal{N} \big( z^{\mp}(t_{n+1/2}), z^{\pm}(t_{n+1/2}), v^h \big) 
\label{eq:exact-eq} 
\\
            & + \nu^+ \big( \nabla z^{\pm}(t_{n+1/2}), \nabla v^h \big) 
            + \nu^- \big( \nabla z^{\mp}(t_{n+1/2}), \nabla v^h \big) 
            - \big( p^{\pm}(t_{n+1/2}), \nabla \cdot v^h \big) \notag \\
& 
= \big( f(t_{n+1/2}), v^h \big) 
+ \Big( \frac{ \color{red}z^{\pm}(t_{n+1}) - z^{\pm}(t_n)}{\tau_n} - z_{t}^{\pm}(t_{n+1/2}), v^h \Big), \qquad \forall v^{h} \in V^{h}. \notag 
        \end{align}
We denote by $I_{\rm{St}} z_{n}^{\pm}$ the velocity components of Stokes projection of $(z_{n}^{\pm},0)$ onto $V^h \times Q^h$ and decompose the error $e_{n}^{\pm}$ as
\begin{gather}
e_{n}^\pm 
= \xi^{\pm,h}_{n} + \color{red}\eta^{\pm}(t_{n})\normalcolor, \quad 
\xi^{\pm,h}_{n} := z_{n}^{\pm,h} -  I_{\rm{St}} \color{red} z^{\pm}(t_{n}), \quad
\color{red} \eta^{\pm}(t_{n}) := I_{\rm{St}} z^{\pm}(t_{n}) - z^{\pm}(t_{n}).
            \label{eq:decomp-error}
        \end{gather}
For convenience of the presentation, 
we denote 
\begin{align*}
& 
\xi^{\pm}_{n+1/2} := (\xi^{\pm}_{n+1} + \xi^{\pm}_{n} )/2 \normalcolor
,
\\
& 
\color{red}z_{n+1/2}^{\pm} 
:= (z^{\pm}(t_{n+1}) + z^{\pm}(t_{n}))/2,
\quad
\eta^{\pm}_{n+1/2} := (\eta^{\pm}(t_{n+1}) + \eta^{\pm}(t_{n}) )/2\normalcolor.
\end{align*}
We subtract \eqref{eq:exact-eq} from the first equation of \eqref{eq:CN-MHD-Weak} 
to obtain: 
\begin{confidential}
    \color{darkblue}
        \begin{align*}
            &\Big( \frac{z_{n+1}^{\pm,h} - z_n^{\pm,h}}{\tau_n}, v^h \Big) 
            - \Big( \frac{z_{n+1}^{\pm} - z_n^{\pm}}{\tau_n}, v^h \Big) 
            + \nu^+ \big( \nabla z_{n+1/2}^{\pm,h}, \nabla v^h \big)
            - \nu^+ \big( \nabla z^{\pm}(t_{n+1/2}), \nabla v^h \big)  \\
            & + \nu^- \big( \nabla z_{n+1/2}^{\mp,h}, \nabla v^h \big)
            - \nu^- \big( \nabla z^{\mp}(t_{n+1/2}), \nabla v^h \big) 
            + \big( p^{\pm}(t_{n+1/2}) - q^{h}, \nabla \cdot v^h \big) \\
            & \mp \mathcal{N} \big(B_{\circ}, z_{n+1/2}^{\pm,h}, v^h \big) 
            \pm \mathcal{N} \big( B_{\circ}, z^{\pm}(t_{n+1/2}), v^h \big)
            + \mathcal{N} \big( z_{n+1/2}^{\mp,h}, z_{n+1/2}^{\pm,h}, v^h \big)
            - \mathcal{N} \big( z^{\mp}(t_{n+1/2}), z^{\pm}(t_{n+1/2}), v^h \big) \\
            &= - \Big( \frac{z_{n+1}^{\pm} - z_n^{\pm}}{\tau_n} - z_{t}^{\pm}(t_{n+1/2}), v^h \Big).
        \end{align*}
\begin{align*}
&
\Big( \frac{e_{n+1}^{\pm} - e_n^{\pm}}{\tau_n}, v^h \Big) 
+ \nu^+ \big( \nabla z_{n+1/2}^{\pm,h} - \nabla z_{n+1/2}^{\pm}, \nabla v^h \big)
            + \nu^+ \big( \nabla z_{n+1/2}^{\pm} - \nabla z^{\pm}(t_{n+1/2}), \nabla v^h \big)  
\\
&
+ \nu^- \big( \nabla z_{n+1/2}^{\mp,h} - \nabla z_{n+1/2}^{\mp}, \nabla v^h \big)
            + \nu^- \big( \nabla z_{n+1/2}^{\mp} - \nabla z^{\mp}(t_{n+1/2}), \nabla v^h \big)
            + \big( p^{\pm}(t_{n+1/2}) - q^{h}, \nabla \cdot v^h \big) \\
            & \mp \mathcal{N} \big(B_{\circ}, z_{n+1/2}^{\pm,h} - z_{n+1/2}^{\pm}, v^h \big)
            \mp \mathcal{N} \big( B_{\circ}, z_{n+1/2}^{\pm} - z^{\pm}(t_{n+1/2}), v^h \big) \\
            &= \mathcal{N} \big( z^{\mp}(t_{n+1/2}), z^{\pm}(t_{n+1/2}), v^h \big)
            - \mathcal{N} \big( z_{n+1/2}^{\mp,h}, z_{n+1/2}^{\pm,h}, v^h \big)
            - \Big( \frac{z_{n+1}^{\pm} - z_n^{\pm}}{\tau_n} - z_{t}^{\pm}(t_{n+1/2}), v^h \Big). 
        \end{align*}
    \normalcolor
\end{confidential}
\begin{align}
&
\notag
\Big( \frac{e_{n+1}^{\pm} - e_n^{\pm}}{\tau_n}, v^h \Big) 
+ \nu^+ \big( \nabla  (e_{n+1}^{\pm} +  e_{n}^{\pm} )/2 , \nabla v^h \big) 
+ \nu^- \big( \nabla  ( e_{n+1}^{\mp} + e_{n}^{\mp} ) / 2 , \nabla v^h \big) 
\\
& \qquad
\mp \mathcal{N} \big(B_{\circ},  (e_{n+1}^{\mp}  + e_{n}^{\mp} ) / 2 , v^h \big)
\notag 
\\
&
= 
\nu^+ \big( \nabla z^{\pm}(t_{n+1/2}) - \nabla ((z^{\pm}(t_{n+1}) + z^{\pm}(t_{n}))/2) , \nabla v^h \big)
\notag 
\\
&
\quad
+ \nu^- \big( \nabla z^{\mp}(t_{n+1/2}) - \nabla ((z^{\mp}(t_{n+1}) + z^{\mp}(t_{n}))/2) ,  \nabla v^h \big) 
\notag 
\\
&
\quad
\pm \mathcal{N} \big( B_{\circ}, (z^{\pm}(t_{n+1}) + z^{\pm}(t_{n}))/2 - z^{\pm}(t_{n+1/2}), v^h \big)
            - \big( p^{\pm}(t_{n+1/2}) - q^{h}, \nabla \cdot v^h \big) 
\notag 
\\
& 
\quad
+ \mathcal{N} \big( z^{\mp}(t_{n+1/2}), z^{\pm}(t_{n+1/2}), v^h \big)
            - \mathcal{N} \big( z_{n+1/2}^{\mp,h}, z_{n+1/2}^{\pm,h}, v^h \big) 
\notag 
\\
&
\quad
- \Big( \frac{
	z^{\pm}(t_{n+1}) - z^{\pm}(t_n)}{\tau_n} - z_{t}^{\pm}(t_{n+1/2}), v^h \Big)
,
\quad \forall (v^{h},q^{h}) \in V^{h} \times Q^{h}
.
\notag 
        \end{align}
\begin{confidential}
    \color{darkblue}
    \begin{align*}
        &\Big( \frac{\xi_{n+1}^{\pm,h} - \xi_{n}^{\pm,h}}{\tau_n}, v^h \Big) 
        + \Big( \frac{\eta_{n+1}^{\pm} - \eta_n^{\pm}}{\tau_n}, v^h \Big) 
        + \nu^+ \big( \nabla \xi_{n+1/2}^{\pm,h} + \nabla \eta_{n+1/2}^{\pm}, \nabla v^h \big)
        \\
        &+ \nu^- \big( \nabla \xi_{n+1/2}^{\mp,h} + \nabla \eta_{n+1/2}^{\mp}, \nabla v^h \big)
        \mp \mathcal{N} \big(B_{\circ}, \xi_{n+1/2}^{\pm,h} + \eta_{n+1/2}^{\pm}, v^h \big) \\ 
        &= \nu^+ \big( \nabla z^{\pm}(t_{n+1/2}) - \nabla z_{n+1/2}^{\pm}, \nabla v^h \big)
        + \nu^- \big( \nabla z^{\mp}(t_{n+1/2}) - \nabla z_{n+1/2}^{\mp}, \nabla v^h \big) \notag \\
        &\pm \mathcal{N} \big( B_{\circ}, z_{n+1/2}^{\pm} - z^{\pm}(t_{n+1/2}), v^h \big)
        - \big( p^{\pm}(t_{n+1/2}) - q^{h}, \nabla \cdot v^h \big) \notag \\
        &+ \mathcal{N} \big( z^{\mp}(t_{n+1/2}), z^{\pm}(t_{n+1/2}), v^h \big)
        - \mathcal{N} \big( z_{n+1/2}^{\mp,h}, z_{n+1/2}^{\pm,h}, v^h \big) 
        - \Big( \frac{z_{n+1}^{\pm} - z_n^{\pm}}{\tau_n} - z_{t}^{\pm}(t_{n+1/2}), v^h \Big)
,
\quad \forall (v^{h},q^{h}) \in V^{h} \times Q^{h}
.
    \end{align*}
    \normalcolor
\end{confidential}        
By the decomposition of the error $e_{n}^{\pm}$ 
\eqref{eq:decomp-error}, 
the above 
can be equivalently written as
\begin{align}
\notag
&
\Big( \frac{\xi_{n+1}^{\pm,h} - \xi_{n}^{\pm,h}}{\tau_n}, v^h \Big)
            + \nu^+ \big( \nabla \xi_{n+1/2}^{\pm,h}, \nabla v^h \big)
            + \nu^- \big( \nabla \xi_{n+1/2}^{\mp,h}, \nabla v^h \big) 
            \mp \mathcal{N} \big(B_{\circ}, \xi_{n+1/2}^{\pm,h}, v^h \big)  
\\
& 
=
- \Big( \frac{
	\eta^{\pm}(t_{n+1}) - \eta^{\pm}(t_n)}{\tau_n}, v^h \Big)
- \nu^+ \big( \nabla \color{red}\eta_{n+1/2}^{\pm}\normalcolor, \nabla v^h \big)
            - \nu^- \big( \nabla \color{red}\eta_{n+1/2}^{\mp}\normalcolor, \nabla v^h \big)
            \pm \mathcal{N} \big(B_{\circ}, \color{red}\eta_{n+1/2}^{\pm}\normalcolor, v^h \big) \notag 
\\
& \quad
+ \nu^+ \big( \nabla z^{\pm}(t_{n+1/2}) - \nabla \color{red}z_{n+1/2}^{\pm}\normalcolor, \nabla v^h \big)
            + \nu^- \big( \nabla z^{\mp}(t_{n+1/2}) - \nabla \color{red}z_{n+1/2}^{\mp}\normalcolor, \nabla v^h \big)
            \notag 
\\
& \quad
 - \big( p^{\pm}(t_{n+1/2}) - q^{h}, \nabla \cdot v^h \big)
            - \Big( \frac{\color{red}z^{\pm}(t_{n+1})\normalcolor - \color{red}z^{\pm}(t_n)\normalcolor}{\tau_n} - z_{t}^{\pm}(t_{n+1/2}), v^h \Big)
            \notag 
\\
&\quad
\pm \mathcal{N} \big( B_{\circ}, \color{red}z_{n+1/2}^{\pm}\normalcolor - z^{\pm}(t_{n+1/2}), v^h \big)
            + \mathcal{N} \big( z^{\mp}(t_{n+1/2}), z^{\pm}(t_{n+1/2}), v^h \big) \notag 
\\
&
\quad
- \mathcal{N} \big( z_{n+1/2}^{\mp,h}, z_{n+1/2}^{\pm,h}, v^h \big)
.
\notag
\end{align}
We set $v^{h} = \xi_{n+1/2}^{\pm,h}$ 
and use skew-symmetry of $\mathcal{N}$
to obtain 
\begin{align}
\label{eq:errorL2-eq3}
&
\frac{1}{2\tau_{n}} \big( \| \xi_{n+1}^{\pm,h} \|^{2} - \| \xi_{n}^{\pm,h} \|^{2} \big) + \nu^+ \| \nabla \xi_{n+1/2}^{\pm,h} \|^{2} 
            + \nu^- \big( \nabla \xi_{n+1/2}^{\mp,h}, \nabla \xi_{n+1/2}^{\pm,h} \big) 
\\
& 
=
 - \Big( \frac{\color{red}\eta^{\pm}(t_{n+1})\normalcolor - \color{red}\eta^{\pm}(t_n)\normalcolor}{\tau_n}, \xi_{n+1/2}^{\pm,h} \Big)
- \nu^+ \big( \nabla {\color{red} \eta_{n+1/2}^{\pm} } , \nabla \xi_{n+1/2}^{\pm,h} \big)
            - \nu^- \big( \nabla {\color{red} \eta_{n+1/2}^{\mp} }, \nabla \xi_{n+1/2}^{\pm,h} \big) \notag 
\\
&
\quad
 + \nu^+ \big( \nabla z^{\pm}(t_{n+1/2}) - \nabla \color{red}z_{n+1/2}^{\pm}\normalcolor, \nabla \xi_{n+1/2}^{\pm,h} \big)
            + \nu^- \big( \nabla z^{\mp}(t_{n+1/2}) - \nabla \color{red}z_{n+1/2}^{\mp}\normalcolor, \nabla \xi_{n+1/2}^{\pm,h} \big) 
\notag 
\\
&
\quad
- \big( p^{\pm}(t_{n+1/2}) - q^{h}, \nabla \cdot \xi_{n+1/2}^{\pm,h} \big)
            - \Big( \frac{\color{red}z^{\pm}(t_{n+1})\normalcolor - \color{red}z^{\pm}(t_n)\normalcolor}{\tau_n} - z_{t}^{\pm}(t_{n+1/2}), \xi_{n+1/2}^{\pm,h} \Big) 
\notag \\
&
\quad
\pm \mathcal{N} \big(B_{\circ}, {\color{red} \eta_{n+1/2}^{\pm} } , \xi_{n+1/2}^{\pm,h} \big) 
            \pm \mathcal{N} \big( B_{\circ}, \color{red}z_{n+1/2}^{\pm}\normalcolor - z^{\pm}(t_{n+1/2}), \xi_{n+1/2}^{\pm,h} \big) 
\notag 
\\
&
\quad
+ \mathcal{N} \big( z^{\mp}(t_{n+1/2}), z^{\pm}(t_{n+1/2}), \xi_{n+1/2}^{\pm,h} \big) 
            - \mathcal{N} \big(\color{red} z_{n+1/2}^{\mp}\normalcolor, \color{red}z_{n+1/2}^{\pm}\normalcolor, \xi_{n+1/2}^{\pm,h} \big) 
\notag 
\\
&
\quad
+ \mathcal{N} \big( \color{red}z_{n+1/2}^{\mp}, z_{n+1/2}^{\pm}\normalcolor, \xi_{n+1/2}^{\pm,h} \big)
            - \mathcal{N} \big( z_{n+1/2}^{\mp,h}, z_{n+1/2}^{\pm,h}, \xi_{n+1/2}^{\pm,h} \big)
= I_1 +  \cdots + I_{11}. 
\notag 
\end{align}
Now we address each of the terms on the right hand side of \eqref{eq:errorL2-eq3}.
By the Cauchy-Schwarz
, Poincar\'e 
and Young's 
inequalities, 
the estimate of the Stokes projection 
\eqref{eq:Stoke-Approx}, 
the approximations in \eqref{eq:approx-thm}
        and H\"older's inequality 
we have
\begin{confidential}
    \color{darkblue}
        \begin{align*}
        - \Big( \frac{\eta_{n+1}^{\pm} - \eta_n^{\pm}}{\tau_n}, \xi_{n+1/2}^{\pm,h} \Big) 
        \leq& \frac{1}{\tau_n} \| \eta_{n+1}^{\pm} - \eta_n^{\pm} \| \| \xi_{n+1/2}^{\pm,h} \| 
        \leq \frac{C}{\tau_n \sqrt{\epsilon \nu^{\ast}}} \| \eta_{n+1}^{\pm} - \eta_n^{\pm} \| \sqrt{\epsilon \nu^{\ast}}\| \nabla \xi_{n+1/2}^{\pm,h} \| 
        \end{align*}
    \normalcolor
\end{confidential}
\begin{align}
\label{eq:errorL2-term1} 
I_1 
& 
= 
- \Big( \frac{\color{red}\eta^{\pm}(t_{n+1})\normalcolor - \color{red}\eta^{\pm}(t_n) \normalcolor}{\tau_n}, \xi_{n+1/2}^{\pm,h} \Big) 
\\
& \leq 
\frac{C (\epsilon)}{\tau_n^{2} \nu^{\ast}} \| \color{red} \eta^{\pm}(t_{n+1}) - \eta^{\pm}(t_n) \normalcolor\|^{2} + \epsilon \nu^{\ast} \| \nabla \xi_{n+1/2}^{\pm,h} \|^{2}  
\notag
\\
& 
\leq \frac{C (\epsilon) h^{2r+2}}{\tau_n^{2} \nu^{\ast}} \| \color{red}z^{\pm}(t_{n+1}) - z^{\pm}(t_n)\normalcolor \|_{r+1}^{2} + \epsilon \nu^{\ast} \| \nabla \xi_{n+1/2}^{\pm,h} \|^{2}, \notag 
\\
&
\leq \frac{C (\epsilon) h^{2r+2}}{\tau_n \nu^{\ast}} 
            \int_{t_{n}}^{t_{n+1}} \| z^{\pm} \|_{r+1}^{2} dt 
            + \epsilon \nu^{\ast} \| \nabla \xi_{n+1/2}^{\pm,h} \|^{2}, \notag 
\end{align}
where $\epsilon > 0$ is 
a constant to be 
chosen later, 
and $C(\epsilon) >0$ is a constant depending on $\epsilon$.
Using the Cauchy-Schwarz
, Young's inequality, 
the estimate 
\eqref{eq:Stoke-Approx}, the approximations 
\eqref{eq:approx-thm} and \eqref{eq:consistency-eq1} 
we also get
\begin{align}
I_2 
&
=
- \nu^+ \big( \nabla \color{red}\eta_{n+1/2}^{\pm}\normalcolor, \nabla \xi_{n+1/2}^{\pm,h} \big) 
            \label{eq:errorL2-term2} 
\leq  
C \nu^+ h^{r} \| \color{red}z_{n+1/2}^{\pm} \normalcolor\|_{r+1} \| \nabla \xi_{n+1/2}^{\pm,h} \|
\\
& 
\leq
 \frac{C(\epsilon)(\nu^{+})^{2} h^{2r}}{ \nu^{\ast}}
            \big( \| \color{red}z_{n+1/2}^{\pm}\normalcolor - z(t_{n+1/2}) \|_{r+1}^{2} + \| z(t_{n+1/2}) \|_{r+1}^{2} \big) + \epsilon \nu^{\ast} \| \nabla \xi_{n+1/2}^{\pm,h} \|^{2} 
\notag 
\\
&
\leq
 \frac{C(\epsilon)(\nu^{+})^{2} h^{2r}}{ \nu^{\ast}} \Big( \tau_{\rm{max}}^{3} \int_{t_{n}}^{t_{n+1}} \| z_{tt}^{\pm} \|_{r+1}^{2} dt + \| z^{\pm}(t_{n+1/2}) \|_{r+1}^{2} \Big) + \epsilon \nu^{\ast} \| \nabla \xi_{n+1/2}^{\pm,h} \|^{2}. \notag 
\end{align}
Similarly, 
\begin{align}
I_3 
&
= 
- \nu^- \big( \nabla \color{red}\eta_{n+1/2}^{\mp}\normalcolor, \nabla \xi_{n+1/2}^{\pm,h} \big) 
\label{eq:errorL2-term3} 
\\
&
\leq
 \frac{C(\epsilon)(\nu^{-})^{2} h^{2r}}{ \nu^{\ast}} \Big( \tau_{\rm{max}}^{3} \int_{t_{n}}^{t_{n+1}} \| z_{tt}^{\mp} \|_{r+1}^{2} dt + \| z^{\mp}(t_{n+1/2}) \|_{r+1}^{2} \Big) + \epsilon \nu^{\ast} \| \nabla \xi_{n+1/2}^{\pm,h} \|^{2}. 
 \notag 
\end{align}
By the Cauchy-Schwarz and Young inequalities, using \eqref{eq:consistency-eq1} 
we have
\begin{align}
I_4
&
= 
\nu^+ \big( \nabla z^{\pm}(t_{n+1/2}) - \nabla \color{red}z_{n+1/2}^{\pm}\normalcolor, \nabla \xi_{n+1/2}^{\pm,h} \big) 
\label{eq:errorL2-term4} 
\\
&
\leq
\nu^+ \| \nabla z^{\pm}(t_{n+1/2}) - \nabla \color{red}z_{n+1/2}^{\pm}\normalcolor \|
            \| \nabla \xi_{n+1/2}^{\pm,h} \| 
\notag 
\\
&
\leq
\frac{ C(\epsilon) (\nu^{+})^{2} \tau_{\rm{max}}^{3}}{\nu^{\ast}} 
            \int_{t_{n}}^{t_{n+1}} \| \nabla z_{tt}^{\pm} \|^{2} dt 
            + \epsilon \nu^{\ast} \| \nabla \xi_{n+1/2}^{\pm,h} \|^{2}
,
\notag 
\end{align}
and
\begin{align}
I_5 
&
= 
\nu^- \big( \nabla z^{\mp}(t_{n+1/2}) - \nabla \color{red}z_{n+1/2}^{\mp}\normalcolor, \nabla \xi_{n+1/2}^{\pm,h} \big) 
\label{eq:errorL2-term5}
\\
&
\leq
\nu^- \| \nabla z^{\mp}(t_{n+1/2}) - \nabla \color{red}z_{n+1/2}^{\mp}\normalcolor \| 
            \| \nabla \xi_{n+1/2}^{\pm,h} \| 
\notag
\\
&
\leq
\frac{ C(\epsilon) (\nu^{-})^{2} \tau_{\rm{max}}^{3}}{\nu^{\ast}} 
            \int_{t_{n}}^{t_{n+1}} \| \nabla z_{tt}^{\mp} \|^{2} dt 
            + \epsilon \nu^{\ast} \| \nabla \xi_{n+1/2}^{\pm,h} \|^{2}.
\notag 
\end{align}
Now we set $q^{h}$ to be the $L^{2}$-projection of $p^{\pm}(t_{n+1/2})$, 
and use the Cauchy-Schwarz inequality, 
\eqref{eq:approx-thm} and Young's inequality to obtain
\begin{align}
I_6 
&
= 
- \big( p^{\pm}(t_{n+1/2}) - q^{h}, \nabla \cdot \xi_{n+1/2}^{\pm,h} \big)
\leq \sqrt{d} \| p^{\pm}(t_{n+1/2}) - q^{h} \| \| \nabla \xi_{n+1/2}^{\pm,h} \| 
\label{eq:errorL2-term6} 
\\
& 
\leq
\frac{C(\epsilon) h^{2s+2}}{\nu^{\ast}} \| p^{\pm}(t_{n+1/2}) \|_{s+1}^{2}
            + \epsilon \nu^{\ast} \| \nabla \xi_{n+1/2}^{\pm,h} \|^{2}. \notag 
        \end{align}
Using the  Cauchy-Schwarz 
and Poincar\'e inequalit
ies, \eqref{eq:consistency-eq2} 
and Young's inequality we obtain
\begin{align}
\label{eq:errorL2-term7}  
I_7 
&
=  
- \Big( \frac{\color{red}z^{\pm}(t_{n+1}) \!-\! z^{\pm}(t_n) }{\tau_n} \!-\! z_{t}^{\pm}(t_{n+1/2}), \xi_{n+1/2}^{\pm,h} \Big) 
\\
&
\leq
 \frac{ C(\epsilon) \tau_{\rm{max}}^{3}}{\nu^{\ast}} \int_{t_{n}}^{t_{n+1}} \| \nabla z_{ttt}^{\mp} \|^{2} dt 
            + \epsilon \nu^{\ast} \| \nabla \xi_{n+1/2}^{\pm,h} \|^{2}. \notag 
        \end{align}
By Holder's inequality and the Sobolev embedding theorem, the Poincar\'e inequality, 
using 
$\nabla B_{\circ} = 0$, the approximations 
\eqref{eq:Stoke-Approx}, \eqref{eq:approx-thm} and \eqref{eq:consistency-eq1}, 
we have
\begin{confidential}
    \color{darkblue}
        \begin{align*}
            &\mathcal{N} \big(B_{\circ}, \eta_{n+1/2}^{\pm}, \xi_{n+1/2}^{\pm,h} \big) \\
            =& \big( (B_{\circ} \cdot \nabla ) \eta_{n+1/2}^{\pm}, \xi_{n+1/2}^{\pm,h} \big)
            + \frac{1}{2} \big( (\nabla \cdot B_{\circ}) \eta_{n+1/2}^{\pm}, \xi_{n+1/2}^{\pm,h} \big) \\
            =& \big( (B_{\circ} \cdot \nabla ) \eta_{n+1/2}^{\pm}, \xi_{n+1/2}^{\pm,h} \big) 
            \leq \| B_{\circ} \|_{L^{3}} \| \nabla \eta_{n+1/2}^{\pm} \| \| \xi_{n+1/2}^{\pm,h} \|_{L^{6}} \\
            \leq& C \big( \| B_{\circ} \| \| B_{\circ} \|_{1} \big)^{1/2} \| \nabla \eta_{n+1/2}^{\pm} \| \| \xi_{n+1/2}^{\pm,h} \|_{1} 
            \leq C \| B_{\circ} \| \| \nabla \eta_{n+1/2}^{\pm} \| \| \nabla \xi_{n+1/2}^{\pm,h} \|
        \end{align*}
    \normalcolor
\end{confidential}
\begin{align}
I_8 
= 
&
\pm \mathcal{N} \big(B_{\circ}, \color{red}\eta^{\pm}(t_{n+1/2})\normalcolor, \xi_{n+1/2}^{\pm,h} \big)  
\label{eq:errorL2-term8}   
\\
&
\leq C \big( \| B_{\circ} \| \| B_{\circ} \|_{1} \big)^{1/2} \| \nabla \color{red}\eta_{n+1/2}^{\pm}\normalcolor \| \| \xi_{n+1/2}^{\pm,h} \|_{1} \notag 
\\
&
\leq C \| B_{\circ} \| \| \nabla \color{red}\eta_{n+1/2}^{\pm}\normalcolor \| \| \nabla \xi_{n+1/2}^{\pm,h} \| \notag 
\\
& 
\leq Ch^{r} \| B_{\circ} \| \big( \| \color{red}z_{n+1/2}^{\pm}\normalcolor - z^{\pm}(t_{n+1/2}) \|_{r+1} + \| z^{\pm}(t_{n+1/2}) \|_{r+1} \big) \| \nabla \xi_{n+1/2}^{\pm,h} \| \notag 
\\
&
\leq \frac{C(\epsilon)h^{2r} }{\nu^{\ast}} \| B_{\circ} \|^{2}
            \big( \tau_{\rm{max}}^{3} \int_{t_{n}}^{t_{n+1}} \| z_{tt}^{\pm} \|_{r+1}^{2} dt + \| z^{\pm}(t_{n+1/2}) \|_{r+1}^{2} \big) 
            + \epsilon \nu^{\ast} \| \nabla \xi_{n+1/2}^{\pm,h} \|^{2}
, 
\notag 
\end{align}
and        
\begin{align}
I_9
= 
&
\pm \mathcal{N} \big( B_{\circ}, \color{red}z_{n+1/2}^{\pm}\normalcolor - z^{\pm}(t_{n+1/2}), \xi_{n+1/2}^{\pm,h} \big) 
\label{eq:errorL2-term9} 
\\
&
\leq C \| B_{\circ} \| \| \nabla\color{red} z_{n+1/2}^{\pm}\normalcolor - \nabla z^{\pm}(t_{n+1/2}) \| \| \nabla \xi_{n+1/2}^{\pm,h} \| 
\notag 
\\
 &
 \leq \frac{C(\epsilon) \tau_{\rm{max}}^{3}}{\nu^{\ast}} \| B_{\circ} \|^{2} 
            \int_{t_{n}}^{t_{n+1}} \| \nabla z_{tt}^{\pm} \|^{2} dt
            + \epsilon \nu^{\ast} \| \nabla \xi_{n+1/2}^{\pm,h} \|^{2}. \notag 
\end{align}
Similarly to \eqref{eq:errorL2-term9}
\begin{confidential}
    \color{darkblue}
        \begin{align*}
            &\mathcal{N} \big( z^{\mp}(t_{n+1\!/\!2}), z^{\pm}(t_{n+1\!/\!2}), \xi_{n+1\!/\!2}^{\pm,h} \big) 
            \!-\! \mathcal{N} \big( z_{n+1\!/\!2}^{\mp}, z_{n+1\!/\!2}^{\pm}, \xi_{n+1\!/\!2}^{\pm,h} \big) \\
            =& \mathcal{N} \big( z^{\mp}(t_{n+1\!/\!2}) - z_{n+1\!/\!2}^{\mp}, z^{\pm}(t_{n+1\!/\!2}), \xi_{n+1\!/\!2}^{\pm,h} \big) 
            + \mathcal{N} \big( z_{n+1\!/\!2}^{\mp}, z^{\pm}(t_{n+1\!/\!2}), \xi_{n+1\!/\!2}^{\pm,h} \big) 
            \!-\! \mathcal{N} \big( z_{n+1\!/\!2}^{\mp}, z_{n+1\!/\!2}^{\pm}, \xi_{n+1\!/\!2}^{\pm,h} \big) 
        \end{align*}
    \normalcolor
\end{confidential}
we 
obtain
\begin{align}
I_{10}
&
=
\mathcal{N} \big( z^{\mp}(t_{n+1\!/\!2}), z^{\pm}(t_{n+1\!/\!2}), \xi_{n+1\!/\!2}^{\pm,h} \big) 
\!-\! \mathcal{N} \big( z_{n+1\!/\!2}^{\mp}, z_{n+1\!/\!2}^{\pm}, \xi_{n+1\!/\!2}^{\pm,h} \big) 
            \label{eq:errorL2-term10-11} 
\\
& 
\leq
\mathcal{N} \big( z^{\mp}(t_{n+1\!/\!2}) \!-\! z_{n+1\!/\!2}^{\mp}, z^{\pm}(t_{n+1\!/\!2}), \xi_{n+1\!/\!2}^{\pm,h} \big) 
            + \mathcal{N} \big( \color{red}z_{n+1\!/\!2}^{\mp}\normalcolor , z^{\pm}(t_{n+1\!/\!2}) \!-\! \color{red}z_{n+1\!/\!2}^{\pm}\normalcolor, \xi_{n+1\!/\!2}^{\pm,h} \big) 
\notag 
\\
&
\leq 
C \| \nabla ( z^{\mp}(t_{n+1\!/\!2}) \!-\! \color{red}z_{n+1\!/\!2}^{\mp}\normalcolor ) \| 
            \| \nabla z^{\pm}(t_{n+1\!/\!2}) \| \| \nabla \xi_{n+1\!/\!2}^{\pm,h} \| 
\notag 
\\
& \quad
+ C \| \nabla z_{n+1\!/\!2}^{\mp} \| \| \nabla ( z^{\pm}(t_{n+1\!/\!2}) \!-\! \color{red}z_{n+1\!/\!2}^{\pm} )\normalcolor \| \| \nabla \xi_{n+1\!/\!2}^{\pm,h} \| \notag 
\\
& 
\leq \frac{C(\epsilon) \tau_{\rm{max}}^{3} }{\nu^{\ast}} \Big( \| |z^{\mp}| \|_{\infty,1}^{2} \int_{t_{n}}^{t_{n + 1}}   \| \nabla z_{tt}^{\pm} \|^{2} dt 
 + \| |z^{\pm}| \|_{\infty,1,1/2}^{2} \int_{t_{n}}^{t_{n\!+\!1}}  \| \nabla z_{tt}^{\mp} \|^{2} dt \Big) 
\notag
\\
& 
\qquad
+ \epsilon \nu^{\ast} \| \nabla \xi_{n+1\!/\!2}^{\pm,h} \|^{2}. \notag 
        \end{align}
By the skew-symmetriy property 
\eqref{eq:N-form} we can write  $I_{11}$ as
\begin{confidential}
    \color{darkblue}
        \begin{align*}
            &\mathcal{N} \big( \color{red}z_{n+1/2}^{\mp}\normalcolor, z_{n+1/2}^{\pm}, \xi_{n+1/2}^{\pm,h} \big)
            - \mathcal{N} \big( z_{n+1/2}^{\mp,h}, z_{n+1/2}^{\pm,h}, \xi_{n+1/2}^{\pm,h} \big) \\
            =& \mathcal{N} \big( z_{n+1/2}^{\mp}, z_{n+1/2}^{\pm}, \xi_{n+1/2}^{\pm,h} \big)
            - \mathcal{N} \big( z_{n+1/2}^{\mp,h}, z_{n+1/2}^{\pm}, \xi_{n+1/2}^{\pm,h} \big)
            + \mathcal{N} \big( z_{n+1/2}^{\mp,h}, z_{n+1/2}^{\pm}, \xi_{n+1/2}^{\pm,h} \big)
            \notag \\
            &- \mathcal{N} \big( z_{n+1/2}^{\mp,h}, z_{n+1/2}^{\pm,h}, \xi_{n+1/2}^{\pm,h} \big) \notag \\
            =& \mathcal{N} \big( - e_{n+1/2}^{\mp}, z_{n+1/2}^{\pm}, \xi_{n+1/2}^{\pm,h} \big)
            + \mathcal{N} \big( z_{n+1/2}^{\mp,h}, - e_{n+1/2}^{\pm}, \xi_{n+1/2}^{\pm,h} \big) \notag \\
            =& - \mathcal{N} \big( \xi_{n+1/2}^{\mp,h} + \eta_{n+1/2}^{\mp}, z_{n+1/2}^{\pm}, \xi_{n+1/2}^{\pm,h} \big)
            - \mathcal{N} \big( z_{n+1/2}^{\mp,h}, \xi_{n+1/2}^{\pm,h} + \eta_{n+1/2}^{\pm}, \xi_{n+1/2}^{\pm,h} \big) \notag \\
            =& - \mathcal{N} \big( \xi_{n+1/2}^{\mp,h}, z_{n+1/2}^{\pm}, \xi_{n+1/2}^{\pm,h} \big)
            - \mathcal{N} \big( \eta_{n+1/2}^{\mp}, z_{n+1/2}^{\pm}, \xi_{n+1/2}^{\pm,h} \big)
            - \mathcal{N} \big( z_{n+1/2}^{\mp,h}, \eta_{n+1/2}^{\pm}, \xi_{n+1/2}^{\pm,h} \big)\\
            =& - \mathcal{N} \big( \xi_{n+1/2}^{\mp,h}, z_{n+1/2}^{\pm}, \xi_{n+1/2}^{\pm,h} \big)
            - \mathcal{N} \big( \eta_{n+1/2}^{\mp}, z_{n+1/2}^{\pm}, \xi_{n+1/2}^{\pm,h} \big)
            - \mathcal{N} \big( z_{n+1/2}^{\mp,h} - z_{n+1/2}^{\mp}, \eta_{n+1/2}^{\pm}, \xi_{n+1/2}^{\pm,h} \big) \notag \\
            &- \mathcal{N} \big( z_{n+1/2}^{\mp}, \eta_{n+1/2}^{\pm}, \xi_{n+1/2}^{\pm,h} \big)
            \notag \\
            =& - \mathcal{N} \big( \xi_{n+1/2}^{\mp,h}, z_{n+1/2}^{\pm}, \xi_{n+1/2}^{\pm,h} \big)
            - \mathcal{N} \big( \eta_{n+1/2}^{\mp}, z_{n+1/2}^{\pm}, \xi_{n+1/2}^{\pm,h} \big)
            - \mathcal{N} \big( e_{n+1/2}^{\mp}, \eta_{n+1/2}^{\pm}, \xi_{n+1/2}^{\pm,h} \big) \notag \\
            &- \mathcal{N} \big( z_{n+1/2}^{\mp}, \eta_{n+1/2}^{\pm}, \xi_{n+1/2}^{\pm,h} \big)
            \notag \\
            =& - \mathcal{N} \big( \xi_{n+1/2}^{\mp,h}, z_{n+1/2}^{\pm}, \xi_{n+1/2}^{\pm,h} \big)
            - \mathcal{N} \big( \eta_{n+1/2}^{\mp}, z_{n+1/2}^{\pm}, \xi_{n+1/2}^{\pm,h} \big)
            - \mathcal{N} \big( \xi_{n+1/2}^{\mp,h} + \eta_{n+1/2}^{\mp}, \eta_{n+1/2}^{\pm}, \xi_{n+1/2}^{\pm,h} \big) \notag \\
            &- \mathcal{N} \big( z_{n+1/2}^{\mp}, \eta_{n+1/2}^{\pm}, \xi_{n+1/2}^{\pm,h} \big)
            \notag \\
            =& - \mathcal{N} \big( \xi_{n+1/2}^{\mp,h}, z_{n+1/2}^{\pm}, \xi_{n+1/2}^{\pm,h} \big)
            - \mathcal{N} \big( \eta_{n+1/2}^{\mp}, z_{n+1/2}^{\pm}, \xi_{n+1/2}^{\pm,h} \big)
            - \mathcal{N} \big( \xi_{n+1/2}^{\mp,h}, \eta_{n+1/2}^{\pm}, \xi_{n+1/2}^{\pm,h} \big) \notag \\ 
            & - \mathcal{N} \big( \eta_{n+1/2}^{\mp}, \eta_{n+1/2}^{\pm}, \xi_{n+1/2}^{\pm,h} \big)
            - \mathcal{N} \big( z_{n+1/2}^{\mp}, \eta_{n+1/2}^{\pm}, \xi_{n+1/2}^{\pm,h} \big).
        \end{align*}
    \normalcolor
\end{confidential}
\begin{align*}
I_{11}
& 
=
\mathcal{N} \big( \color{red}z_{n+1\!/\!2}^{\mp}\normalcolor, \color{red}z_{n+1\!/\!2}^{\pm}\normalcolor, \xi_{n+1\!/\!2}^{\pm,h} \big)
            - \mathcal{N} \big( z_{n+1\!/\!2}^{\mp,h}, z_{n+1\!/\!2}^{\pm,h}, \xi_{n+1\!/\!2}^{\pm,h} \big) 
\\
&
= - \mathcal{N} \big( \xi_{n+1\!/\!2}^{\mp,h}, \color{red}z_{n+1\!/\!2}^{\pm}\normalcolor, \xi_{n+1\!/\!2}^{\pm,h} \big)
            - \mathcal{N} \big( \color{red}\eta_{n+1\!/\!2}^{\mp}\normalcolor, \color{red}z_{n+1\!/\!2}^{\pm}\normalcolor, \xi_{n+1\!/\!2}^{\pm,h} \big)
            - \mathcal{N} \big( \xi_{n+1\!/\!2}^{\mp,h}, \color{red}\eta_{n+1\!/\!2}^{\pm}\normalcolor, \xi_{n+1\!/\!2}^{\pm,h} \big) 
\notag 
\\ 
& 
\qquad
- \mathcal{N} \big( \color{red}\eta_{n+1\!/\!2}^{\mp}\normalcolor, \color{red}\eta_{n+1\!/\!2}^{\pm}\normalcolor, \xi_{n+1\!/\!2}^{\pm,h} \big)
            - \mathcal{N} \big(\color{red} z_{n+1\!/\!2}^{\mp}\normalcolor, \color{red}\eta_{n+1\!/\!2}^{\pm}\normalcolor, \xi_{n+1\!/\!2}^{\pm,h} \big)
= I_{11}^{(1)} + \cdots +  I_{11}^{(5)}
.
\notag 
\end{align*}
Now we evaluate each term in the right hand side above as follows.
By H\"older's inequality, \eqref{eq:Ladyzh-L3}, 
        Poincar\'e 
and Young's inequality
we have
\begin{confidential}
    \color{darkblue}
        \begin{align*}
            &-\mathcal{N} \big( \xi_{n+1/2}^{\mp,h}, z_{n+1/2}^{\pm}, \xi_{n+1/2}^{\pm,h} \big)
            \\
            =& - \frac{1}{2} \big( (\xi_{n+1/2}^{\mp,h} \cdot \nabla ) z_{n+1/2}^{\pm}, \xi_{n+1/2}^{\pm,h} \big)
            - \frac{1}{2} \big( (\xi_{n+1/2}^{\mp,h} \cdot \nabla ) \xi_{n+1/2}^{\pm,h}, z_{n+1/2}^{\pm} \big) \\
            \leq & \frac{1}{2} \| \xi_{n+1/2}^{\mp,h} \|_{L^{4}} \| \nabla z_{n+1/2}^{\pm} \| \| \xi_{n+1/2}^{\pm,h} \|_{L^{4}}
            + \frac{1}{2} \| \xi_{n+1/2}^{\mp,h} \|_{L^{4}} \| \nabla \xi_{n+1/2}^{\pm,h} \|
            \| z_{n+1/2}^{\pm} \|_{L^{4}} \\
            \leq & \frac{1}{2} \Big( \Big( \frac{ 2(d-1) }{d^{3/2}} \Big)^{d/4} 
            \| \xi_{n+1/2}^{\mp,h} \|^{1-d/4} \| \nabla \xi_{n+1/2}^{\mp,h} \|^{d/4} \Big) \| \nabla z_{n+1/2}^{\pm} \|
            \Big( \Big( \frac{ 2(d-1) }{d^{3/2}} \Big)^{d/4} 
            \| \xi_{n+1/2}^{\pm,h} \|^{1-d/4} \| \nabla \xi_{n+1/2}^{\pm,h} \|^{d/4} \Big) \\
            &+ \frac{1}{2} \Big( \Big( \frac{ 2(d-1) }{d^{3/2}} \Big)^{d/4} 
            \| \xi_{n+1/2}^{\mp,h} \|^{1-d/4} \| \nabla \xi_{n+1/2}^{\mp,h} \|^{d/4} \Big) \| 
            \| \nabla \xi_{n+1/2}^{\pm,h} \| 
            \Big( \Big( \frac{ 2(d-1) }{d^{3/2}} \Big)^{d/4} 
            \| z_{n+1/2}^{\pm} \|^{1-d/4} \| \nabla z_{n+1/2}^{\pm} \|^{d/4} \Big)
            \\
            \leq& \frac{1}{4} \Big( \frac{ 2(d-1) }{d^{3/2}} \Big)^{d/2} 
            \| \nabla z_{n+1/2}^{\pm} \| \Big[ \| \xi_{n+1/2}^{\mp,h} \|^{2(1-d/4)} 
            (4 \epsilon \nu^{\ast}/d)^{-d/4} (4 \epsilon \nu^{\ast}/d)^{d/4} \| \nabla \xi_{n+1/2}^{\mp,h} \|^{d/2} \\
            &\qquad \qquad \qquad \qquad \qquad \qquad + \| \xi_{n+1/2}^{\pm,h} \|^{2(1-d/4)} 
            (4 \epsilon \nu^{\ast}/d)^{-d/4} (4 \epsilon \nu^{\ast}/d)^{d/4}
            \| \nabla \xi_{n+1/2}^{\pm,h} \|^{d/2} \Big] \\
            &+ \frac{1}{2} \Big( \frac{ 2(d-1) }{d^{3/2}} \Big)^{d/2} 
            (2 \epsilon \nu^{\ast})^{-1/2} (2 \epsilon \nu^{\ast})^{1/2}
            \| \nabla \xi_{n+1/2}^{\pm,h} \| \| \xi_{n+1/2}^{\mp,h} \|^{1-d/4} \| \nabla \xi_{n+1/2}^{\mp,h} \|^{d/4} \| z_{n+1/2}^{\pm} \|^{1-d/4} \| \nabla z_{n+1/2}^{\pm} \|^{d/4} \\
            \leq& \frac{1}{4} \frac{ 2^{d/4} (d-1)^{d/2} }{d^{3d/4}} 
            \frac{d^{d/4}}{2^{d/2} (\epsilon \nu^{\ast})^{d/4}} \| \nabla z_{n+1/2}^{\pm} \|
            \Big[ \| \xi_{n+1/2}^{\mp,h} \|^{2(1-d/4)} 
            (4 \epsilon \nu^{\ast}/d)^{d/4} \| \nabla \xi_{n+1/2}^{\mp,h} \|^{d/2} \\
            &\qquad \qquad \qquad \qquad \qquad \qquad \qquad \qquad \qquad
            + \| \xi_{n+1/2}^{\pm,h} \|^{2(1-d/4)} 
            (4 \epsilon \nu^{\ast}/d)^{d/4} \| \nabla \xi_{n+1/2}^{\pm,h} \|^{d/2} \Big] \\
            &+ \frac{(d-1)^{d/2}}{2^{(3 - d)/2} d^{3d/4} (\epsilon \nu^{\ast})^{1/2} } 
            (2 \epsilon \nu^{\ast})^{1/2}
            \| \nabla \xi_{n+1/2}^{\pm,h} \| \| \xi_{n+1/2}^{\mp,h} \|^{1-d/4} \| \nabla \xi_{n+1/2}^{\mp,h} \|^{d/4} \| z_{n+1/2}^{\pm} \|^{1-d/4} \| \nabla z_{n+1/2}^{\pm} \|^{d/4} \\
            \leq& \frac{(d-1)^{d/2}}{2^{(8+d)/4} d^{d/2} (\epsilon \nu^{\ast})^{d/4}}
            \| \nabla z_{n+1/2}^{\pm} \|
            \Big[ \| \xi_{n+1/2}^{\mp,h} \|^{2(1-d/4)} 
            (4 \epsilon \nu^{\ast}/d)^{d/4} \| \nabla \xi_{n+1/2}^{\mp,h} \|^{d/2} \\
            &\qquad \qquad \qquad \qquad \qquad \qquad \qquad \qquad \qquad
            + \| \xi_{n+1/2}^{\pm,h} \|^{2(1-d/4)} 
            (4 \epsilon \nu^{\ast}/d)^{d/4} \| \nabla \xi_{n+1/2}^{\pm,h} \|^{d/2} \Big] \\
            &+ \frac{(d-1)^{d/2}}{2^{(3 - d)/2} d^{3d/4} (\epsilon \nu^{\ast})^{1/2} } 
            (2 \epsilon \nu^{\ast})^{1/2}
            \| \nabla \xi_{n+1/2}^{\pm,h} \| \| \xi_{n+1/2}^{\mp,h} \|^{1-d/4} \| \nabla \xi_{n+1/2}^{\mp,h} \|^{d/4} \| z_{n+1/2}^{\pm} \|^{1-d/4} \| \nabla z_{n+1/2}^{\pm} \|^{d/4} \\
            \leq& \epsilon \nu^{\ast} \| \nabla \xi_{n+1/2}^{\mp,h} \|^{2}
            + \epsilon \nu^{\ast} \| \nabla \xi_{n+1/2}^{\pm,h} \|^{2}
            + \epsilon \nu^{\ast} \| \nabla \xi_{n+1/2}^{\pm,h} \|^{2} \\
            &+ \frac{4-d}{4} \frac{(d-1)^{2d/(4-d)}}{2^{(8+d)/(4-d)} d^{2d/(4-d)} (\epsilon \nu^{\ast})^{d/(4-d)}} \| \nabla z_{n+1/2}^{\pm} \|^{4/(4-d)} \Big( \| \xi_{n+1/2}^{\mp,h} \|^{2} + \| \xi_{n+1/2}^{\pm,h} \|^{2} \Big) \\
            &+ \frac{1}{2}\frac{(d-1)^{d}}{2^{3 - d} d^{3d/2} (\epsilon \nu^{\ast}) } 
            \| \xi_{n+1/2}^{\mp,h} \|^{2(1-d/4)} (4\epsilon \nu^{\ast}/d)^{-d/4} (4\epsilon \nu^{\ast}/d)^{d/4} \| \nabla \xi_{n+1/2}^{\mp,h} \|^{d/2} \| z_{n+1/2}^{\pm} \|^{2(1-d/4)} \| \nabla z_{n+1/2}^{\pm} \|^{d/2} \\
            \leq& \epsilon \nu^{\ast} \| \nabla \xi_{n+1/2}^{\mp,h} \|^{2} 
            + 2\epsilon \nu^{\ast} \| \nabla \xi_{n+1/2}^{\pm,h} \|^{2}
            + \frac{(4-d)(d-1)^{2d/(4-d)}}{2^{(16-d)/(4-d)} d^{2d/(4-d)} (\epsilon \nu^{\ast})^{d/(4-d)}} \| \nabla z_{n+1/2}^{\pm} \|^{4/(4-d)} \Big( \| \xi_{n+1/2}^{\mp,h} \|^{2} + \| \xi_{n+1/2}^{\pm,h} \|^{2} \Big) \\ 
            &+ \epsilon \nu^{\ast} \| \nabla \xi_{n+1/2}^{\mp,h} \|^{2}
            + \frac{4-d}{4} \Big[\frac{(d-1)^{d}}{2^{4 - d} d^{3d/2} (\epsilon \nu^{\ast}) } 
            \frac{d^{d/4}}{ 2^{d/2} (\epsilon \nu^{\ast})^{d/4}} \Big]^{4/(4-d)}
            \| \xi_{n+1/2}^{\mp,h} \|^{2} \| z_{n+1/2}^{\pm} \|^{2} 
            \| \nabla z_{n+1/2}^{\pm} \|
        \end{align*}
        The above calculation doesn't work! 
        \begin{align*}
            &-\mathcal{N} \big( \xi_{n+1/2}^{\mp,h}, z_{n+1/2}^{\pm}, \xi_{n+1/2}^{\pm,h} \big)
            \\
            =& - \frac{1}{2} \big( (\xi_{n+1/2}^{\mp,h} \cdot \nabla ) z_{n+1/2}^{\pm}, \xi_{n+1/2}^{\pm,h} \big)
            - \frac{1}{2} \big( (\xi_{n+1/2}^{\mp,h} \cdot \nabla ) \xi_{n+1/2}^{\pm,h}, z_{n+1/2}^{\pm} \big) \\
            \leq& \frac{1}{2} \| \xi_{n+1/2}^{\mp,h} \|_{L^{3}} \| \nabla z_{n+1/2}^{\pm} \| \| \xi_{n+1/2}^{\pm,h} \|_{L^{6}} 
            + \frac{1}{2} \| \xi_{n+1/2}^{\mp,h} \|_{L^{3}} \| \nabla \xi_{n+1/2}^{\pm,h} \|
            \| z_{n+1/2}^{\pm} \|_{L^{6}} \\
            \leq& \frac{1}{2} \Big( \big( \frac{ 1}{d^{1/2}} \big)^{d/6} 
            \| \xi_{n+1/2}^{\mp,h} \|^{1-d/6} \| \nabla \xi_{n+1/2}^{\mp,h} \|^{d/6}  \Big) \| \nabla z_{n+1/2}^{\pm} \|
            \Big( \big( \frac{ 3(d-1) }{d^{3/2}} \big)^{d/3}
            \| \xi_{n+1/2}^{\pm,h} \|^{1-d/3} \| \nabla \xi_{n+1/2}^{\pm,h} \|^{d/3} \Big) \\
            &+ \frac{1}{2} \Big( \big( \frac{ 1}{d^{1/2}} \big)^{d/6} 
            \| \xi_{n+1/2}^{\mp,h} \|^{1-d/6} \| \nabla \xi_{n+1/2}^{\mp,h} \|^{d/6}  \Big) 
            \| \nabla \xi_{n+1/2}^{\pm,h} \| 
            \Big( \big( \frac{ 3(d-1) }{d^{3/2}} \big)^{d/3}
            \| z_{n+1/2}^{\pm} \|^{1-d/3} \| \nabla z_{n+1/2}^{\pm} \|^{d/3} \Big) \\
            \leq& \frac{(3(d-1))^{d/3} }{2 d^{7d/12}} C_{P}^{1/2 - d/6} 
            \| \xi_{n+1/2}^{\mp,h} \|^{1/2} \| \nabla \xi_{n+1/2}^{\mp,h} \|^{1/2}
            \| \nabla z_{n+1/2}^{\pm} \| C_{P}^{1-d/3} \| \nabla \xi_{n+1/2}^{\pm,h} \| \\
            &+ \frac{(3(d-1))^{d/3} }{2 d^{7d/12}} C_{P}^{1/2 - d/6} 
            \| \xi_{n+1/2}^{\mp,h} \|^{1/2} \| \nabla \xi_{n+1/2}^{\mp,h} \|^{1/2} 
            \| \nabla \xi_{n+1/2}^{\pm,h} \| C_{P}^{1-d/3} \| \nabla z_{n+1/2}^{\pm} \| \\
            \leq& \frac{(3(d-1))^{d/3} }{d^{7d/12}} C_{P}^{(3-d)/2} 
            \| \xi_{n+1/2}^{\mp,h} \|^{1/2} \| \nabla \xi_{n+1/2}^{\mp,h} \|^{1/2} 
            \| \nabla z_{n+1/2}^{\pm} \| (2 \epsilon \nu^{\ast})^{-1/2}  (2 \epsilon \nu^{\ast})^{1/2} \| \nabla \xi_{n+1/2}^{\pm,h} \| \\
            \leq& \epsilon \nu^{\ast} \| \nabla \xi_{n+1/2}^{\pm,h} \|^{2}
            + \frac{(3(d-1))^{2d/3} }{4 \epsilon \nu^{\ast} d^{7d/6}} C_{P}^{3-d}  
            \| \xi_{n+1/2}^{\mp,h} \| (2 \epsilon \nu^{\ast})^{-1/2}  (2 \epsilon \nu^{\ast})^{1/2} \| \nabla \xi_{n+1/2}^{\mp,h} \| \| \nabla z_{n+1/2}^{\pm} \|^{2} \\
            \leq& \epsilon \nu^{\ast} \| \nabla \xi_{n+1/2}^{\pm,h} \|^{2} 
            + \epsilon \nu^{\ast} \| \nabla \xi_{n+1/2}^{\mp,h} \|^{2}
            + \frac{1}{2} \Big[ \frac{(3(d-1))^{2d/3} C_{P}^{3-d}}{4\sqrt{2} (\epsilon \nu^{\ast})^{3/2} d^{7d/6}} \Big]^{2} \| \xi_{n+1/2}^{\mp,h} \|^{2} 
            \| \nabla z_{n+1/2}^{\pm} \|^{4} \\
            =& \epsilon \nu^{\ast} \| \nabla \xi_{n+1/2}^{\pm,h} \|^{2} 
            + \epsilon \nu^{\ast} \| \nabla \xi_{n+1/2}^{\mp,h} \|^{2}
            + \frac{(3(d-1))^{4d/3} C_{P}^{2(3-d)}}{64 (\epsilon \nu^{\ast})^{3} d^{7d/3}}
            \| \xi_{n+1/2}^{\mp,h} \|^{2} \| \nabla z_{n+1/2}^{\pm} \|^{4},
        \end{align*}
        where $C_{P}$ is the bound coefficient of Poincar\'e inequality.
    \normalcolor
\end{confidential}
\begin{align}
& 
I_{11}^{(1)}
=
- \mathcal{N} \big( \xi_{n+1/2}^{\mp,h}, {\color{red} z_{n+1/2}^{\pm} } , \xi_{n+1/2}^{\pm,h} \big)
            \label{eq:NL-term1}
\\
& 
\leq \epsilon \nu^{\ast} \| \nabla \xi_{n+1/2}^{\pm,h} \|^{2} 
            + \epsilon \nu^{\ast} \| \nabla \xi_{n+1/2}^{\mp,h} \|^{2}
            + \frac{(3(d-1))^{4d/3} C_{P}^{2(3-d)}}{64 (\epsilon \nu^{\ast})^{3} d^{7d/3}}
            \| \nabla {\color{red} z_{n+1/2}^{\pm} } \|^{4} \| \xi_{n+1/2}^{\mp,h} \|^{2} . 
\notag 
        \end{align}
Similarly, 
        \begin{align}
& 
I_{11}^{(2)}
=
 - \mathcal{N} \big( {\color{red}\eta_{n+1/2}^{\mp} } , {\color{red} z_{n+1/2}^{\pm} } , \xi_{n+1/2}^{\pm,h} \big) \!
\label{eq:NL-term2}   
\\
& 
\leq
 \epsilon \nu^{\ast} \| \nabla \xi_{n+1/2}^{\pm,h} \|^{2} 
            \!+\! \frac{C(\epsilon)}{\nu^{\ast}} \Big( \| \nabla {\color{red}  z_{n+1}^{\pm} } \|^{2} +  \| {\color{red}  \nabla z_{n}^{\pm} } \|^{2} \Big)\! \| \nabla {\color{red}  \eta_{n+1/2}^{\mp} } \|^{2}. 
\notag 
        \end{align}
Using  H\"older's inequality, \eqref{eq:Ladyzh-L3}, 
Poincar\'e 
and Young's inequality, 
\eqref{eq:Stoke-Approx}, \eqref{eq:approx-thm} and the inverse inequality 
\eqref{eq:inv-inequal}
we get
\begin{align}
 & 
 I_{11}^{(3)}
 = 
- \mathcal{N} \big( \xi_{n+1/2}^{\mp,h}, {\color{red} \eta_{n+1/2}^{\pm} }, \xi_{n+1/2}^{\pm,h} \big) 
\label{eq:NL-term3} 
\\
& 
\leq
 \frac{(3(d-1))^{d/3} }{d^{7d/12}} C_{P}^{(3-d)/2} \big( \| \xi_{n+1/2}^{\mp,h} \| \| \nabla \xi_{n+1/2}^{\mp,h} \| \big)^{1/2}
            \| \nabla {\color{red} \eta_{n+1/2}^{\pm} } \| \| \nabla \xi_{n+1/2}^{\pm,h} \| 
\notag 
\\
& 
\leq
 \epsilon \nu^{\ast} \| \nabla \xi_{n+1/2}^{\pm,h} \|^{2} 
+ \frac{(3(d-1))^{2d/3}C_{P}^{3-d} C_{I} (C_{A}^{1,0})^{2} }{\epsilon \nu^{\ast} d^{7d/6}} \big(1 + \frac{1}{\beta_{\tt{is}}} \big)^{2} h | {\color{red} z_{n+1/2}^{\pm} } |_{2}
            \| \xi_{n+1/2}^{\mp,h} \|^{2}, 
\notag 
\end{align}
also
\begin{align}
&
 I_{11}^{(4)} 
= 
- \mathcal{N} \big( {\color{red}  \eta_{n+1/2}^{\mp} }, {\color{red}  \eta_{n+1/2}^{\pm} }, \xi_{n+1/2}^{\pm,h} \big) 
\label{eq:NL-term4} 
\leq C \big( \| \nabla {\color{red} \eta_{n+1}^{\mp} } \| + \| \nabla {\color{red}  \eta_{n}^{\mp} } \| \big) \| \nabla {\color{red}  \eta_{n+1/2}^{\pm} } \| \| \nabla  \xi_{n+1/2}^{\pm,h} \| 
\\
& 
\leq \epsilon \nu^{\ast} \| \nabla \xi_{n+1/2}^{\pm,h} \|^{2}
            + \frac{C(\epsilon)}{\nu^{\ast}} \big( \| \nabla {\color{red}  z_{n+1}^{\mp} } \|^{2} 
            + \| \nabla {\color{red}  z_{n}^{\mp} } \|^{2} \big) \| \nabla {\color{red}  \eta_{n+1/2}^{\pm} } \|^{2}, 
\notag 
\end{align}
and finally
\begin{align}
&
I_{11}^{(5)}
=
- \mathcal{N} \big( {\color{red}  z_{n+1/2}^{\mp} } , {\color{red}  \eta_{n+1/2}^{\pm} } , \xi_{n+1/2}^{\pm,h} \big)
\label{eq:NL-term5} 
\leq C \big( \| {\color{red} \nabla z_{n+1}^{\mp} } \| + \| \nabla {\color{red} z_{n}^{\mp} } \| \big)
            \| \nabla {\color{red}  \eta_{n+1/2}^{\pm} } \| \| \nabla \xi_{n+1/2}^{\pm,h} \| 
\\
& 
\leq \epsilon \nu^{\ast} \| \nabla \xi_{n+1/2}^{\pm,h} \|^{2}
            + \frac{C(\epsilon)}{\nu^{\ast}} \big( \| \nabla {\color{red}  z_{n+1}^{\mp} } \|^{2} 
            + \| \nabla {\color{red}  z_{n}^{\mp} } \|^{2} \big) 
            \| \nabla {\color{red}  \eta_{n+1/2}^{\pm} } \|^{2}. \notag  
        \end{align}
Therefore, we combine \eqref{eq:NL-term1}
-\eqref{eq:NL-term5} to obtain 
\begin{align}
I_{11}
& 
=
\mathcal{N} \big( {\color{red} z_{n+1\!/\!2}^{\mp} }, {\color{red}  z_{n+1\!/\!2}^{\pm} }, \xi_{n+1\!/\!2}^{\pm,h} \big)
            - \mathcal{N} \big( {\color{red} z_{n+1\!/\!2}^{\mp,h} }, {\color{red} z_{n+1\!/\!2}^{\pm,h} }, \xi_{n+1\!/\!2}^{\pm,h} \big) 
\label{eq:errorL2-term12-13-eq1}  
\\
&
\leq
\frac{(3(d\!-\!1))^{2d/3}C_{P}^{3-d} C_{I} (C_{A}^{1,0})^{2} }{\epsilon \nu^{\ast} d^{7d/6}} \big(1 + \frac{1}{\beta_{\tt{is}}} \big)^{2} h | {\color{red} z_{n+1/2}^{\pm}} |_{2} \| \xi_{n+1/2}^{\mp,h} \|^{2} 
\notag 
\\
& 
\quad 
+ \frac{(3(d\!-\!1))^{4d/3} C_{P}^{2(3-d)}}{64 (\epsilon \nu^{\ast})^{3} d^{7d/3}} \| \nabla {\color{red} z_{n+1/2}^{\pm} } \|^{4} \| \xi_{n+1/2}^{\mp,h} \|^{2} 
\notag
\\
& 
\quad
+ \frac{C(\epsilon)}{\nu^{\ast}} \Big( \| \nabla {\color{red} z_{n+1}^{\pm} } \|^{2} \!+\! \| \nabla {\color{red} z_{n}^{\pm} } \|^{2} \Big)
             \| \nabla {\color{red} \eta_{n+1/2}^{\mp} } \|^{2}
\notag 
\\
& 
\quad
+ \frac{C(\epsilon)}{\nu^{\ast}} \big( \| \nabla {\color{red} z_{n+1}^{\mp} } \|^{2} 
+ \| \nabla {\color{red} z_{n}^{\mp} } \|^{2} \big) \| \nabla {\color{red} \eta_{n+1/2}^{\pm} } \|^{2}
            + 5 \epsilon \nu^{\ast} \| \nabla \xi_{n+1/2}^{\pm,h} \|^{2}
            + \epsilon \nu^{\ast} \| \nabla \xi_{n+1/2}^{\mp,h} \|^{2}
. 
\notag 
        \end{align}
%
Now we use \eqref{eq:Stoke-Approx}, \eqref{eq:approx-thm}, the triangle inequality and \eqref{eq:consistency-eq1} 
to evaluate 
\begin{align}
& 
\| \nabla {\color{red} \eta_{n+1\!/\!2}^{\pm} } \|^{2}
\leq  
Ch^{2r} \| {\color{red} z_{n+1/2}^{\pm} } \|_{r+1}^{2}  
\notag
\leq 
C h^{2r} \tau_{\rm{max}}^{3}  \int_{t_{n}}^{t_{n+1}} \| z^{\pm}_{tt} \|_{r+1}^{2} dt + Ch^{2r} \| z^{\pm}(t_{n+1/2}) \|_{r+1}^{2}
, 
\notag  
\notag  
\end{align}
%
to get
\begin{align}
& 
I_{11}
= 
\mathcal{N} \big( {\color{red} z_{n+1\!/\!2}^{\mp} }, {\color{red} z_{n+1\!/\!2}^{\pm} }, \xi_{n+1\!/\!2}^{\pm,h} \big)
            - \mathcal{N} \big( {\color{red} z_{n+1\!/\!2}^{\mp,h} } , {\color{red}  z_{n+1\!/\!2}^{\pm,h} } , \xi_{n+1\!/\!2}^{\pm,h} \big) 
\label{eq:errorL2-term12-13}  
\\
&
\leq
\frac{(3(d\!-\!1))^{2d/3}C_{P}^{3-d} C_{I} (C_{A}^{1,0})^{2} }{\epsilon \nu^{\ast} d^{7d/6}} \big(1 + \frac{1}{\beta_{\tt{is}}} \big)^{2} h | z_{n+1/2}^{\pm} |_{2} \| \xi_{n+1/2}^{\mp,h} \|^{2} 
\notag 
\\
& \quad
+ \frac{(3(d\!-\!1))^{4d/3} C_{P}^{2(3-d)}}{64 (\epsilon \nu^{\ast})^{3} d^{7d/3}} \| \nabla {\color{red} z_{n+1/2}^{\pm} } \|^{4} \| \xi_{n+1/2}^{\mp,h} \|^{2}
            + 5 \epsilon \nu^{\ast} \| \nabla \xi_{n+1/2}^{\pm,h} \|^{2}
            + \epsilon \nu^{\ast} \| \nabla \xi_{n+1/2}^{\mp,h} \|^{2} \notag 
\\
&
\quad
+ \frac{C(\epsilon) h^{2r}}{\nu^{\ast}} \| | \nabla z^{\mp} | \|_{\infty,0}^{2} 
            \Big( \tau_{\rm{max}}^{3}  \int_{t_{n}}^{t_{n+1}} \| z^{\pm}_{tt} \|_{r+1}^{2} dt + \| z^{\pm}(t_{n+1/2}) \|_{r+1}^{2} \Big) \notag 
\\
&
\quad
+ \frac{C(\epsilon) h^{2r}}{\nu^{\ast}} \| | \nabla z^{\pm} | \|_{\infty,0}^{2}
            \Big( \tau_{\rm{max}}^{3} \int_{t_{n}}^{t_{n+1}} \| z^{\mp}_{tt} \|_{r+1}^{2} dt + \| z^{\mp}(t_{n+1/2}) \|_{r+1}^{2} \Big). \notag 
        \end{align}
We combine \eqref{eq:errorL2-eq3}-\eqref{eq:errorL2-term10-11} 
and \eqref{eq:errorL2-term12-13} 
to obtain the following bound for \eqref{eq:errorL2-eq3}
\begin{confidential}
		\color{darkblue}
			\begin{align*}
				&\frac{1}{2\tau_{n}} \big( \| \xi_{n+1}^{\pm,h} \|^{2} - \| \xi_{n}^{\pm,h} \|^{2} \big) + \nu^+ \| \nabla \xi_{n+1/2}^{\pm,h} \|^{2} 
				+ \nu^- \big( \nabla \xi_{n+1/2}^{\mp,h}, \nabla \xi_{n+1/2}^{\pm,h} \big)  \\
				\leq& \frac{C (\epsilon) h^{2r+2}}{\tau_n \nu^{\ast}} \int_{t_{n}}^{t_{n+1}} \| z^{\pm} \|_{r+1}^{2} dt 
				+ \epsilon \nu^{\ast} \| \nabla \xi_{n+1/2}^{\pm,h} \|^{2} \\
				&+ \frac{C(\epsilon)(\nu^{+})^{2} h^{2r}}{ \nu^{\ast}} \Big( \tau_{\rm{max}}^{3} \int_{t_{n}}^{t_{n+1}} \| z_{tt}^{\pm} \|_{r+1}^{2} dt + \| z^{\pm}(t_{n+1/2}) \|_{r+1}^{2} \Big) + \epsilon \nu^{\ast} \| \nabla \xi_{n+1/2}^{\pm,h} \|^{2} \\
				&+ \frac{C(\epsilon)(\nu^{-})^{2} h^{2r}}{ \nu^{\ast}} \Big( \tau_{\rm{max}}^{3} \int_{t_{n}}^{t_{n+1}} \| z_{tt}^{\mp} \|_{r+1}^{2} dt + \| z^{\mp}(t_{n+1/2}) \|_{r+1}^{2} \Big) + \epsilon \nu^{\ast} \| \nabla \xi_{n+1/2}^{\pm,h} \|^{2} \\
				&+ \frac{ C(\epsilon) (\nu^{-})^{2} \tau_{\rm{max}}^{3}}{\nu^{\ast}} 
				\int_{t_{n}}^{t_{n+1}} \| \nabla z_{tt}^{\mp} \|^{2} dt + \epsilon \nu^{\ast} \| \nabla \xi_{n+1/2}^{\pm,h} \|^{2} \\
				&+ \frac{ C(\epsilon) (\nu^{-})^{2} \tau_{\rm{max}}^{3}}{\nu^{\ast}} 
				\int_{t_{n}}^{t_{n+1}} \| \nabla z_{tt}^{\mp} \|^{2} dt + \epsilon \nu^{\ast} \| \nabla \xi_{n+1/2}^{\pm,h} \|^{2} \\
				&+ \frac{C(\epsilon) h^{2s+2}}{\nu^{\ast}} \| p^{\pm}(t_{n+1/2}) \|_{s+1}^{2}
				+ \epsilon \nu^{\ast} \| \nabla \xi_{n+1/2}^{\pm,h} \|^{2} \\
				&+ \frac{ C(\epsilon) \tau_{\rm{max}}^{3}}{\nu^{\ast}} \int_{t_{n}}^{t_{n+1}} \| \nabla z_{ttt}^{\mp} \|^{2} dt 
				+ \epsilon \nu^{\ast} \| \nabla \xi_{n+1/2}^{\pm,h} \|^{2} \\
				&+ \frac{C(\epsilon)h^{2r} }{\nu^{\ast}} \| B_{\circ} \|^{2}
				\big( \tau_{\rm{max}}^{3} \int_{t_{n}}^{t_{n+1}} \| z_{tt}^{\pm} \|_{r+1}^{2} dt + \| z^{\pm}(t_{n+1/2}) \|_{r+1}^{2} \big) 
				+ \epsilon \nu^{\ast} \| \nabla \xi_{n+1/2}^{\pm,h} \|^{2} \\
				&+ \frac{C(\epsilon) \tau_{\rm{max}}^{3}}{\nu^{\ast}} \| B_{\circ} \|^{2} 
				\int_{t_{n}}^{t_{n+1}} \| \nabla z_{tt}^{\pm} \|^{2} dt + \epsilon \nu^{\ast} \| \nabla \xi_{n+1/2}^{\pm,h} \|^{2} \\
				&+ \frac{C(\epsilon) \tau_{\rm{max}}^{3} }{\nu^{\ast}} \Big( \| |z^{\mp}| \|_{\infty,1}^{2} \!\!\int_{t_{n}}^{t_{n\!+\!1}} \!\! \| \nabla z_{tt}^{\pm} \|^{2} dt + \| |z^{\pm}| \|_{\infty,1,1/2}^{2} \!\!\int_{t_{n}}^{t_{n\!+\!1}} \!\! \| \nabla z_{tt}^{\mp} \|^{2} dt \Big) \!+\! \epsilon \nu^{\ast} \| \nabla \xi_{n+1\!/\!2}^{\pm,h} \|^{2} \\
				&+ \frac{(3(d\!-\!1))^{2d/3}C_{P}^{3-d} C_{I} (C_{A}^{1,0})^{2} }{\epsilon \nu^{\ast} d^{7d/6}} \big(1 + \frac{1}{\beta_{\tt{is}}} \big)^{2} h | z_{n+1/2}^{\pm} |_{2} \| \xi_{n+1/2}^{\mp,h} \|^{2} 
				\notag \\
				&+ \frac{(3(d\!-\!1))^{4d/3} C_{P}^{2(3-d)}}{64 (\epsilon \nu^{\ast})^{3} d^{7d/3}} \| \nabla z_{n+1/2}^{\pm} \|^{4} \| \xi_{n+1/2}^{\mp,h} \|^{2}
				+ 5 \epsilon \nu^{\ast} \| \nabla \xi_{n+1/2}^{\pm,h} \|^{2}
				+ \epsilon \nu^{\ast} \| \nabla \xi_{n+1/2}^{\mp,h} \|^{2} \notag \\
				&+ \frac{C(\epsilon) h^{2r}}{\nu^{\ast}} \| | \nabla z^{\mp} | \|_{\infty,0}^{2} 
				\Big( \tau_{\rm{max}}^{3}  \int_{t_{n}}^{t_{n+1}} \| z^{\pm}_{tt} \|_{r+1}^{2} dt + \| z^{\pm}(t_{n+1/2}) \|_{r+1}^{2} \Big) \notag \\
				&+ \frac{C(\epsilon) h^{2r}}{\nu^{\ast}} \| | \nabla z^{\pm} | \|_{\infty,0}^{2}
				\Big( \tau_{\rm{max}}^{3} \int_{t_{n}}^{t_{n+1}} \| z^{\mp}_{tt} \|_{r+1}^{2} dt + \| z^{\mp}(t_{n+1/2}) \|_{r+1}^{2} \Big).
			\end{align*}
\begin{align*}
&
\frac{1}{2\tau_{n}} \big( \| \xi_{n+1}^{\pm,h} \|^{2} - \| \xi_{n}^{\pm,h} \|^{2} \big) + \nu^+ \| \nabla \xi_{n+1/2}^{\pm,h} \|^{2} 
				+ \frac{|\nu^{-}|}{2} \big\|\nabla \xi_{n+1/2}^{\mp,h} + \text{sign}(\nu^{-}) \nabla \xi_{n+1/2}^{\pm,h} \big\|^2 
				- \frac{|\nu^{-}|}{2} \big( \|\nabla \xi_{n+1/2}^{\mp,h} \|^2 + \|\nabla \xi_{n+1/2}^{\pm,h} \|^2 \big) 
\\
& 
\leq \frac{C (\epsilon) h^{2r+2}}{\tau_n \nu^{\ast}} \int_{t_{n}}^{t_{n+1}} \| z^{\pm} \|_{r+1}^{2} dt
                + \frac{C(\epsilon)(\nu^{+})^{2} h^{2r}}{ \nu^{\ast}} \Big( \tau_{\rm{max}}^{3} \int_{t_{n}}^{t_{n+1}} \| z_{tt}^{\pm} \|_{r+1}^{2} dt + \| z^{\pm}(t_{n+1/2}) \|_{r+1}^{2} \Big) \\
                &+ \frac{C(\epsilon)(\nu^{-})^{2} h^{2r}}{ \nu^{\ast}} \Big( \tau_{\rm{max}}^{3} \int_{t_{n}}^{t_{n+1}} \| z_{tt}^{\mp} \|_{r+1}^{2} dt + \| z^{\mp}(t_{n+1/2}) \|_{r+1}^{2} \Big) \\
                &+ \frac{ C(\epsilon) (\nu^{-})^{2} \tau_{\rm{max}}^{3}}{\nu^{\ast}} 
                \Big( \int_{t_{n}}^{t_{n+1}} \| \nabla z_{tt}^{\mp} \|^{2} dt 
                + \int_{t_{n}}^{t_{n+1}} \| \nabla z_{tt}^{\mp} \|^{2} dt \Big)
                + \frac{C(\epsilon) h^{2s+2}}{\nu^{\ast}} \| p^{\pm}(t_{n+1/2}) \|_{s+1}^{2}
                + \frac{ C(\epsilon) \tau_{\rm{max}}^{3}}{\nu^{\ast}} \int_{t_{n}}^{t_{n+1}} \| \nabla z_{ttt}^{\mp} \|^{2} dt \\
                &+ \frac{C(\epsilon)h^{2r} }{\nu^{\ast}} \| B_{\circ} \|^{2}
				\big( \tau_{\rm{max}}^{3} \int_{t_{n}}^{t_{n+1}} \| z_{tt}^{\pm} \|_{r+1}^{2} dt + \| z^{\pm}(t_{n+1/2}) \|_{r+1}^{2} \big)
                + \frac{C(\epsilon) \tau_{\rm{max}}^{3}}{\nu^{\ast}} \| B_{\circ} \|^{2} 
				\int_{t_{n}}^{t_{n+1}} \| \nabla z_{tt}^{\pm} \|^{2} dt \\
                &+ \frac{C(\epsilon) \tau_{\rm{max}}^{3} }{\nu^{\ast}} \Big( \| |z^{\mp}| \|_{\infty,1}^{2} \!\!\int_{t_{n}}^{t_{n\!+\!1}} \!\! \| \nabla z_{tt}^{\pm} \|^{2} dt + \| |z^{\pm}| \|_{\infty,1,1/2}^{2} \!\!\int_{t_{n}}^{t_{n\!+\!1}} \!\! \| \nabla z_{tt}^{\mp} \|^{2} dt \Big) \\
                &+ \frac{C(\epsilon) h^{2r}}{\nu^{\ast}} \| | \nabla z^{\mp} | \|_{\infty,0}^{2} 
				\Big( \tau_{\rm{max}}^{3}  \int_{t_{n}}^{t_{n+1}} \| z^{\pm}_{tt} \|_{r+1}^{2} dt + \| z^{\pm}(t_{n+1/2}) \|_{r+1}^{2} \Big)  \\
				&+ \frac{C(\epsilon) h^{2r}}{\nu^{\ast}} \| | \nabla z^{\pm} | \|_{\infty,0}^{2}
				\Big( \tau_{\rm{max}}^{3} \int_{t_{n}}^{t_{n+1}} \| z^{\mp}_{tt} \|_{r+1}^{2} dt + \| z^{\mp}(t_{n+1/2}) \|_{r+1}^{2} \Big) \\
                &+ \frac{(3(d\!-\!1))^{2d/3}C_{P}^{3-d} C_{I} (C_{A}^{1,0})^{2} }{\epsilon \nu^{\ast} d^{7d/6}} \big(1 + \frac{1}{\beta_{\tt{is}}} \big)^{2} h | z_{n+1/2}^{\pm} |_{2} \| \xi_{n+1/2}^{\mp,h} \|^{2} 
				\notag \\
				&+ \frac{(3(d\!-\!1))^{4d/3} C_{P}^{2(3-d)}}{64 (\epsilon \nu^{\ast})^{3} d^{7d/3}} \| \nabla z_{n+1/2}^{\pm} \|^{4} \| \xi_{n+1/2}^{\mp,h} \|^{2}
				+ 15 \epsilon \nu^{\ast} \| \nabla \xi_{n+1/2}^{\pm,h} \|^{2}
				+ \epsilon \nu^{\ast} \| \nabla \xi_{n+1/2}^{\mp,h} \|^{2} \notag \\ 
			\end{align*}
		\normalcolor
\end{confidential}
\begin{align}
&
\frac{1}{2\tau_{n}} \big( \| \xi_{n+1}^{\pm,h} \|^{2} - \| \xi_{n}^{\pm,h} \|^{2} \big) + \nu^+ \| \nabla \xi_{n+1/2}^{\pm,h} \|^{2}
            - \frac{|\nu^{-}|}{2} \big( \|\nabla \xi_{n+1/2}^{\mp,h} \|^2 + \|\nabla \xi_{n+1/2}^{\pm,h} \|^2 \big) 
\label{eq:errorL2-eq4}
\\
&
\quad
+ \frac{|\nu^{-}|}{2} \big\|\nabla \xi_{n+1/2}^{\mp,h} + \text{sign}(\nu^{-}) \nabla \xi_{n+1/2}^{\pm,h} \big\|^2 
 - 15 \epsilon \nu^{\ast} \| \nabla \xi_{n+1/2}^{\pm,h} \|^{2}
 -  \epsilon \nu^{\ast} \| \nabla \xi_{n+1/2}^{\mp,h} \|^{2}  
\notag 
\\
&
\leq 
\frac{(3(d\!-\!1))^{2d/3}C_{P}^{3-d} C_{I} (C_{A}^{1,0})^{2} }{\epsilon \nu^{\ast} d^{7d/6}} \big(1 + \frac{1}{\beta_{\tt{is}}} \big)^{2} h | {\color{red}  z_{n+1/2}^{\pm} } |_{2} \| \xi_{n+1/2}^{\mp,h} \|^{2} 
\notag 
\\
& \quad
+ \frac{(3(d\!-\!1))^{4d/3} C_{P}^{2(3-d)}}{64 (\epsilon \nu^{\ast})^{3} d^{7d/3}} \| \nabla {\color{red} z_{n+1/2}^{\pm} } \|^{4} \| \xi_{n+1/2}^{\mp,h} \|^{2}
            + \frac{C (\epsilon) h^{2r+2}}{\tau_n \nu^{\ast}} \int_{t_{n}}^{t_{n+1}} \| z^{\pm} \|_{r+1}^{2} dt 
\notag 
\\
& \quad
+ \frac{C(\epsilon)(\nu^{+})^{2} h^{2r}}{ \nu^{\ast}} \Big( \tau_{\rm{max}}^{3} \int_{t_{n}}^{t_{n+1}} \| z_{tt}^{\pm} \|_{r+1}^{2} dt + \| z^{\pm}(t_{n+1/2}) \|_{r+1}^{2} \Big) 
\notag 
\\
& \quad
+ \frac{C(\epsilon)(\nu^{-})^{2} h^{2r}}{ \nu^{\ast}} \Big( \tau_{\rm{max}}^{3} \int_{t_{n}}^{t_{n+1}} \| z_{tt}^{\mp} \|_{r+1}^{2} dt + \| z^{\mp}(t_{n+1/2}) \|_{r+1}^{2} \Big) 
\notag 
\\
& \quad
+ \frac{ C(\epsilon) (\nu^{-})^{2} \tau_{\rm{max}}^{3}}{\nu^{\ast}} 
            \Big( \int_{t_{n}}^{t_{n+1}} \!\| \nabla z_{tt}^{\mp} \|^{2} dt 
            +\! \int_{t_{n}}^{t_{n+1}} \!\| \nabla z_{tt}^{\mp} \|^{2} dt \Big) 
            +\! \frac{C(\epsilon) h^{2s+2}}{\nu^{\ast}} \| p^{\pm}(t_{n+1/2}) \|_{s+1}^{2} 
\notag 
\\
& \quad
+ \frac{ C(\epsilon) \tau_{\rm{max}}^{3}}{\nu^{\ast}} \int_{t_{n}}^{t_{n+1}} \| \nabla z_{ttt}^{\mp} \|^{2} dt
            + \frac{C(\epsilon) \tau_{\rm{max}}^{3}}{\nu^{\ast}} \| B_{\circ} \|^{2} 
			\int_{t_{n}}^{t_{n+1}} \| \nabla z_{tt}^{\pm} \|^{2} dt 
\notag 
\\
& \quad
+ \frac{C(\epsilon)h^{2r} }{\nu^{\ast}} \| B_{\circ} \|^{2}
            \big( \tau_{\rm{max}}^{3} \int_{t_{n}}^{t_{n+1}} \| z_{tt}^{\pm} \|_{r+1}^{2} dt + \| z^{\pm}(t_{n+1/2}) \|_{r+1}^{2} \big) 
\notag 
\\
& \quad
+ \frac{C(\epsilon) \tau_{\rm{max}}^{3} }{\nu^{\ast}} \Big( \| |z^{\mp}| \|_{\infty,1}^{2} \!\!\int_{t_{n}}^{t_{n\!+\!1}} \!\! \| \nabla z_{tt}^{\pm} \|^{2} dt + \| |z^{\pm}| \|_{\infty,1,1/2}^{2} \!\!\int_{t_{n}}^{t_{n\!+\!1}} \!\! \| \nabla z_{tt}^{\mp} \|^{2} dt \Big) 
\notag 
\\
& \quad
+ \frac{C(\epsilon) h^{2r}}{\nu^{\ast}} \| | \nabla z^{\mp} | \|_{\infty,0}^{2} 
			\Big( \tau_{\rm{max}}^{3}  \int_{t_{n}}^{t_{n+1}} \| z^{\pm}_{tt} \|_{r+1}^{2} dt + \| z^{\pm}(t_{n+1/2}) \|_{r+1}^{2} \Big)  
\notag 
\\
& \quad
+ \frac{C(\epsilon) h^{2r}}{\nu^{\ast}} \| | \nabla z^{\pm} | \|_{\infty,0}^{2}
			\Big( \tau_{\rm{max}}^{3} \int_{t_{n}}^{t_{n+1}} \| z^{\mp}_{tt} \|_{r+1}^{2} dt + \| z^{\mp}(t_{n+1/2}) \|_{r+1}^{2} \Big). 
\notag 
\end{align}
Now adding the two 
relations above
, 
and summing 
over $n$ from $0$ to $N-1$ yields
\begin{confidential}
    \color{darkblue}
        \begin{align*}
            &\frac{1}{2\tau_{n}} \big( \| \xi_{n+1}^{+,h} \|^{2} - \| \xi_{n}^{+,h} \|^{2} \big)
            + \frac{1}{2\tau_{n}} \big( \| \xi_{n+1}^{-,h} \|^{2} - \| \xi_{n}^{-,h} \|^{2} \big)
            + (\nu^+ - |\nu^{-}| ) \big( \| \nabla \xi_{n+1/2}^{+,h} \|^{2} + \| \nabla \xi_{n+1/2}^{-,h} \|^{2} \big) \\
            & - 16 \epsilon \nu^{\ast} \big( \| \nabla \xi_{n+1/2}^{+,h} \|^{2} 
            + \| \nabla \xi_{n+1/2}^{-,h} \|^{2} \big) \\
            &\leq \frac{(3(d\!-\!1))^{2d/3}C_{P}^{3-d} C_{I} (C_{A}^{1,0})^{2} }{\epsilon \nu^{\ast} d^{7d/6}} \big(1 + \frac{1}{\beta_{\tt{is}}} \big)^{2} h 
            \big( | z_{n+1/2}^{+} |_{2} \| \xi_{n+1/2}^{-,h} \|^{2} + | z_{n+1/2}^{-} |_{2} \| \xi_{n+1/2}^{+,h} \|^{2} \big)
            \notag \\
            &+ \frac{(3(d\!-\!1))^{4d/3} C_{P}^{2(3-d)}}{64 (\epsilon \nu^{\ast})^{3} d^{7d/3}} 
            \big( \| \nabla z_{n+1/2}^{+} \|^{4} \| \xi_{n+1/2}^{-,h} \|^{2} + \| \nabla z_{n+1/2}^{-} \|^{4} \| \xi_{n+1/2}^{+,h} \|^{2} \big) + \cdots \cdots 
        \end{align*}
        \begin{align*}
            &\big( \| \xi_{n+1}^{+,h} \|^{2} - \| \xi_{n}^{+,h} \|^{2} \big) 
            + \big( \| \xi_{n+1}^{-,h} \|^{2} - \| \xi_{n}^{-,h} \|^{2} \big) 
            + 2 \nu^{\ast} \tau_{n} \big( 1 - 16 \epsilon \big) 
            \big( \| \nabla \xi_{n+1/2}^{+,h} \|^{2} + \| \nabla \xi_{n+1/2}^{-,h} \|^{2} \big) \\
            &\leq \frac{2(3(d\!-\!1))^{2d/3}C_{P}^{3-d} C_{I} (C_{A}^{1,0})^{2} }{\epsilon \nu^{\ast} d^{7d/6}} \big(1 + \frac{1}{\beta_{\tt{is}}} \big)^{2} h \tau_{n}
            \big( | z_{n+1/2}^{+} |_{2} \| \xi_{n+1/2}^{-,h} \|^{2} + | z_{n+1/2}^{-} |_{2} \| \xi_{n+1/2}^{+,h} \|^{2} \big) \\
            &+ \frac{(3(d\!-\!1))^{4d/3} C_{P}^{2(3-d)} \tau_{n}}{32 (\epsilon \nu^{\ast})^{3} d^{7d/3}} 
            \big( \| \nabla z_{n+1/2}^{+} \|^{4} \| \xi_{n+1/2}^{-,h} \|^{2} + \| \nabla z_{n+1/2}^{-} \|^{4} \| \xi_{n+1/2}^{+,h} \|^{2} \big) + \cdots \cdots 
        \end{align*}
        We denote 
        \begin{align*}
            C_{1}^{\ast} = \frac{2(3(d\!-\!1))^{2d/3}C_{P}^{3-d} C_{I} (C_{A}^{1,0})^{2} }{\epsilon \nu^{\ast} d^{7d/6}}\big(1 + \frac{1}{\beta_{\tt{is}}} \big)^{2}, \quad
            C_{2}^{\ast} = \frac{(3(d\!-\!1))^{4d/3} C_{P}^{2(3-d)}}{32 (\nu^{\ast})^{3} d^{7d/3}} 
        \end{align*}
        \begin{align*}
            &\big( \| \xi_{N}^{+,h} \|^{2} +  \| \xi_{N}^{-,h} \|^{2} \big)
            + 2 \nu^{\ast} \big( 1 - 16 \epsilon \big) \sum_{n=1}^{N} \tau_{n} \big( \| \nabla \xi_{n+1/2}^{+,h} \|^{2} + \| \nabla \xi_{n+1/2}^{-,h} \|^{2} \big) \\ 
            &\leq \sum_{n=1}^{N-1} \frac{C_{1}^{\ast}}{\epsilon} h \tau_{n}
            \big( | z_{n+1/2}^{+} |_{2} \| \xi_{n+1/2}^{-,h} \|^{2} + | z_{n+1/2}^{-} |_{2} \| \xi_{n+1/2}^{+,h} \|^{2} \big)  \notag \\
            &+ \sum_{n=1}^{N-1} \frac{C_{2}^{\ast}}{\epsilon^{3}} \tau_{n} \big( \| \nabla z_{n+1/2}^{+} \|^{4} \| \xi_{n+1/2}^{-,h} \|^{2} + \| \nabla z_{n+1/2}^{-} \|^{4} \| \xi_{n+1/2}^{+,h} \|^{2} \big)
            \notag \\
            &+\!\! \frac{C \!(\epsilon) h^{2r\!+\!2}}{\nu^{\ast}}\! \big( \| z^{+} \|_{2,r\!+\!1}^{2} \!+\! \| z^{-} \|_{2,r\!+\!1}^{2} \big) 
            \!+\! \frac{C(\epsilon)\big(\!(\nu^{+})^{2} \!\!+\!\! (\nu^{-})^{2} \!\big) h^{2r}\tau_{\rm{max}}^{4}}{ \nu^{\ast}}\! 
            \big( \| z_{tt}^{+} \|_{2,r\!+\!1}^{2} \!\!+\!\! \| z_{tt}^{-} \|_{2,r\!+\!1}^{2} \big) \notag \\
            &+ \frac{C(\epsilon) \big((\nu^{+})^{2} + (\nu^{-})^{2} \big) h^{2r}}{ \nu^{\ast}}
            \big( \| |z^{+}| \|_{2,r+1,1/2}^{2} + \| |z^{-}| \|_{2,r+1,1/2}^{2} \big) \notag \\
            &+\! \frac{ C(\epsilon) (\nu^{-})^{2} \tau_{\rm{max}}^{4}}{\nu^{\ast}} 
            \big( \| \nabla z_{tt}^{+} \|_{2,0}^{2} \!+\! \| \nabla z_{tt}^{-} \|_{2,0}^{2} \big)
            \!+\! \frac{C(\epsilon) h^{2s\!+\!2}}{\nu^{\ast}} \big( \| |p^{+}| \|_{2,s\!+\!1,1\!/\!2}^{2} 
            \!+\! \| |p^{-}| \|_{2,s\!+\!1,1\!/\!2}^{2} \big) \notag \\
            &+ \frac{ C(\epsilon) \tau_{\rm{max}}^{4}}{\nu^{\ast}} \big( \| \nabla z_{ttt}^{+} \|_{2,0}^{2} + \| \nabla z_{ttt}^{-} \|_{2,0}^{2} \big)
            + \frac{C(\epsilon) \tau_{\rm{max}}^{4}}{\nu^{\ast}} \| B_{\circ} \|^{2} \big( \| \nabla z_{tt}^{+} \|_{2,0}^{2} + \| \nabla z_{tt}^{-} \|_{2,0}^{2} \big)
            \notag \\
            &+ \frac{C(\epsilon)h^{2r} \tau_{\rm{max}}^{4}}{\nu^{\ast}} \| B_{\circ} \|^{2}
            \big( \| z_{tt}^{\pm} \|_{2,r+1}^{2} + \| z_{tt}^{\pm} \|_{2,r+1}^{2} \big) \notag \\
            &+ \frac{C(\epsilon)h^{2r} }{\nu^{\ast}} \| B_{\circ} \|^{2} \big( \| |z^{+}| \|_{2,r+1,1/2}^{2} + \| |z^{-}| \|_{2,r+1,1/2}^{2} \big) 
            \notag \\
            &+ \frac{C(\epsilon) \tau_{\rm{max}}^{4} }{\nu^{\ast}} \big( \| |z^{-}| \|_{\infty,1}^{2} \| \nabla z_{tt}^{+} \|_{2,0}^{2} + \| |z^{+}| \|_{\infty,1}^{2} \| \nabla z_{tt}^{-} \|_{2,0}^{2} 
            + \| |z^{+}| \|_{\infty,1,1/2}^{2} \| \nabla z_{tt}^{-} \|_{2,0}^{2} \notag \\
            &\qquad \qquad \qquad \qquad + \| |z^{-}| \|_{\infty,1,1/2}^{2} \| \nabla z_{tt}^{+} \|_{2,0}^{2} \big) 
            \notag \\
            &+ \frac{C(\epsilon) h^{2r}\tau_{\rm{max}}^{4}}{\nu^{\ast}} \big( \| | \nabla z^{-} | \|_{\infty,0}^{2} \| z^{+}_{tt} \|_{2,r+1}^{2} + \| | \nabla z^{+} | \|_{\infty,0}^{2} \| z^{-}_{tt} \|_{2,r+1}^{2} \big)  \notag \\
            & + \frac{C(\epsilon) h^{2r}}{\nu^{\ast}} \big( \| |z^{+}| \|_{2,r+1,1/2}^{2} +  \| |z^{-}| \|_{2,r+1,1/2}^{2} \big)
            + \big( \| \xi_{0}^{+,h} \|^{2} +  \| \xi_{0}^{-,h} \|^{2} \big) \notag 
        \end{align*}
        Equivalently, 
    \normalcolor
\end{confidential}
\begin{align}
&
\Big[ 1 - \Big( \frac{C_{1}^{\ast}h Z_{N-1}^{1}}{\epsilon} + \frac{C_{2}^{\ast} Z_{N-1}^{2}}{\epsilon^{3}} \Big) \tau_{N-1} \Big] \big( \| \xi_{N}^{+,h} \|^{2} +  \| \xi_{N}^{-,h} \|^{2} \big) 
\label{eq:errorL2-eq5}  
\\
&
\quad
+ 2 \nu^{\ast} \big( 1 - 16 \epsilon \big) \sum_{n=1}^{N-1} \tau_{n} \big( \| \nabla \xi_{n+1/2}^{+,h} \|^{2} + \| \nabla \xi_{n+1/2}^{-,h} \|^{2} \big) \notag \\ 
            &\leq \sum_{n=1}^{N-1}\! \Big( \frac{C_{1}^{\ast}h Z_{n}^{1}}{\epsilon} \!+\! \frac{C_{2}^{\ast} Z_{n}^{2}}{\epsilon^{3}} \Big) \tau_{n} \big( \| \xi_{n}^{+,h} \|^{2} +  \| \xi_{n}^{-,h} \|^{2} \big)
            + C(\epsilon) F \big( h^{2r}, h^{2s+2}, \tau_{\rm{max}}^{4}, \| B_{\circ} \|^{2} \big).\notag 
        \end{align}
\begin{confidential}
    \color{darkblue}
        where 
        \begin{align*}
            Z_{n}^{(1)} = \max\{ | z_{n+1/2}^{+} |_{2}, | z_{n+1/2}^{-} |_{2} \}, \quad
            Z_{n}^{(2)} = \max\{ \| \nabla z_{n+1/2}^{+} \|^{4}, \| \nabla z_{n+1/2}^{-} \|^{4} \}.
        \end{align*}
    \normalcolor
\end{confidential}
Here we impose the positivity conditions
\begin{align*}
            &1 - \Big( \frac{C_{1}^{\ast}h Z_{N-1}^{1}}{\epsilon} + \frac{C_{2}^{\ast} Z_{N-1}^{2}}{\epsilon^{3}} \Big) \tau_{N-1} > 0, \\
            & 1 - 16 \epsilon > 0, 
        \end{align*}
which 
imply 
the time step restrictions 
\eqref{eq:step-res-Error-L2}. 
T
he discrete Gr\"onwall inequality \cite[p. 369]{MR1043610} 
applied to \eqref{eq:errorL2-eq5} gives
\begin{confidential}
    \color{darkblue}
        \begin{align*}
            &\epsilon < \frac{1}{16}, \quad \frac{1}{\epsilon} > 16, \\
            &1 > \Big( \frac{C_{1}^{\ast}h Z_{N-1}^{(1)}}{\epsilon} + \frac{C_{2}^{\ast} Z_{N-1}^{(2)}}{\epsilon^{3}} \Big) \tau_{N-1} > 16 \big( C_{1}^{\ast}h Z_{N-1}^{(1)} + 256 C_{2}^{\ast} Z_{N-1}^{(2)} \big) \tau_{N-1} \\
            & \tau_{N-1} < \frac{1}{16 \big( C_{1}^{\ast}h Z_{N-1}^{(1)} + 256 C_{2}^{\ast} Z_{N-1}^{(2)} \big)}.
        \end{align*}
    \normalcolor
\end{confidential}
        \begin{align}
            &\big( \| \xi_{N}^{+,h} \|^{2} +  \| \xi_{N}^{-,h} \|^{2} \big) 
            + \nu^{\ast} \sum_{n=1}^{N-1} \tau_{n} \big( \| \nabla \xi_{n+1/2}^{+,h} \|^{2} + \| \nabla \xi_{n+1/2}^{-,h} \|^{2} \big)
            \label{eq:xi_estimate} \\
            &\leq C \exp \Big( C \sum_{n=1}^{N-1} \big( C_{1}^{\ast}h Z_{n}^{1} + C_{2}^{\ast} Z_{n}^{2} \big) \tau_{n} \Big) F \big( h^{2r}, h^{2s+2}, \tau_{\rm{max}}^{4}, \| B_{\circ} \|^{2} \big). \notag  
        \end{align}
%
%
We use the triangle inequality, 
\eqref{eq:Stoke-Approx}, \eqref{eq:approx-thm}, \eqref{eq:consistency-eq1} 
and \eqref{eq:xi_estimate} to 
obtain
\begin{align}
            &\| e_{N}^{\pm} \|^{2} \leq 2 \| \xi_{N}^{\pm,h} \|^{2} + 2 \| \eta_{N}^{\pm} \|^{2} 
\leq 2 \| \xi_{N}^{\pm,h} \|^{2} + C h^{2r} \| |z^{\pm}| \|_{\infty,r+1}^{2},
\label{eq:errorL2-xi-eta} 
\end{align}
and
\begin{align}
&
\nu^{\ast} \sum_{n=1}^{N-1} \tau_{n} \| \nabla ( e_{n+1}^{\pm}  +  e_{n}^{\pm}  ) / 2\|^{2} 
\leq 
2 \nu^{\ast} \!\sum_{n=1}^{N-1} \tau_{n} \big( \| \nabla \xi_{n+1/2}^{\pm,h} \|^{2} 
+ \| {\color{red} \nabla \eta_{n+1/2}^{\pm} } \|^{2} \big) 
\notag 
\\
&
\leq 
2 \nu^{\ast} \!\!\sum_{n=1}^{N-1}\!\! \tau_{n} \| \nabla \xi_{n+1/2}^{\pm,h} \|^{2} 
+ C \nu^{\ast} h^{2r} \tau_{n} \!\!\sum_{n=1}^{N-1} \!\!\big(\color{red} \| z_{n+1/2}^{\pm} 
- z^{\pm}(t_{n+1/2}) \normalcolor\|_{r+1}^{2} \!+\! \| z^{\pm}(t_{n+1/2}) \|_{r+1}^{2} \big)
\notag
\\
&
\leq 
2 \nu^{\ast} \!\!\sum_{n=1}^{N-1} \tau_{n} \| \nabla \xi_{n+1/2}^{\pm,h} \|^{2} \!+\! C \nu^{\ast} h^{2r} \big(\tau_{\rm{max}}^{4} \| z^{\pm} \|_{2,r+1}^{2} 
+ \| |z^{\pm}| \|_{2,r+1,1/2}^{2} \big). 
\notag 
\end{align} 
Finally, 
we apply 
\eqref{eq:Stoke-Approx}, \eqref{eq:approx-thm} 
in
\eqref{eq:errorL2-xi-eta} and use \eqref{eq:errorL2-eq5} to 
derive 
\eqref{eq:error-L2-final}. 
\end{proof}

\section{Time Adaptivity} 
To construct the time adaptive mechanism for the fully-coupled implicit monolithic method \eqref{eq:CN-MHD} with the BE partitioned iteration \eqref{eq:BE-iter} (PIM algorithm), we  estimate its local truncation error (LTE) by the fully-explicit two-step AB2-like scheme\footnote{The derivation of scheme is similar to that of two-step Adam-Bashforth (AB2) scheme, hence we call it AB2-like scheme.} and adopt the improved time step controller proposed in \cite{MR1227985} to adjust time steps. 
The AB2-like scheme for the ordinary differential equation $y' = 
g(t,y)$ is
\begin{align}
    y_{n\!+\!1}^{\tt{AB2-like}} \!&=\! y_{n} \!+\! \frac{t_{n\!+\!1} \!-\!t_{n}}{2 (t_{n-1/2} \!-\!t_{n-3/2})} 
    \Big[ (t_{n\!+\!1} \!+\!t_{n} \!-\!2 t_{n-3/2} )  g^{\tt{Mid}} (t_{n-1/2}, y_{n-1/2})   
    \label{eq:AB2-like-CN} \\
    & \qquad \qquad \qquad \qquad \qquad \qquad 
    \!-\! (t_{n\!+\!1} \!+\!t_{n} \!-\! 2 t_{n-1/2} ) g^{\tt{Mid}} (t_{n-3/2}, y_{n-3/2}) \Big], \notag 
\end{align}
where $g^{\tt{Mid}}(t_{n-1/2}, y_{n-1/2}) $ and $g^{\tt{Mid}} (t_{n-3/2}, y_{n-3/2})$ are calculated by the midpoint rule at time 
nodes
$t_{n}$, $t_{n-1}$ and $t_{n-1}$, $t_{n-2}$, respectively, i.e.,
\begin{align*}
    g^{\tt{Mid}}(t_{n-1/2}, y_{n-1/2}) 
    &= \frac{y_{n} - y_{n-1}}{\tau_{n-1}}, \quad
    g^{\tt{Mid}} (t_{n-3/2}, y_{n-3/2}) =
    \frac{y_{n-1} - y_{n-2}}{\tau_{n-2}}.
\end{align*}
From \eqref{eq:AB2-like-CN} 
we see that $y_{n\!+\!1}^{\tt{AB2-like}}$	is 
a
certain interpolant of the previous three midpoint rule solutions $\{ y_{n}, y_{n-1}, y_{n-2} \}$. 
Hence, the AB2-like scheme \eqref{eq:AB2-like-CN} is fully explicit.
We refer to \cite{MR4092601} for the derivation of the explicit AB2-like scheme \eqref{eq:AB2-like-CN} and the following estimator of LTE $\widehat{T}_{n+1}$ of the midpoint rule by AB2-like scheme
	\begin{align*}
		\widehat{T}_{n+1}
		=&\frac{1/24}{\mathcal{R}^{(n)} - 1/24} \big( y_{n+1}^{\tt{Mid}} - y_{n+1}^{\tt{AB2-like}} \big),
		\notag \\
		\mathcal{R}^{(n)} \!=\! & \frac{1}{12} \Big[ 2 \!+\! \frac{3}{\rho_n} \Big( 1 +  \frac{1}{2\rho_{n\!-\!1}} \Big)
		\Big( 1 + \frac{1}{2 \rho_n} \Big) +\! \frac{3}{2 \rho_n} \Big( 1\!+\!\frac{1}{\rho_n} + \frac{1}{2} \frac{1}{\rho_{n\!-\!1}} \frac{1}{\rho_n} \Big)
		\Big], 
	\end{align*}
	and $\rho_n = \tau_{n}/\tau_{n-1}$ denotes the ratio of consecutive time steps.
\begin{confidential}
    \color{darkblue}
    $\mathcal{R}^{(n)}$ for the DLN algorithm is 
    \begin{align*}
        \mathcal{R}^{(n)} \!=\! & \frac{1}{12} \Big[ 2 \!+\! \frac{3}{\tau_n} \Big( 1 + \big(1 \!-\! \beta_2^{(n\!-\!2)} \big) \frac{1}{\tau_{n\!-\!1}} \!+\!\beta_0^{(n\!-\! 2)} \frac{1}{\tau_{n\!-\!2}} \frac{1}{\tau_{n\!-\!1}} \Big)
		\Big( 1 + \big( 1 \!-\!\beta_2^{(n\!-\!1)} \big)\frac{1}{\tau_n}\!+\!\beta_0^{(n\!-\!1)} \frac{1}{\tau_{n\!-\!1}} \frac{1}{\tau_n} \Big) \notag \\
		&\quad \  +\! \frac{3}{\tau_n} \Big( 1\!+\!\frac{1}{\tau_n} + \big(1 \!-\!\beta_2^{(n\!-\!2)} \big) \frac{1}{\tau_{n\!-\!1}} \frac{1}{\tau_n}\!+\!\beta_0^{(n\!-\!2)} \frac{1}{\tau_{n\!-\!2}} \frac{1}{\tau_{n\!-\!1}} \frac{1}{\tau_n} \Big)
		\Big( 1-\beta_2^{(n\!-\!1)} \!+\! \beta_0^{(n\!-\!1)} \frac{1}{\tau_{n\!-\!1}} \Big) \Big], 
    \end{align*}
    \normalcolor
\end{confidential}
The AB2-like solutions for the MHD system in Els\"asser variables \eqref{eq:MHD-Elsa} at time $t_{n+1}$ is
\begin{align}
    z_{n+1}^{\pm,\tt{AB2-like}} \!&=\! z_{n}^{\pm} \!+\! \frac{t_{n\!+\!1} \!-\!t_{n}}{2 (t_{n-1/2} \!-\!t_{n-3/2})} 
    \Big[ (t_{n\!+\!1} \!+\!t_{n} \!-\!2 t_{n-3/2} )  \frac{z_{n}^{\pm} - z_{n-1}^{\pm}}{\tau_{n-1}}   
    \label{eq:AB2-like-MHD} \\
    & \qquad \qquad \qquad \qquad \qquad \qquad 
    \!-\! (t_{n\!+\!1} \!+\!t_{n} \!-\! 2 t_{n-1/2} ) \frac{z_{n-1}^{\pm} - z_{n-2}^{\pm}}{\tau_{n-2}} \Big], \notag
\end{align}
where $\{z_{n}^{\pm}, z_{n-1}^{\pm}, z_{n-2}^{\pm} \}$ are three previous solutions by the PIM algorithm in \eqref{eq:CN-MHD}. 
The corresponding estimator of LTE for the PIM algorithm \eqref{eq:CN-MHD} is 
\begin{gather}
    \widehat{T}_{n+1}
    = \max\bigg\{ \frac{|1/24| \big\| z_{n+1}^{+} - z_{n+1}^{+,\tt{AB2-like}} \big\|}{|\mathcal{R}^{(n)} - 1/24| \big\| z_{n+1}^{+} \big\|}, 
    \frac{|1/24| \big\| z_{n+1}^{-} - z_{n+1}^{-,\tt{AB2-like}} \big\|}{|\mathcal{R}^{(n)} - 1/24| \big\| z_{n+1}^{-} \big\|}
     \bigg\}.
    \label{eq:Estimator-LTE-CN}
\end{gather}
The improved time step controller is 
\begin{gather}
    \tau_{n+1} = \tau_{n} \cdot \min \Big\{ 1.5, \max \Big\{ 0.2, \kappa \big( {\tt{Tol}}/\| \widehat{T}_{n+1} \|  \big)^{\frac{1}{3}} \Big\} \Big\},
    \label{eq:improve-controller}
\end{gather}
where $\kappa \in (0, 1]$ is the safety factor and ${\tt{Tol}}$ is the required tolerance for the LTE.
We note that, for efficiency, 
the step controller in \eqref{eq:improve-controller} increases the step size if the estimator for LTE $\widehat{T}_{n+1}$ is small with respect to ${\tt{Tol}}$. 
Meanwhile, for the robustness of computations, the controller bounds the new time step $\tau_{n+1}$ between $0.2\tau_{n}$ and $1.5\tau_{n}$.
We summarize the time adaptivity mechanism in the following algorithm. 
\LinesNumberedHidden
    \begin{algorithm}[ptbh]
    	\caption{Adaptivity of the PIM algorithm \eqref{eq:CN-MHD} }
    	\label{alg:Adap-DLN}
    	\KwIn{tolerance ${\tt{Tol}}$, previous solutions $z_{n}^{\pm},z_{n-1}^{\pm},z_{n-2}^{\pm}$ by the PIM algorithm in \eqref{eq:CN-MHD},
    		current time step $\tau_{n}$, two previous time step $\tau_{n-1},\tau_{n-2}$, safety factor $\kappa$, \;}
    	compute the current solution $z_{n+1}^{\pm,\tt{PIM}}$ and $p_{n+1}^{\pm}$ 
    	by the refactorization process \;
    	compute the AB2-like solution $z_{n+1}^{\pm,\tt{AB2-like}}$ by \eqref{eq:AB2-like-MHD} \;
    	use $\tau_{n},\tau_{n-1},\tau_{n-2}$ to update $\rho_{n},\rho_{n-1}$ \;
    	compute $\widehat{T}_{n+1}$ by \eqref{eq:Estimator-LTE-CN}      \tcp*{relative estimator}
    	\uIf{$ \widehat{T}_{n+1}  < \rm{Tol}$}
    	{
    		$z_{n+1}^{\pm} \Leftarrow z_{n+1}^{\pm,\tt PIM}$  \tcp*{accept the result}
    		$\tau_{n\!+\!1} \!\Leftarrow \! \tau_{n} \cdot \min \big\{ \!1.5, \max \big\{\!0.2, \kappa \big(\!\frac {\text{Tol}}{ \widehat{T}_{n+1}  }\!\big)^{1/3} \big\} \!\big\}$  
    		\tcp*{adjust step by \eqref{eq:improve-controller}}
    	}\Else
    	{
    		\tt{// adjust current step to recompute}  
    		$\tau_{n} \!\Leftarrow \! \tau_{n} \cdot \min \big\{ 1.5, \max \big\{0.2, \kappa \big(\frac {\text{Tol}}{ \widehat{T}_{n+1}  }\big)^{1/3} \big\} \big\}$ \;
    	}
\end{algorithm}

\section{Numerical Tests}
\label{sec:Num-Tests}
For all numerical tests, we use the Taylor-Hood($\mathbb{P}2-\mathbb{P}1$) finite element space for spatial discretization and software Freefem++ for programming. 
We apply the PIM algorithm \eqref{eq:CN-MHD} to all numerical tests in this section and compare it with another commonly-used algorithm: \\
\textit{The constant time-stepping second order BDF2-AB2 method \cite{MR3831959}} 
\begin{align}
    \begin{cases}
        \displaystyle \frac{3z^{\pm}_{n+1} \!-\! 4z^{\pm}_n + z^{\pm}_{n-1}}{\tau} \mp \big(B_\circ \!\cdot \!\nabla \big) z^{\pm}_{n+1}
        + \big( (2 z^{\mp}_{n} - z^{\mp}_{n-1}) \!\cdot\! \nabla \big) z^{\pm}_{n+1} \\
        \displaystyle \qquad \qquad \qquad - \nu^+ \Delta z^{\pm}_{n+1}
        - \nu^- \Delta \big( 2 z^{\mp}_{n} - z^{\mp}_{n-1} \big) 
        + \nabla p^{\pm}_{n+1} \!=\! 0, \\
        \displaystyle \qquad \qquad \nabla \cdot z^{\pm}_{n+1} = 0, 
    \end{cases}
    \label{eq:BDF2-AB2}
\end{align}
where $\tau > 0$ is the constant time step size.
The stopping criterion for the $k$-th iteration in \eqref{eq:BE-iter} is
\begin{align*}
    \big\|z^{\pm}_{(k)} - z^{\pm}_{(k - 1)}\big\| / \|z^{\pm}_{(k)}\| \leq {\tt{Tol}},
\end{align*}
where ${\tt{Tol}}$ is the required tolerance.

\subsection{Electrically Conducting 2D Travelling Wave}
In this section we verify the rate of convergence of constant time-stepping PIM algorithm on an electrically conducting two-dimensional travelling wave problem \cite{WiLaTr15} on the domain $\Omega = [0.5, 1.5]^2$.
The exact solutions (in Els\"asser variables) are
\begin{equation}
    \begin{split}
        z^+ & = \begin{bmatrix}
        \frac{3}{4} + \frac{1}{4} \cos(2\pi(x - t)) \sin(2\pi(x - t)) e^{- 8\pi^2 \nu t} + \frac{1}{10} (y + 1)^2 e^{\nu_m t} \\
        - \frac{1}{4} \sin(2\pi(x - t)) \cos(2\pi(x - t)) e^{- 8\pi^2 \nu t} + \frac{1}{10} (x + 1)^2 e^{\nu_m t}
        \end{bmatrix}, \\
        z^- & = \begin{bmatrix}
         \frac{3}{4} + \frac{1}{4} \cos(2\pi(x - t)) \sin(2\pi(x - t)) e^{- 8\pi^2 \nu t} - \frac{1}{10} (y + 1)^2 e^{\nu_m t} \\
        - \frac{1}{4} \sin(2\pi(x - t)) \cos(2\pi(x - t)) e^{- 8\pi^2 \nu t} - \frac{1}{10} (x + 1)^2 e^{\nu_m t}
        \end{bmatrix}, \\
        p & = -\frac{1}{64} \big( \cos(4\pi(x - t)) + \cos(4\pi(y - t)) \big) e^{- 16 \pi^2 \nu t}.
    \end{split}
    \label{2d_flow_true}
\end{equation}
We set $\nu = \nu_m = 2.5 \times 10^{-4}$, the time step $\Delta t = h$, and the required tolerance $\tt{Tol}=1.e-6$ for the convergence of the iterations, at each time step. We test over three values of $B_{\circ}$: $(0,0)$, $(1,1)$, $(10,10)$, and simulate the problem over time interval $[0,1]$.
The initial conditions and boundary conditions are given by the exact solutions. 
The results
of the PIM algorithm \eqref{eq:CN-MHD}, and the BDF2-AB2 algorithm \eqref{eq:BDF2-AB2} are presented in 
\Cref{table:2D-wave-B00,table:2D-wave-B11,table:2D-wave-B1010}.
Both algorithms have the second-order convergence rate in $L^{2}$-norm 
for $B_{\circ} = (1,1), (10,10)$. 
Moreover, the PIM algorithm \eqref{eq:CN-MHD} obtains smaller errors for small magnitude 
of $B_{\circ}$ while BDF2-AB2 algorithm \eqref{eq:BDF2-AB2} loses the order of accuracy for the special case of $B_{\circ} = (0,0)$.
        \begin{table}[H]
            \centering 
            \begin{tabular}{c c c c c c }
            \hline\hline 
            \multicolumn{6}{c}{PIM algorithm} \\
            \hline 
            $\Delta t = h$ & $\| z^+ - z^{+,h} \|_{\infty,0}$ & Rate & $\| z^- - z^{-,h} \|_{\infty,0}$ & Rate & Average iteration \\ [0.5ex] 
            \hline 
            $1/16$  & $1.6803\rm{e}-2$  &           & $2.5740\rm{e}-2$  &           & $9.31$ \\
            $1/32$  & $2.2706\rm{e}-3$  & $2.8876$  & $5.9990\rm{e}-3$  & $2.1012$  & $6.19$ \\ 
            $1/64$  & $5.3284\rm{e}-4$  & $2.0913$  & $1.4732\rm{e}-3$  & $2.0258$  & $4.12$ \\
            $1/128$ & $1.3197\rm{e}-4$  & $2.0135$  & $3.6695\rm{e}-4$  & $2.0053$  & $3.02$ \\
            \hline 
            \hline
            \multicolumn{6}{c}{BDF2-AB2 algorithm} \\
            \hline
            $\Delta t = h$ & $\| z^+ - z^{+,h} \|_{\infty,0}$ & Rate & $\| z^- - z^{-,h} \|_{\infty,0}$ & Rate \\ [0.5ex] 
            \hline 
            $1/16$  & $4.2726\rm{e}-2$  &           & $6.5633\rm{e}-2$  &          \\
            $1/32$  & $1.0623\rm{e}-2$  & $2.0079$  & $2.2256\rm{e}-2$  & $1.5602$ \\ 
            $1/64$  & $9.7834\rm{e}-3$  & $0.1188$  & $7.9360\rm{e}-3$  & $1.4877$ \\
            $1/128$ & $5.0295\rm{e}-3$  & $0.9600$  & $3.2317\rm{e}-3$  & $1.2961$ \\
            \hline\hline
            \end{tabular}
            \caption{Error and rate of convergence with $B_\circ = (0,0)$}
            \label{table:2D-wave-B00}
        \end{table}


        \begin{table}[H]
            \centering 
            \begin{tabular}{c c c c c c }
            \hline\hline 
            \multicolumn{6}{c}{PIM algorithm} \\
            \hline 
            $\Delta t = h$ & $\| z^+ - z^{+,h} \|_{\infty,0}$ & Rate & $\| z^- - z^{-,h} \|_{\infty,0}$ & Rate & Average iteration \\ [0.5ex] 
            \hline 
            $1/16$  & $1.0656\rm{e}-2$  &           & $1.5860\rm{e}-2$  &          & $9.44$ \\
            $1/32$  & $2.0089\rm{e}-3$  & $2.4072$  & $3.3290\rm{e}-3$  & $2.2523$ & $6.09$ \\ 
            $1/64$  & $5.1407\rm{e}-4$  & $1.9663$  & $7.9070\rm{e}-4$  & $2.0739$ & $4.05$ \\
            $1/128$ & $1.2962\rm{e}-4$  & $1.9877$  & $1.9466\rm{e}-4$  & $2.0739$ & $3.01$ \\
            \hline 
            \hline
            \multicolumn{6}{c}{BDF2-AB2 algorithm} \\
            \hline
            $\Delta t = h$ & $\| z^+ - z^{+,h} \|_{\infty,0}$ & Rate & $\| z^- - z^{-,h} \|_{\infty,0}$ & Rate \\ [0.5ex] 
            \hline 
            $1/16$  & $2.5927\rm{e}-2$  &           & $4.0365\rm{e}-2$  &           \\
            $1/32$  & $5.1292\rm{e}-3$  & $2.3376$  & $1.0111\rm{e}-2$  & $1.9972$  \\ 
            $1/64$  & $1.2879\rm{e}-3$  & $1.9937$  & $2.4751\rm{e}-3$  & $2.0304$  \\
            $1/128$ & $3.2148\rm{e}-4$  & $2.0023$  & $6.0824\rm{e}-4$  & $2.0247$  \\
            \hline\hline
            \end{tabular}
            \caption{Error and rate of convergence with $B_\circ = (1,1)$}
            \label{table:2D-wave-B11}
        \end{table}


        \begin{table}[H]
            \centering 
            \begin{tabular}{c c c c c c }
            \hline\hline 
            \multicolumn{6}{c}{PIM algorithm} \\
            \hline 
            $\Delta t = h$ & $\| z^+ - z^{+,h} \|_{\infty,0}$ & Rate & $\| z^- - z^{-,h} \|_{\infty,0}$ & Rate & Average iteration \\ [0.5ex] 
            \hline 
            $1/16$  & $4.0128\rm{e}-3$  &           & $3.7020\rm{e}-3$  &           & $7.00$    \\
            $1/32$  & $1.0391\rm{e}-3$  & $2.0952$  & $1.0789\rm{e}-3$  & $2.1519$  & $5.91$ \\ 
            $1/64$  & $3.6672\rm{e}-4$  & $2.2101$  & $3.5414\rm{e}-4$  & $2.1153$  & $3.75$ \\
            $1/128$ & $1.1268\rm{e}-4$  & $1.9066$  & $1.1073\rm{e}-4$  & $1.8706$  & $3.02$ \\
            \hline 
            \hline
            \multicolumn{6}{c}{BDF2-AB2 algorithm} \\
            \hline
            $\Delta t = h$ & $\| z^+ - z^{+,h} \|_{\infty,0}$ & Rate & $\| z^- - z^{-,h} \|_{\infty,0}$ & Rate \\ [0.5ex] 
            \hline 
            $1/16$  & $7.2273\rm{e}-3$  &           & $5.6705\rm{e}-3$  &           \\
            $1/32$  & $1.0877\rm{e}-3$  & $2.7322$  & $1.0801\rm{e}-3$  & $2.3923$  \\ 
            $1/64$  & $1.9431\rm{e}-4$  & $2.4849$  & $2.2144\rm{e}-4$  & $2.2862$  \\
            $1/128$ & $4.7211\rm{e}-5$  & $2.0412$  & $5.5349\rm{e}-5$  & $2.0003$  \\
            \hline\hline
            \end{tabular}
            \caption{Error and rate of convergence with $B_\circ = (10,10)$}
            \label{table:2D-wave-B1010}
        \end{table}


\subsection{Hartmann Flows
}
\label{subsec:Hartmann}
The 2D Hartmann Flows problem on the domain $\Omega = [0,L] \times [-1,1]$ has the following exact solutions \cite{MR2289481}
\begin{align}
    u &= (u_1(y),0), \qquad B = (b_1(y),M), \qquad B_\circ = (0,M) ,
    \label{eq:Hart_Flow} \\
    u_1(y) & = \frac{G}{\nu \cdot Ha \cdot \tanh(Ha)} \Big(1 - \frac{\cosh(y \cdot Ha)}{\cosh(Ha)} \Big), \notag  \\
    b_1(y) & = \frac{G}{S} \Big( \frac{\sinh(y \cdot Ha)}{\sinh(Ha)} - y \Big),\qquad 
    p(x,y) = - G x - \frac{1}{2} S b_1^2(y), \notag 
\end{align}
where $M>0$ is the parameter of external magnetic field, $Ha>0$ is the Hartmann number and $L,G,S>0$ are other parameters of the problems. 
We first verify the convergence rate of PIM algorithm by setting $L = 6$, $G = S = 1$, $Ha = 5$, $\nu = 0.1 = \nu_{m}$ and required tolerance $\tt{Tol} = 1.\rm{e}-6$ for the fixed point iteration.
It is easy to check that the MHD equations in \eqref{eq:NSE-MHD} with steady solutions 
in \eqref{eq:Hart_Flow} is a conservative system. 
We simulate the problem by the constant time-stepping PIM algorithm with four values of $B_{\circ}$: $(0,1)$, $(0,10)$, $(0,100)$, $(0,1000)$ over the time interval $[0,1]$ to confirm the rate of convergence.
The two initial solutions and boundary conditions are 
compatible with 
the exact solutions in \eqref{eq:Hart_Flow}.
From \Cref{table:Hartmann-PIM}, we observe that the constant time-stepping PIM algorithm has achieved a super convergence rate in $L^{2}$-norm. 
For the long-time conservation, we compare the PIM algorithm
with the
 BDF2-AB2 algorithm in \eqref{eq:BDF2-AB2} 
over the time interval $[0,10]$.
We set the diameter $h = L/128$, and the constant time step $\Delta t = 1/128$. 
From \Cref{fig:ErrorENB1,fig:ErrorENB10}, 
we see that 
the performance of the BDF2-AB2 algorithm is 
slightly 
better than that of the PIM algorithm for a small value of $M$ ($M=1, 10$) in terms of kinetic energy $\mathcal{E}_{n}$.
However, as the value of $M$ increases, we find that the results of the two algorithms are very close
: the error of the PIM algorithm oscillates around that of the BDF2-AB2 algorithm in \Cref{fig:ErrorENB100,fig:ErrorENB1000}. 
We observe similar situations for the error of cross helicity $\mathcal{H}_{C_{n}}$ 
in \Cref{fig:ErrorCH_Const}.

\begin{table}[H]
    \centering 
    \begin{tabular}{c c c c c c }
    \hline\hline 
    \multicolumn{6}{c}{$B_{\circ} = (0,1)$} \\
    \hline 
    $\Delta t = h$ & $\| z^+ - z^{+,h} \|_{\infty,0}$ & Rate & $\| z^- - z^{-,h} \|_{\infty,0}$ & Rate & Average iteration \\ [0.5ex] 
    \hline 
    $1/16$  & $3.5146\rm{e}-3$  &           & $3.5166\rm{e}-3$  &           & $2.87$ \\
    $1/32$  & $4.2923\rm{e}-4$  & $3.0336$  & $4.2927\rm{e}-4$  & $3.0343$  & $2$ \\ 
    $1/64$  & $5.1377\rm{e}-5$  & $3.0625$  & $5.1377\rm{e}-5$  & $3.0627$  & $1.37$ \\
    $1/128$ & $5.7723\rm{e}-5$  & $3.2866$  & $6.0922\rm{e}-5$  & $3.3215$  & $1$ \\
    \hline 
    \hline
    \multicolumn{6}{c}{$B_{\circ} = (0,10)$} \\
    \hline
    $\Delta t = h$ & $\| z^+ - z^{+,h} \|_{\infty,0}$ & Rate & $\| z^- - z^{-,h} \|_{\infty,0}$ & Rate & Average iteration \\ [0.5ex] 
    \hline 
    $1/16$  & $8.8998\rm{e}-3$  &           & $8.9317\rm{e}-3$  &           & $3$ \\
    $1/32$  & $7.8380\rm{e}-4$  & $3.5052$  & $7.8431\rm{e}-4$  & $3.5094$  & $2$ \\ 
    $1/64$  & $6.7809\rm{e}-5$  & $3.5309$  & $6.7806\rm{e}-5$  & $3.5319$  & $1.89$ \\
    $1/128$ & $6.8948\rm{e}-6$  & $3.2979$  & $6.8948\rm{e}-6$  & $3.2978$  & $1$ \\
    \hline\hline
    \multicolumn{6}{c}{$B_{\circ} = (0,100)$} \\
    \hline
    $\Delta t = h$ & $\| z^+ - z^{+,h} \|_{\infty,0}$ & Rate & $\| z^- - z^{-,h} \|_{\infty,0}$ & Rate & Average iteration \\ [0.5ex] 
    \hline 
    $1/16$  & $3.5855\rm{e}-2$  &           & $3.5794\rm{e}-2$  &           & $3.07$ \\
    $1/32$  & $4.9008\rm{e}-3$  & $2.8711$  & $4.9041\rm{e}-3$  & $2.8676$  & $2.19$ \\ 
    $1/64$  & $4.4118\rm{e}-4$  & $3.4736$  & $4.4140\rm{e}-4$  & $3.4738$  & $2$ \\
    $1/128$ & $3.1643\rm{e}-5$  & $3.8014$  & $3.1652\rm{e}-5$  & $3.8017$  & $2$ \\
    \hline\hline 
    \multicolumn{6}{c}{$B_{\circ} = (0,1000)$} \\
    \hline
    $\Delta t = h$ & $\| z^+ - z^{+,h} \|_{\infty,0}$ & Rate & $\| z^- - z^{-,h} \|_{\infty,0}$ & Rate & Average iteration \\ [0.5ex] 
    \hline 
    $1/16$  & $6.5272\rm{e}-2$  &           & $6.5278\rm{e}-2$ &           & $3$ \\
    $1/32$  & $1.2988\rm{e}-2$  & $2.3293$  & $1.2982\rm{e}-2$ & $2.3301$  & $3$ \\ 
    $1/64$  & $2.0417\rm{e}-3$  & $2.6693$  & $2.0417\rm{e}-3$ & $2.6686$  & $2$ \\
    $1/128$ & $2.5951\rm{e}-4$  & $2.9759$  & $2.5959\rm{e}-4$ & $2.9755$  & $2$ \\
    \hline\hline 
    \end{tabular}
    \caption{Error and rate of convergence for PIM}
    \label{table:Hartmann-PIM}
\end{table}

\begin{confidential}
    \color{darkblue}
    \begin{table}[H]
        \centering 
        \begin{tabular}{c c c c c }
        \hline\hline 
        \multicolumn{5}{c}{$B_{\circ} = (0,1)$} \\
        \hline 
        $\Delta t = h$ & $\| z^+ - z^{+,h} \|_{\infty,0}$ & Rate & $\| z^- - z^{-,h} \|_{\infty,0}$ & Rate  \\ [0.5ex] 
        \hline 
        $1/16$  & $3.3159\rm{e}-3$  &           & $3.3176\rm{e}-3$  &            \\
        $1/32$  & $4.0882\rm{e}-4$  & $3.0199$  & $4.0883\rm{e}-4$  & $3.0206$   \\ 
        $1/64$  & $5.0317\rm{e}-5$  & $3.0223$  & $5.0317\rm{e}-5$  & $3.0224$   \\
        $1/128$ & $6.2518\rm{e}-6$  & $3.0087$  & $6.2518\rm{e}-6$  & $3.0087$   \\
        \hline 
        \hline
        \multicolumn{5}{c}{$B_{\circ} = (0,10)$} \\
        \hline
        $\Delta t = h$ & $\| z^+ - z^{+,h} \|_{\infty,0}$ & Rate & $\| z^- - z^{-,h} \|_{\infty,0}$ & Rate  \\ [0.5ex] 
        \hline 
        $1/16$  & $7.0822\rm{e}-3$  &           & $7.1115\rm{e}-3$  &            \\
        $1/32$  & $6.7629\rm{e}-4$  & $3.3885$  & $6.7685\rm{e}-4$  & $3.3933$   \\ 
        $1/64$  & $6.3002\rm{e}-5$  & $3.4242$  & $6.3000\rm{e}-5$  & $3.4254$   \\
        $1/128$ & $6.7196\rm{e}-6$  & $3.2290$  & $6.7195\rm{e}-6$  & $3.2289$   \\
        \hline\hline
        \multicolumn{5}{c}{$B_{\circ} = (0,100)$} \\
        \hline
        $\Delta t = h$ & $\| z^+ - z^{+,h} \|_{\infty,0}$ & Rate & $\| z^- - z^{-,h} \|_{\infty,0}$ & Rate  \\ [0.5ex] 
        \hline 
        $1/16$  & $1.9757\rm{e}-2$  &           & $1.9707\rm{e}-2$  &            \\
        $1/32$  & $2.9637\rm{e}-3$  & $2.7369$  & $2.9674\rm{e}-3$  & $2.7314$   \\ 
        $1/64$  & $3.0515\rm{e}-4$  & $3.2798$  & $3.0570\rm{e}-4$  & $3.2790$   \\
        $1/128$ & $2.4215\rm{e}-5$  & $3.6556$  & $2.4222\rm{e}-5$  & $3.6577$   \\
        \hline\hline 
        \multicolumn{5}{c}{$B_{\circ} = (0,1000)$} \\
        \hline
        $\Delta t = h$ & $\| z^+ - z^{+,h} \|_{\infty,0}$ & Rate & $\| z^- - z^{-,h} \|_{\infty,0}$ & Rate  \\ [0.5ex] 
        \hline 
        $1/16$  & $3.4083\rm{e}-2$  &           & $3.4119\rm{e}-2$ &            \\
        $1/32$  & $6.7094\rm{e}-3$  & $2.3448$  & $6.7064\rm{e}-3$ & $2.3469$   \\ 
        $1/64$  & $1.0632\rm{e}-3$  & $2.6578$  & $1.0631\rm{e}-3$ & $2.6573$   \\
        $1/128$ & $1.4217\rm{e}-4$  & $2.9027$  & $1.4223\rm{e}-4$ & $2.9019$   \\
        \hline\hline 
        \end{tabular}
        \caption{Error and rate of convergence for BDF2}
        \label{table:Hartmann-BDF2}
    \end{table}

    \normalcolor
\end{confidential}

\begin{figure}
    \centering
    \subfigure[Error of $\mathcal{E}_{n} = \frac{1}{2} \big( \| u \|^{2} + \| B \|^{2} \big)$ with $B_{\circ}=(0,1)$]{\label{fig:ErrorENB1}
        \includegraphics[width=5.0in,height=1.30in]{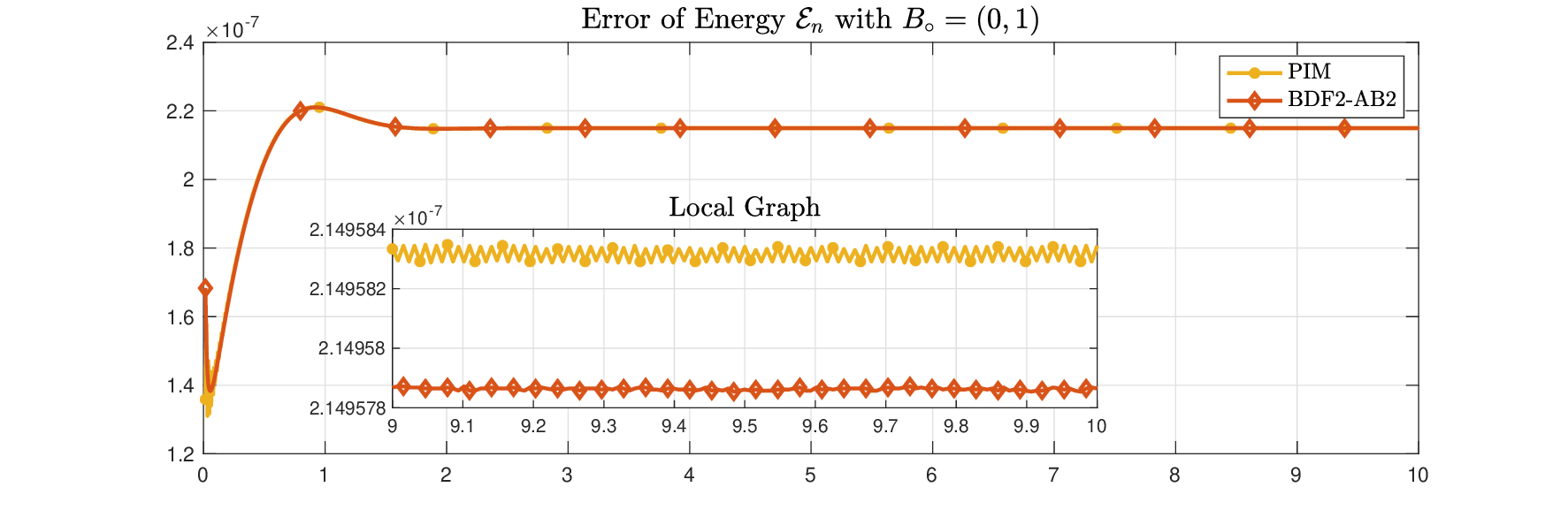}
    } \\ \vspace{0.002cm}
    \subfigure[Error of $\mathcal{E}_{n} = \frac{1}{2} \big( \| u \|^{2} + \| B \|^{2} \big)$ with $B_{\circ}=(0,10)$]{\label{fig:ErrorENB10}
        \includegraphics[width=5.0in,height=1.30in]{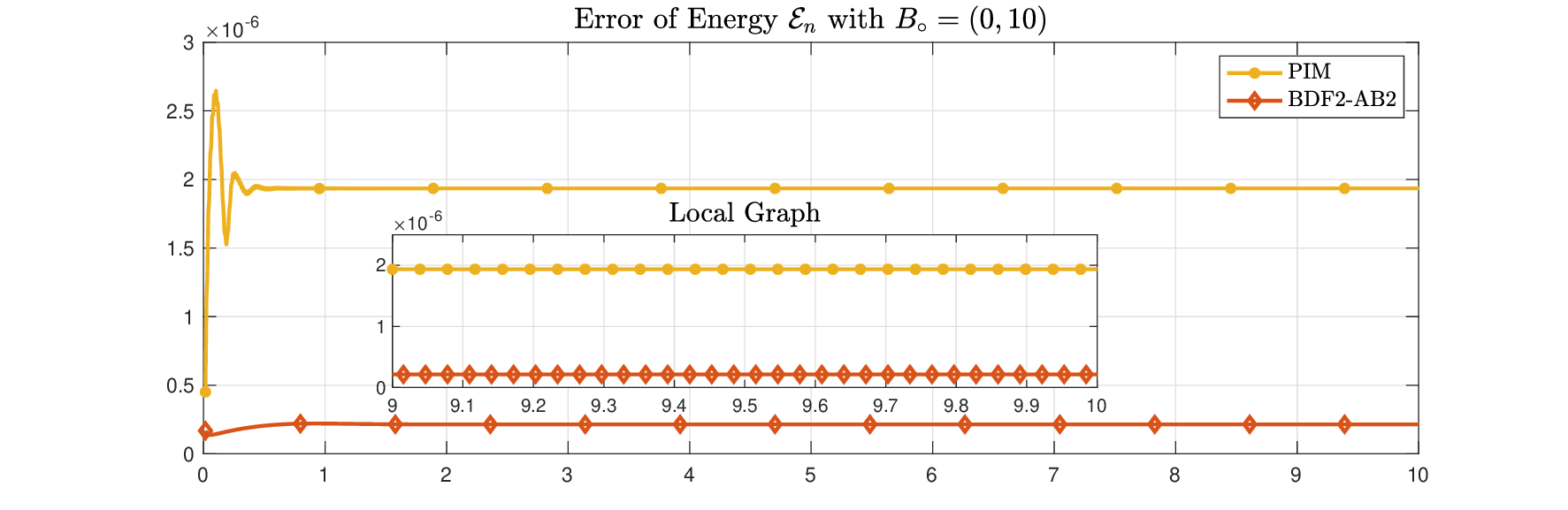}
    } \\ \vspace{0.002cm}
    \subfigure[Error of $\mathcal{E}_{n} = \frac{1}{2} \big( \| u \|^{2} + \| B \|^{2} \big)$ with $B_{\circ}=(0,100)$]{\label{fig:ErrorENB100}
        \includegraphics[width=5.0in,height=1.30in]{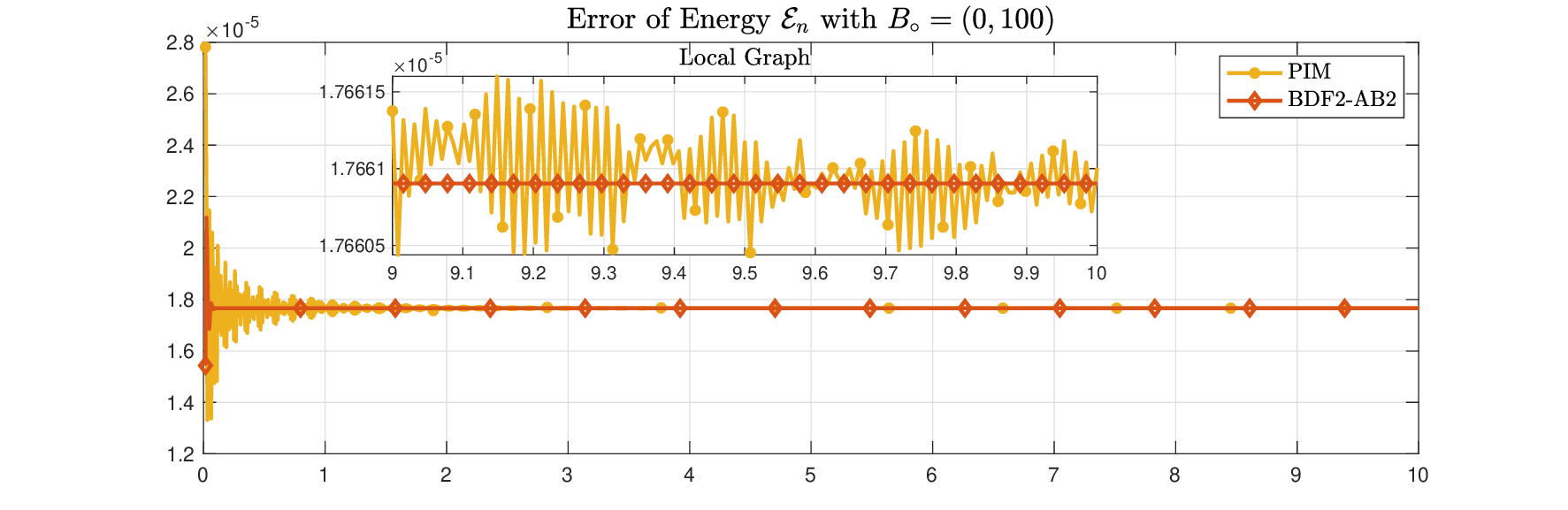}
    } \\ \vspace{0.002cm}
    \subfigure[Error of $\mathcal{E}_{n} = \frac{1}{2} \big( \| u \|^{2} + \| B \|^{2} \big)$ with $B_{\circ}=(0,1000)$]{\label{fig:ErrorENB1000}
        \includegraphics[width=5.0in,height=1.30in]{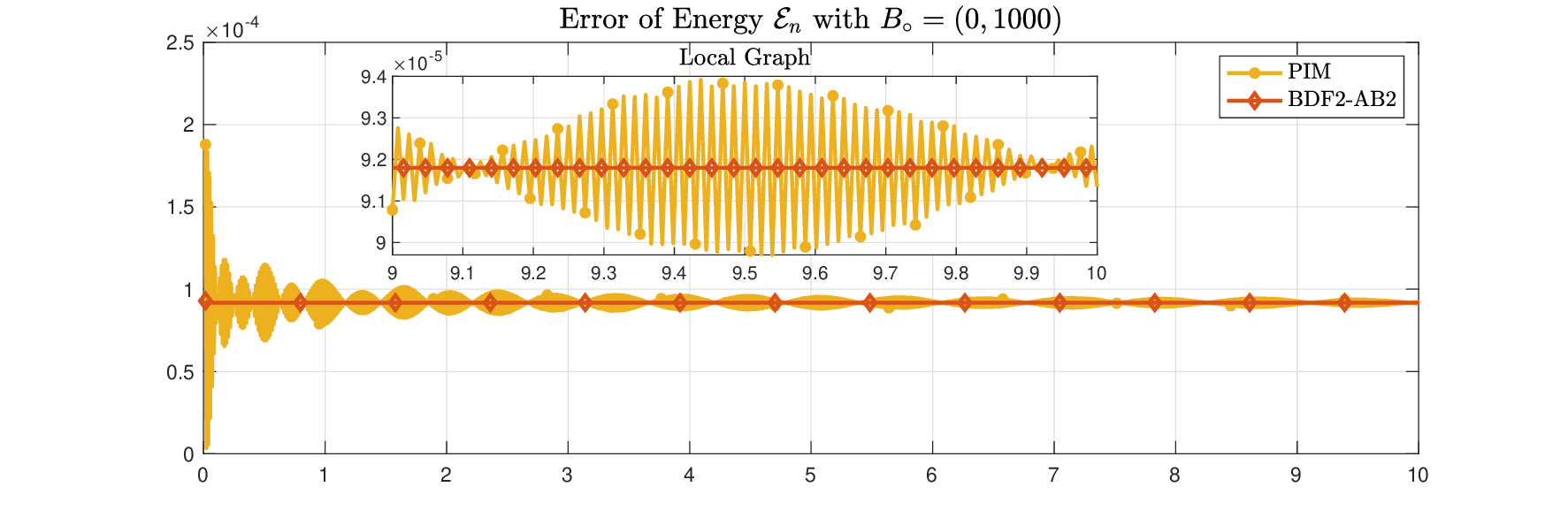}
    }
    \vspace{-0.3cm}
    \caption{The performance of the BDF2-AB2 algorithm is a 
    slightly
    better than that of the PIM algorithm for a small value of $M$ ($M=1, 10$) in terms of kinetic energy $\mathcal{E}_{n}$.
    However, as the value of $M$ increases, we find that the results of the two algorithms are very close
    : the error of the PIM algorithm oscillates around that of the BDF2-AB2 algorithm.
    }
    \label{fig:ErrorEN_Const}
\end{figure}

\begin{figure}
    \centering
    \subfigure[Error of $\mathcal{H}_{C_n}$ with $B_{\circ}=(0,1)$]{\label{fig:ErrorCHB1}
        \includegraphics[width=5.0in,height=1.30in]{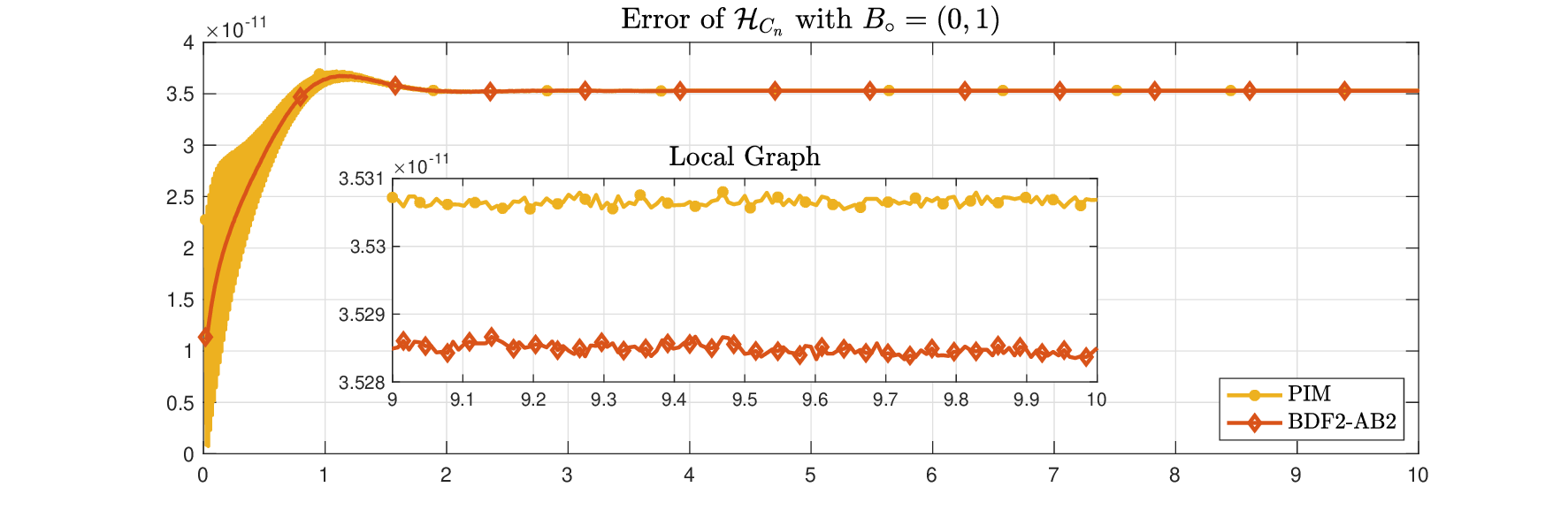}
    } \\ \vspace{0.002cm}
    \subfigure[Error of $\mathcal{H}_{C_n}$ with $B_{\circ}=(0,10)$]{\label{fig:ErrorCHB10}
        \includegraphics[width=5.0in,height=1.30in]{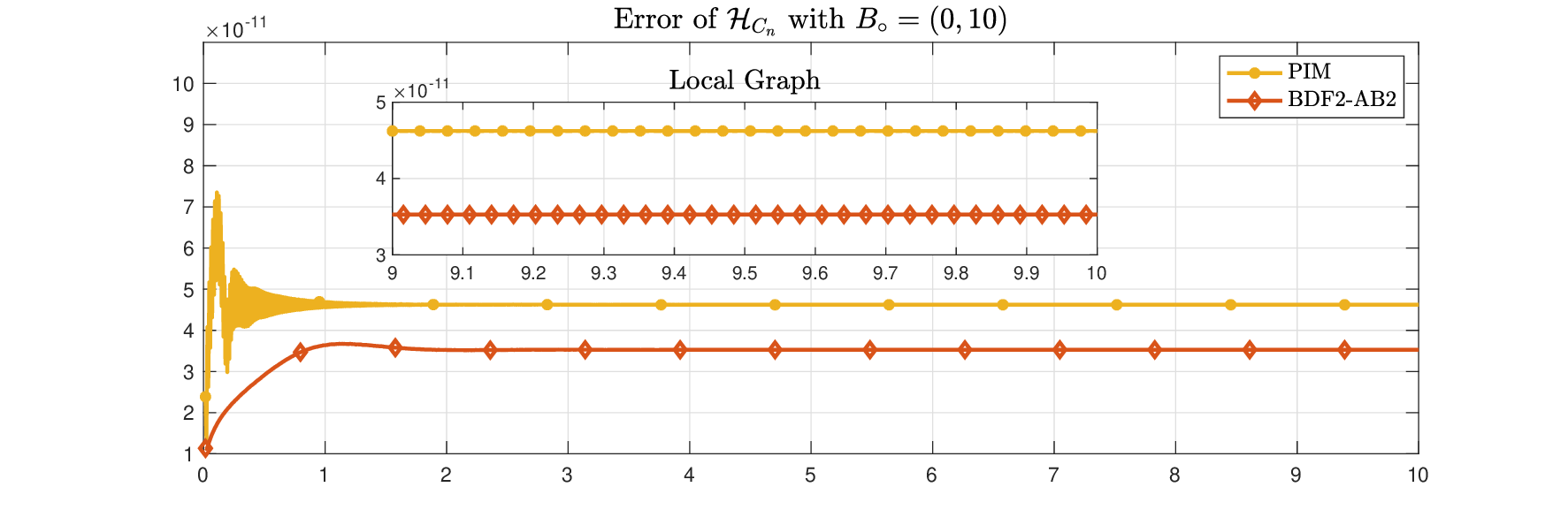}
    } \\ \vspace{0.002cm}
    \subfigure[Error of $\mathcal{H}_{C_n}$ with $B_{\circ}=(0,100)$]{\label{fig:ErrorCHB100}
        \includegraphics[width=5.0in,height=1.30in]{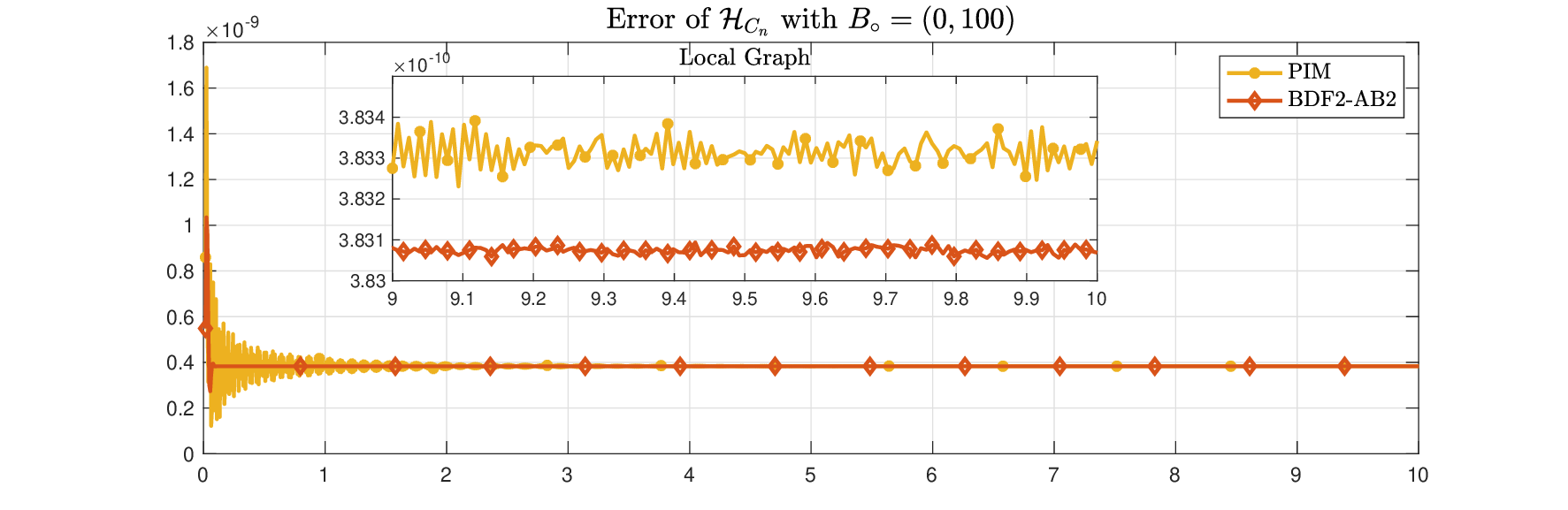}
    } \\ \vspace{0.002cm}
    \subfigure[Error of $\mathcal{H}_{C_n}$ with $B_{\circ}=(0,1000)$]{\label{fig:ErrorCHB1000}
        \includegraphics[width=5.0in,height=1.30in]{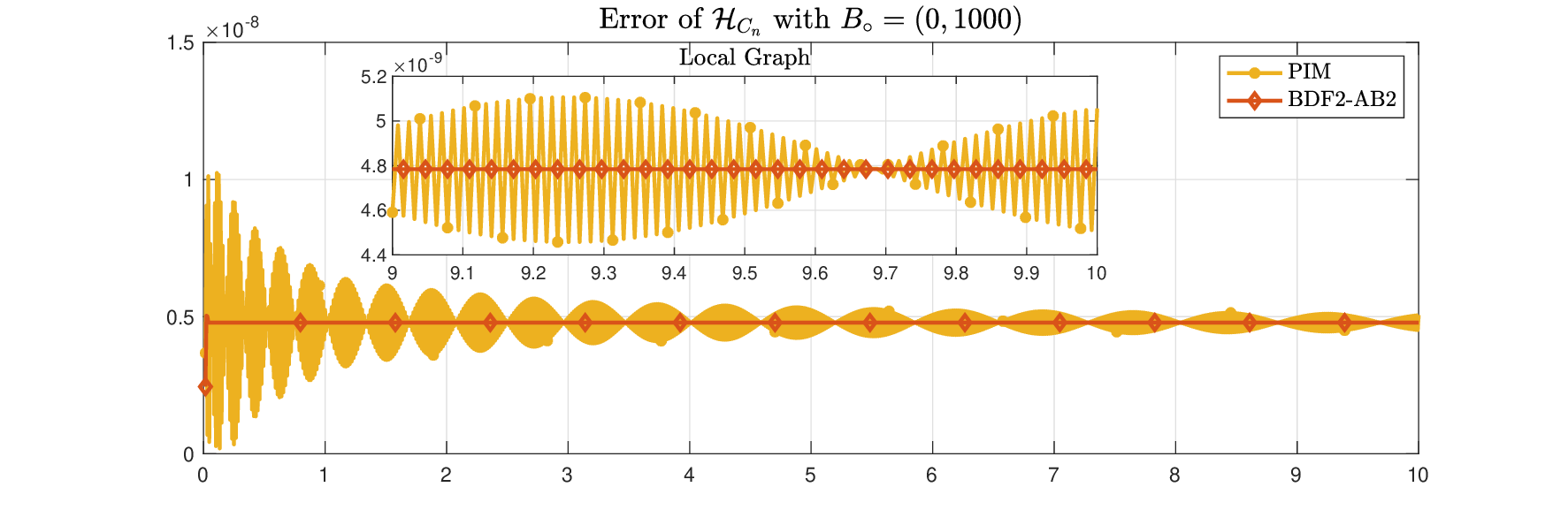}
    }
    \vspace{-0.3cm}
    \caption{The error 
    in the cross helicity $\mathcal{H}_{C_{n}}$ 
    behaves similarly
    to 
    the errors in the kinetic energy, for both algorithms.
    }
    \label{fig:ErrorCH_Const}
\end{figure}

\subsection{Adaptive Test}
In this subsection, we will show that the adaptive algorithm (Algorithm \eqref{alg:Adap-DLN}) has 
certain
advantages over the corresponding constant time-stepping algorithm in one highly stiff test problems. 
We revise the Hartmann flow problem in \eqref{eq:Hart_Flow} by multiplying 
the exact solution \eqref{eq:Hart_Flow}
with the following time 
dependent 
function 
\begin{align*}
    & G(t) = e^{g_{1}(t)} \big[ \cos(g_{2}(t)) + \sin(g_{2}(t)) \big], \qquad
    t_{0} \leq t \leq T, \\
    &\begin{cases}
        g_{1}(t) = 10^{\omega} ( t + 2 e^{-t} -2) , \\
        g_{2}(t) = 10^{\omega} ( 1 - e^{-t} - te^{-1}) .
    \end{cases}
\end{align*}
Here $G(t)$ is the first component of the solution vector to an extremely stiff ordinary differential system proposed by Lindberg \cite{Lin74_BIT}.  
\begin{figure}[h]
    \includegraphics[width=13cm,height=5cm]{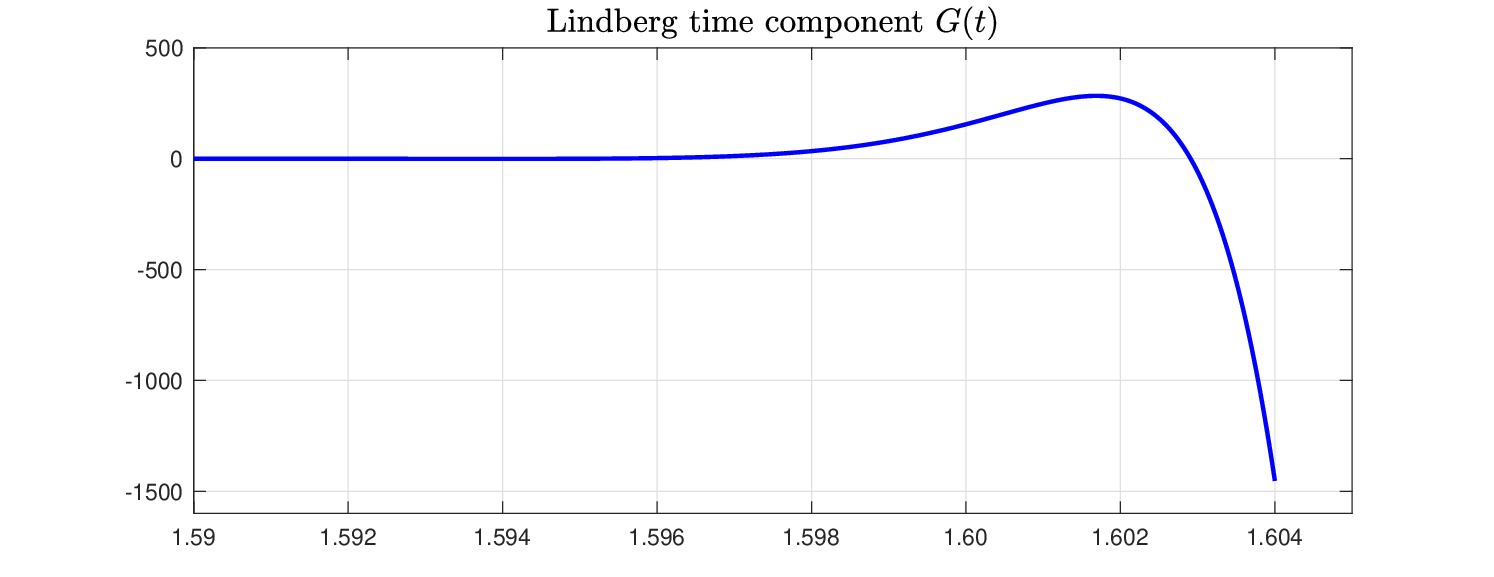}
    \caption{Time component function proposed by Lindberg}
    \label{fig:Time-component}
\end{figure}
The exact solutions of the revised Hartmann flow problem on $\Omega = [0,L] \times [-1,1]$ become
\begin{align}
    u &= (u_1(y),0), \qquad B = (b_1(y),M), \qquad B_\circ = (0,M) ,
    \label{eq:Hart_Flow_Unsteady} \\
    u_1(y) & = \frac{G}{\nu \cdot Ha \cdot \tanh(Ha)} \Big(1 - \frac{\cosh(y \cdot Ha)}{\cosh(Ha)} \Big) G(t), \notag  \\
    b_1(y) & = \frac{G}{S} \Big( \frac{\sinh(y \cdot Ha)}{\sinh(Ha)} - y \Big)G(t) ,\qquad 
    p(x,y) = \Big( - G x - \frac{1}{2} S b_1^2(y) \Big) G(t). \notag 
\end{align}
We set the parameters $L = 6$, $G = S = 1$, $Ha = 5$, $M = 100$, $\nu = 0.1 = \nu_{m}$, $\omega = 3.1$ and simulate the problem 
over the time interval $[1.59, 1.604]$. 
\Cref{fig:Time-component} displays the extreme stiffness of $G(t)$:
$G(t)$ increases rapidly from 0 ($t = 1.596$) to 300 ($t = 1.6015$) and then declines sharply to $-1400$ ($t = 1.604$). 
We set the diameter $h = L/80$,  required tolerance 
${\tt{Tol}} = 1.\rm{e}-4$ and the safety factor $\kappa = 0.95$.
The maximum time step is  
$\tau_{\rm{max}} = 1.\rm{e}-4$ for stability and the minimum time step $\tau_{\rm{min}} = 1.\rm{e}-6$ for efficiency.
Two initial time steps are selected to be the same as $\tau_{\rm{min}}$ for initial time accuracy. 
Initial conditions, boundary conditions and source functions are 
derived from the 
exact solutions. 
We verify the accuracy of the adaptive algorithm (Algorithm \eqref{alg:Adap-DLN}) by computing the error of energy 
in $u,B$
and the error of energy in the Els\"asser variables.
We also apply the corresponding constant time-stepping algorithms with the same number of time steps (342 steps) to confirm the efficiency of the time adaptivity in extremely stiff problems.

From \Cref{fig:KEPlot,fig:KEessPlot}, we observe that both algorithms obtain the true pattern in 
kinetic energy (KE) and kinetic energy in Els\"asser variables, even on the exceedingly stiff  regime time range ($t \in [1.603, 1.604]$). 
Meanwhile, \Cref{fig:CHPlot} shows that the cross helicity of 
both algorithms keeps at a higher level but displays an increasing trend, which is consistent with that of the exact solutions. 
\Cref{fig:Adapt_Error} confirms that the adaptive algorithm outperforms the corresponding constant time-stepping algorithm due to a much smaller error in kinetic energy at the end ($t = 1.604$). 
\Cref{fig:LTEPlot,fig:StepPlot} demonstrate that the adaptive algorithm (Algorithm \eqref{alg:Adap-DLN}) 
balances the efficiency and accuracy: the estimator of LTE $\widehat{T}_{n+1}$
at most time steps 
is 
close to but below the required tolerance ${\tt{Tol = 1.e-4}}$, 
so that  $k_{n}$ at most time steps 
is 
above $\tau_{\rm{min}}$. 
The adaptive algorithm (Algorithm \eqref{alg:Adap-DLN}) assigns more time steps to the highly stiff part ($t \geq 1.602$), which results in a smaller error in the kinetic energy 
at the end ($t = 1.604$).
From \Cref{fig:IterNumPlot}, both algorithms have relatively high computational costs on the extremely stiff time range ($t \in [1.603, 1.604]$) since the number of iterations at each time step increases from 2 to 6. 

\begin{figure}[ptbh]
    \centering
    \subfigure[log-plot of KE]{\label{fig:KEPlot}
        \includegraphics[width=5.0in,height=1.30in]{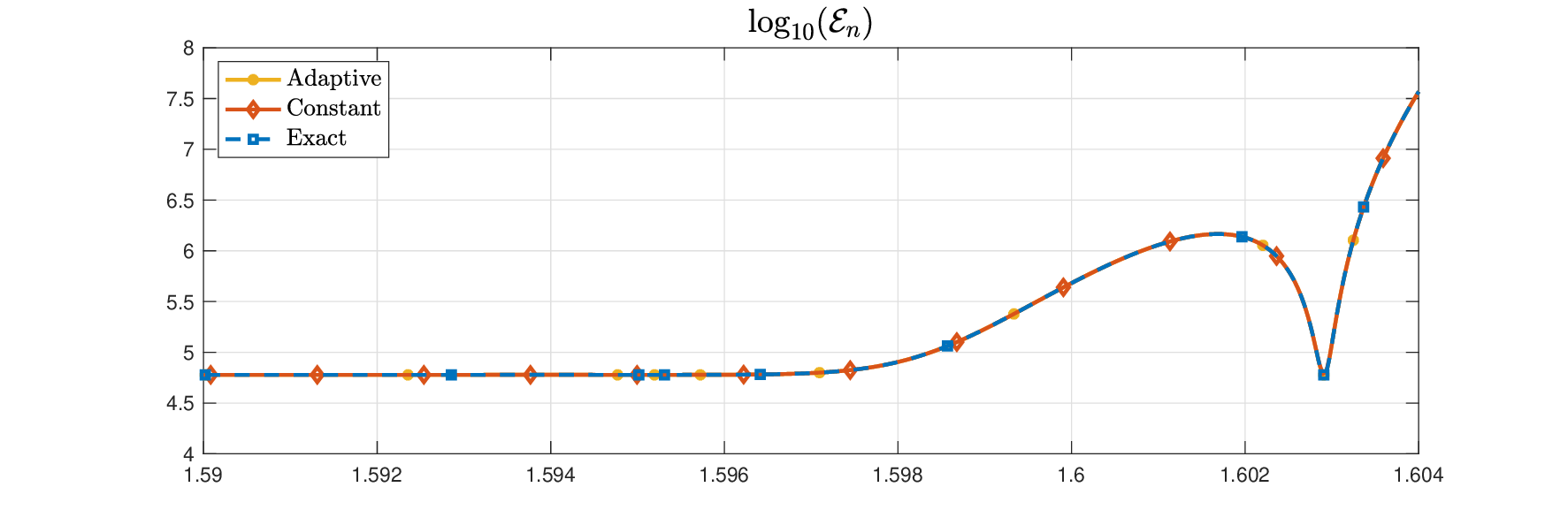}
    } \\ \vspace{0.005cm}
    \subfigure[log-plot of KE in Els\"asser variables]{\label{fig:KEessPlot}
        \includegraphics[width=5.0in,height=1.30in]{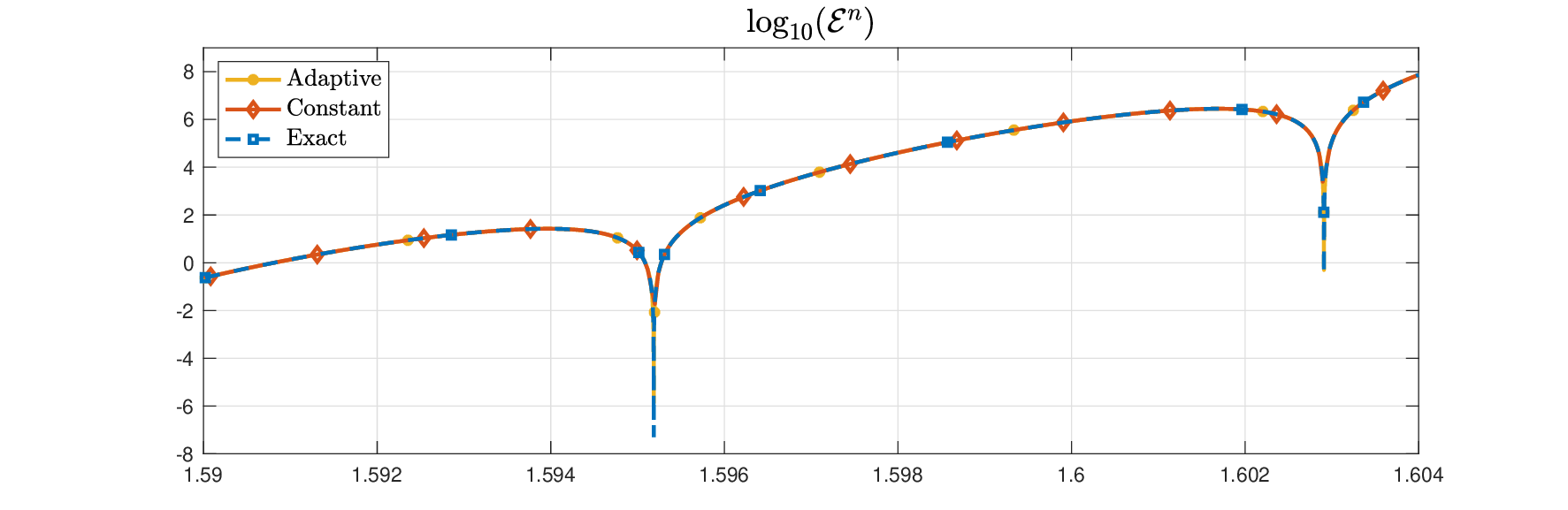}
    } \\ \vspace{0.005cm}
    \subfigure[log-plot of cross helicity]{\label{fig:CHPlot}
        \includegraphics[width=5.0in,height=1.30in]{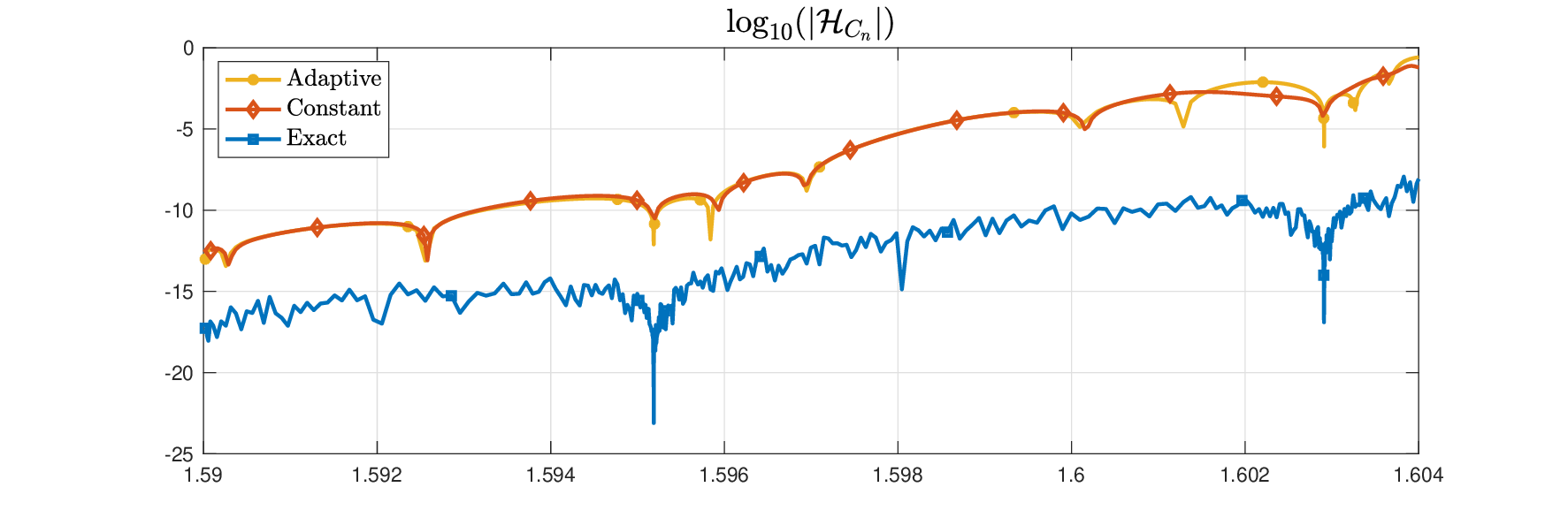}
    }
    \vspace{-0.3cm}
    \caption{Both algorithms obtain the true pattern in 
    kinetic energy (KE) and kinetic energy in Els\"asser variables, even on the extremely stiff time range. 
    Meanwhile cross helicity of 
    both algorithms keep at a higher level but display an increasing trend, which is consistent with that of the exact solutions.
    }
    \label{fig:Adapt_KECH}
\end{figure}

\begin{figure}[ptbh]
    \subfigure[Error of KE]{ \label{fig:ErrorENPlot}
        \hspace{-0.2cm}
        \begin{minipage}[t]{0.45\linewidth}
            \centering
            \includegraphics[width=2.6in,height=1.6in]{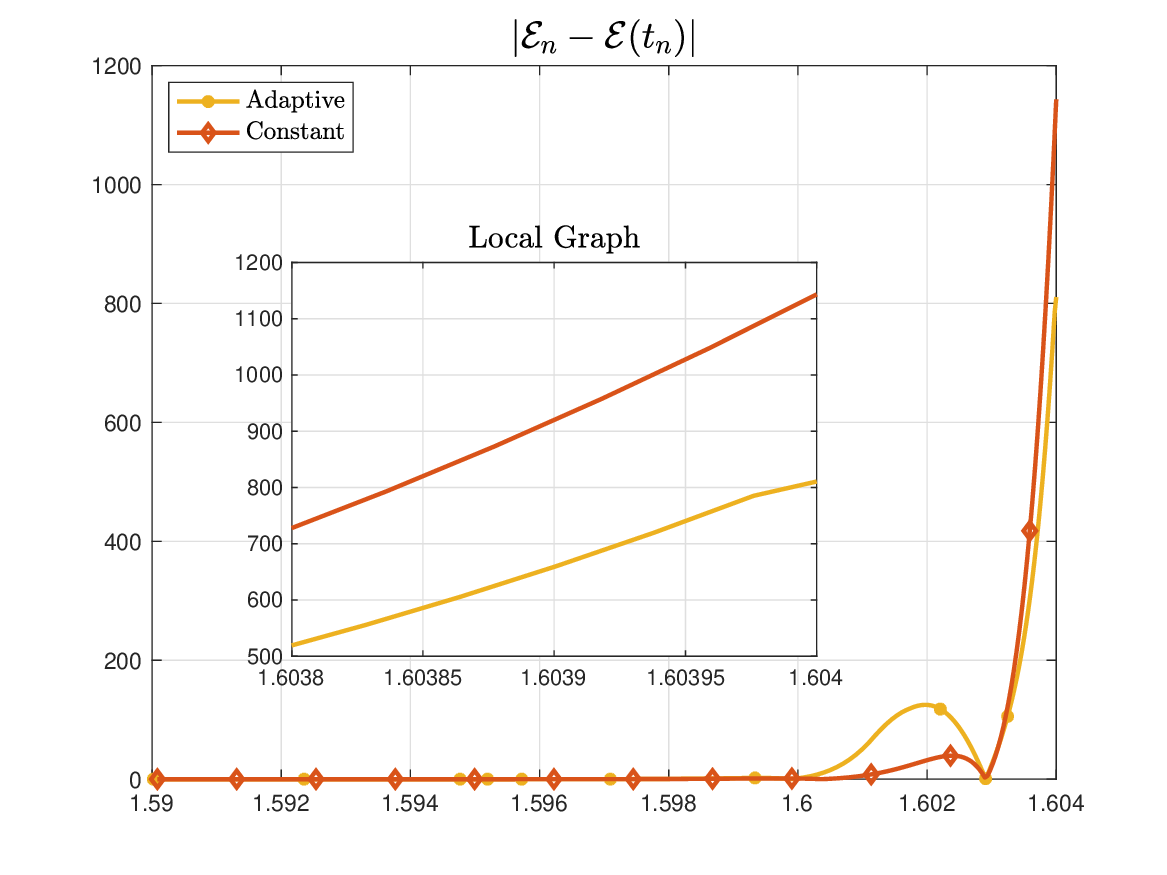}\\
        \end{minipage}
        \quad}%
    \subfigure[Error of KE in Els\"asser variables]{ \label{fig:ErrorENessPLot}
    \hspace{-0.2cm}
        \begin{minipage}[t]{0.45\linewidth}
            \centering
            \includegraphics[width=2.6in,height=1.6in]{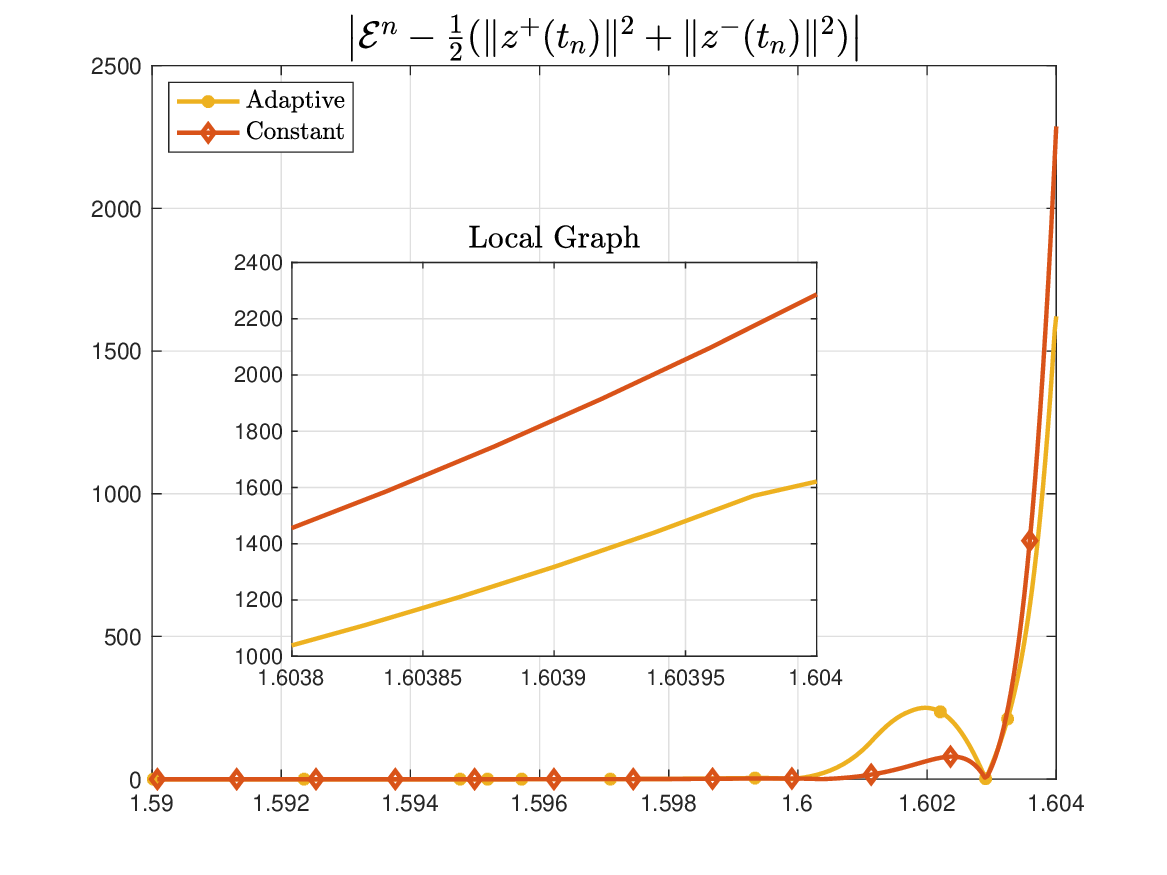}\\
    \end{minipage}}  	
    \par
    \centering
    \vspace{-0.3cm}
    \caption{The adaptive algorithm outperforms the corresponding constant time-stepping algorithm due to much smaller error in kinetic energy at the end ($t = 1.604$).
}
    \label{fig:Adapt_Error}
\end{figure}

\begin{figure}[ptbh]
    \centering
    \subfigure[Estimator of LTE]{\label{fig:LTEPlot}
        \includegraphics[width=5.0in,height=1.3in]{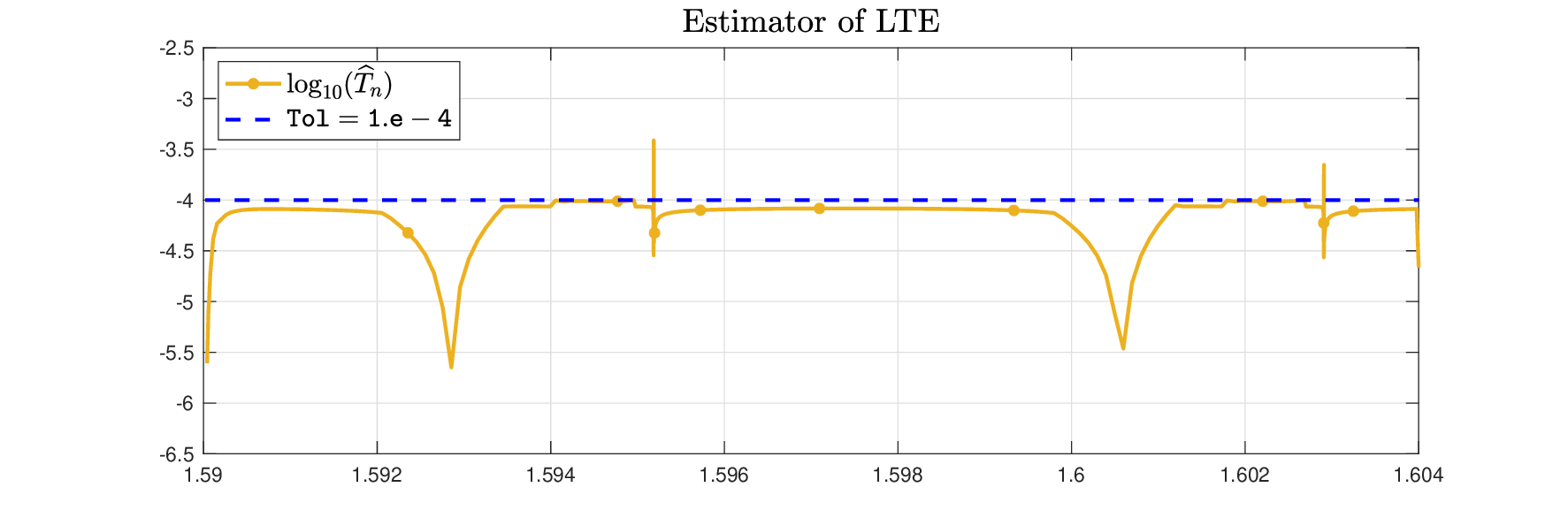}
    } \\ \vspace{0.01cm}
    \subfigure[Time step size]{\label{fig:StepPlot}
        \includegraphics[width=5.0in,height=1.3in]{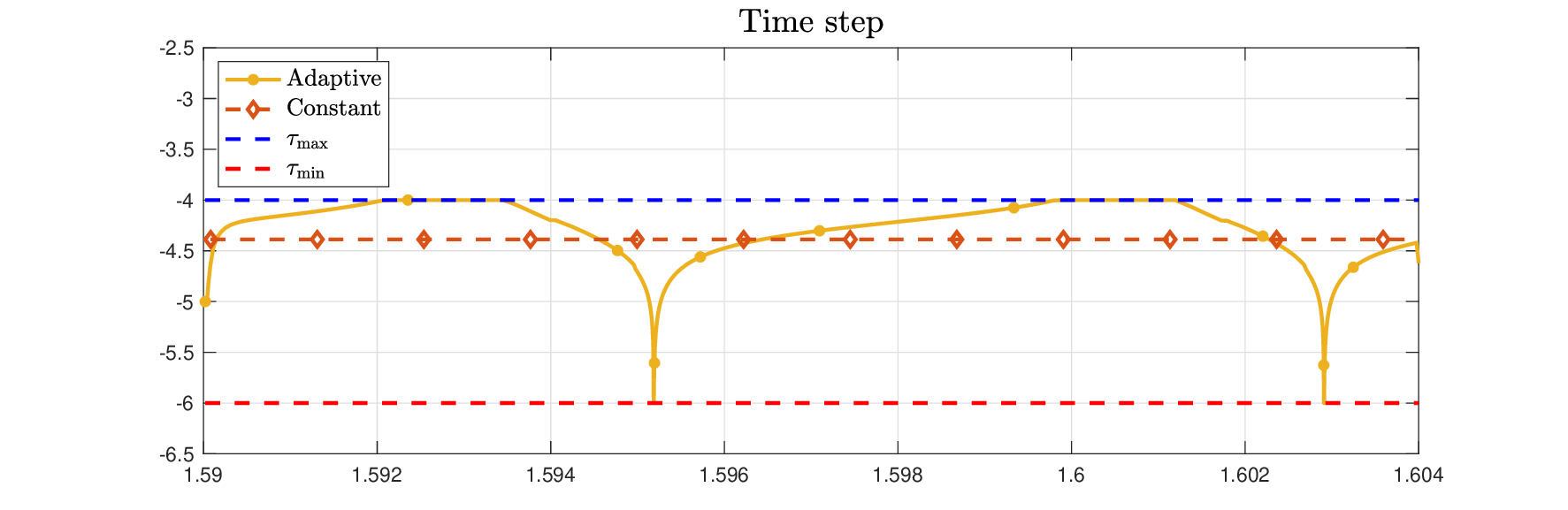}
    } \\ \vspace{0.01cm}
    \subfigure[Number of iteration at each time step]{\label{fig:IterNumPlot}
        \includegraphics[width=5.0in,height=1.3in]{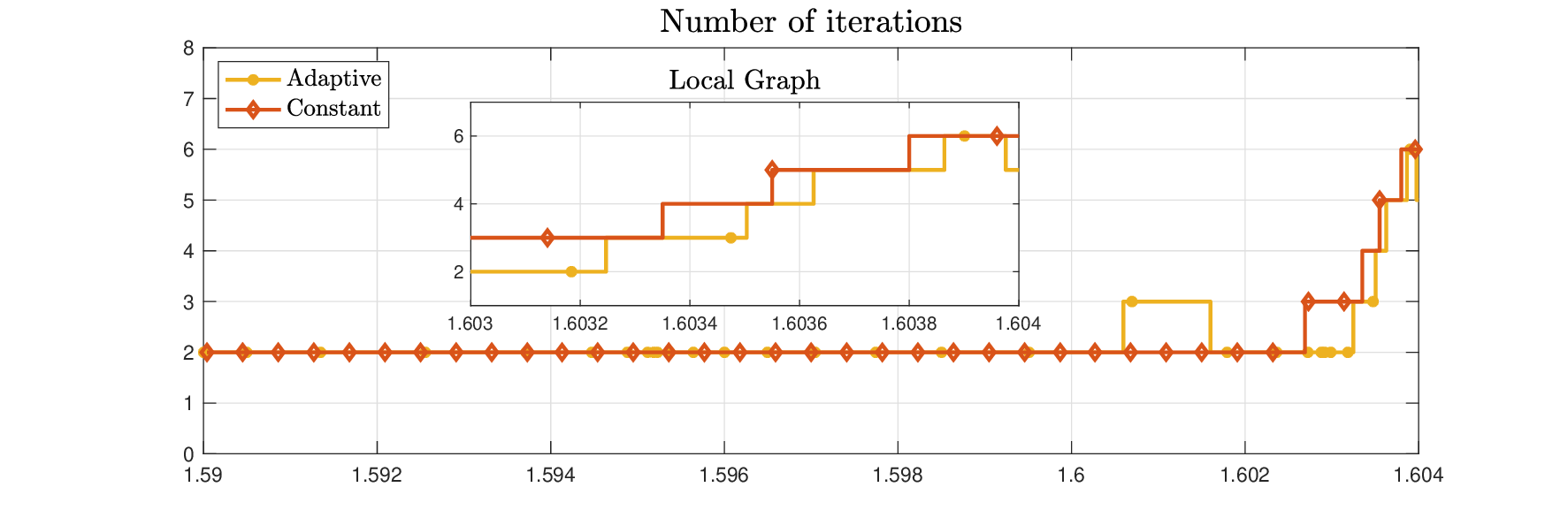}
    }
    \caption{
        The adaptive algorithm (Algorithm \eqref{alg:Adap-DLN}) 
        balances the efficiency and accuracy: the estimator of LTE $\widehat{T}_{n+1}$
        at most time steps 
        is close to, but below the required tolerance ${\tt{Tol = 1.e-4}}$, so that $\tau_{n}$ at most time steps 
        is above $\tau_{\rm{min}}$. 
        The adaptive algorithm (Algorithm \eqref{alg:Adap-DLN}) assigns more time steps to the highly stiff part ($t \geq 1.602$), which results in a smaller error in kinetic energy 
        at the end ($t = 1.604$).
        Both algorithms have relatively high computational costs on the extremely stiff time range ($t \in [1.603, 1.604]$) since the number of iterations at each time step increases from 2 to 6. 
    }
    \label{fig:Adapt_StepLTEIter}
\end{figure}

\section*{Acknowledgement}
We thank Professor Michael Neilan for many helpful discussions.
Catalin Trenchea is partially supported by the National Science Foundation under grant DMS-2208220.

\section{Conclusion} 
In this paper we developed an efficient fully-coupled implicit monolithic method for the MHD system in Els\"asser variables, in which the fully discrete solution is obtained by a partitioned iterative non-linear solver at each time step. 
We proved that the sequence of solutions by the partitioned iteration converges linearly to the solution of the method. 
We also showed that the proposed method has long-time conservation in the model energy, cross-helicity, magnetic helicity, and is second-order accurate under arbitrary time grids 
due to the time discretization of the midpoint rule.
A time-adaptive algorithm based on the local truncation error criterion was proposed, 
which balances time efficiency and accuracy. 
The first two numerical examples 
confirm the long-time conservation
of quadratic Hamiltonians, and the second-order accuracy of the proposed method, while the third 
stiff test problem demonstrates the advantage of time adaptivity.

\bibliographystyle{abbrv}
\bibliography{database2}

\end{document}